\documentclass[a4paper,10pt]{amsart}
\usepackage[top=0.8in, bottom=0.8in, left=1.25in, right=1.25in]{geometry}

\usepackage{fancybox}
\usepackage{amsmath}
\usepackage{amssymb}
\usepackage{amsthm}
\usepackage{mathrsfs}
\usepackage{epstopdf}
\usepackage{epstopdf}
\usepackage{array} 
\usepackage{amscd}
\usepackage{color}
\usepackage{hyperref}

\usepackage[all]{xypic}

\usepackage{amsfonts}
\usepackage{amssymb}
\usepackage{mathrsfs}

\DeclareFontFamily{U}{rsfs}{%
\skewchar\font127}
\DeclareFontShape{U}{rsfs}{m}{n}{%
<-6>rsfs5<6-8.5>rsfs7<8.5->rsfs10}{}
\DeclareSymbolFont{rsfs}{U}{rsfs}{m}{n}
\DeclareSymbolFontAlphabet
{\mathrsfs}{rsfs}
\DeclareRobustCommand*\rsfs{%
\@fontswitch\relax\mathrsfs}

\theoremstyle{plain}
\newtheorem{thm}{Theorem}[section]
\newtheorem{prop}[thm]{Proposition}
\newtheorem{lem}[thm]{Lemma}

\newtheorem{defi}[thm]{Definition}

\newtheorem{rmk}[thm]{Remark}
\newtheorem{cor}[thm]{Corollary}

\newtheorem{prop-defi}[thm]{Proposition-Definition}
\newtheorem{thm-defi}[thm]{Theorem-Definition}
\newtheorem{lem-defi}[thm]{Lemma-Definition}

\newtheorem{conj}[thm]{Conjecture}
\newtheorem{exam}[thm]{Example}

\newdimen\argwidth
\def\db[#1\db]{
 \setbox0=\hbox{$#1$}\argwidth=\wd0
 \setbox0=\hbox{$\left[\box0\right]$}
  \advance\argwidth by -\wd0
 \left[\kern.3\argwidth\box0 \kern.3\argwidth\right]}

\newcommand{\aA}{\mathcal{A}}
\newcommand{\bB}{\mathcal{B}}

\newcommand{\eE}{\mathcal{E}}
\newcommand{\fF}{\mathcal{F}}
\newcommand{\gG}{\mathcal{G}}
\newcommand{\hH}{\mathcal{H}}

\newcommand{\lL}{\mathcal{L}}
\newcommand{\mM}{\mathcal{M}}

\newcommand{\oO}{\mathcal{O}}
\newcommand{\pP}{\mathcal{P}}

\newcommand{\tT}{\mathcal{T}}

\newcommand{\vV}{\mathcal{V}}

\newcommand{\Hom}{\mathop{\rm Hom}\nolimits}

\newcommand{\dR}{\mathbf{R}}

\newcommand{\Pic}{\mathop{\rm Pic}\nolimits}

\newcommand{\id}{\textrm{id}}

\newcommand{\ch}{\mathop{\rm ch}\nolimits}

\newcommand{\Ext}{\mathop{\rm Ext}\nolimits}
\newcommand{\Spec}{\mathop{\rm Spec}\nolimits}
\newcommand{\rank}{\mathop{\rm rank}\nolimits}
\newcommand{\Coh}{\mathop{\rm Coh}\nolimits}

\newcommand{\cneq}{\mathrel{\raise.095ex\hbox{:}\mkern-4.2mu=}}
\newcommand{\eqcn}{\mathrel{=\mkern-4.5mu\raise.095ex\hbox{:}}}

\newcommand{\ext}{\mathop{\rm ext}\nolimits}

\newcommand{\Aut}{\mathop{\rm Aut}\nolimits}

\newcommand{\DT}{\mathop{\rm DT}\nolimits}
\newcommand{\PT}{\mathop{\rm PT}\nolimits}

\newcommand{\RHom}{\mathop{\dR\mathrm{Hom}}\nolimits}

\makeatletter
 
  \@addtoreset{equation}{section}
\makeatother

\title[{Curve counting via stable objects of CY 4-folds}]
{Curve counting via stable objects in derived \\ categories of Calabi-Yau 4-folds}
\date{}
\author{Yalong Cao}
\address{RIKEN Interdisciplinary Theoretical and Mathematical Sciences Program (iTHEMS), 2-1, Hirosawa, Wako-shi, Saitama, 351-0198, Japan}
\email{yalong.cao@riken.jp}
\author{Yukinobu Toda}
\address{Kavli Institute for the Physics and Mathematics of the Universe (WPI),The University of Tokyo Institutes for Advanced Study, The University of Tokyo, Kashiwa, Chiba 277-8583, Japan}
\email{yukinobu.toda@ipmu.jp}

\begin{document}
\maketitle
\begin{abstract}
In our previous paper with Maulik, we proposed a 
conjectural Gopakumar-Vafa (GV) type formula for the generating series of 
stable pair invariants on Calabi-Yau (CY) 4-folds. 
The purpose of this paper is to give an interpretation of 
the above GV type formula in terms of 
wall-crossing phenomena in the derived category. 
We introduce invariants counting 
LePotier's stable pairs on CY 4-folds, and show 
that they count certain stable objects in 
D0-D2-D8 bound states in the derived category. 
We propose a conjectural wall-crossing formula for 
the generating series of our invariants, which 
recovers the conjectural GV type formula. 
Examples are computed for both compact and toric cases to 
support our conjecture.

\end{abstract}



\tableofcontents

\section{Introduction}
\subsection{Background on GV/PT formula on CY 3-folds}
The notion of stable pairs was introduced by Pandharipande-Thomas 
(PT)~\cite{PT}
in order to give a better formulation of Maulik-Nekrasov-Okounkov-Pandharipande (MNOP) conjecture~\cite{MNOP}
relating Gromov-Witten (GW) invariants and Donaldson-Thomas (DT) curve counting invariants on Calabi-Yau (CY) 3-folds.  
By definition, a stable pair on a variety $X$ consists of a pair
\begin{align}\label{intro:PTpair}
(F, s), \ s \colon \oO_X \to F
\end{align}
satisfying PT stability condition:
$F$ is a pure one dimensional sheaf and $s$ is surjective in dimension one. 
 When $X$ is a CY 3-fold, 
we have integer valued invariants $P_{n, \beta} \in \mathbb{Z}$
(called PT invariants)
which virtually count stable pairs (\ref{intro:PTpair})
satisfying $([F], \chi(F))=(\beta, n)$. 
The generating series
\begin{align}\label{intro:PTX}
\PT(X)=\sum_{n, \beta}P_{n, \beta}y^n q^{\beta}
\end{align}
is conjectured to be equal to the 
generating series of GW invariants under some variable change, 
which was proved by Pandharipande-Pixton~\cite{PP} in many cases 
including quintic 3-folds. 

On the other hand, the generating series (\ref{intro:PTX})
on a CY 3-fold 
is expected to 
be written as an infinite product
(see e.g.~\cite[Conj.~6.2]{Toda5}) 
with powers given by Gopakumar-Vafa (GV) invariants
$n_{g, \beta} \in \mathbb{Z}$~\cite{GV, MT}:
\begin{align}\label{intro:PT:GV}
\PT(X)=\prod_{\beta}
\left(\prod_{j=1}^{\infty}(1-(-y)^j q^{\beta})^{jn_{0, \beta}}
\cdot \prod_{g=1}^{\infty} \prod_{k=0}^{2g-2}
(1-(-y)^{g-1-k}q^{\beta})^{(-1)^{k+g}n_{g, \beta}
\left(\begin{subarray}{c}
2g-2 \\
k
\end{subarray} \right)}   \right). 
\end{align}
In fact, such an infinite product can be 
explained from wall-crossing phenomena.
In the second author's previous works~\cite{Toda1, Toda2, Toda3, Toda5},
he investigated wall-crossing phenomena of 
stable D0-D2-D6 bound states in the derived category of 
coherent sheaves, 
by introducing one parameter family of weak stability conditions on 
them. 
The wall-crossing formula of associated 
DT counting invariants can be 
studied using the works of Joyce-Song~\cite{JS}
and Kontsevich-Soibelman~\cite{KS}. 
As a result, it turned out that 
the factor $\prod_{j=1}^{\infty}(1-(-y)^j q^{\beta})^{jn_{0, \beta}}$
is the wall-crossing term
(up to showing multiple cover conjecture of 
Joyce-Song's generalized DT invariants~\cite{JS} for 
one dimensional semistable sheaves), so giving 
an intrinsic understanding of the GV formula (\ref{intro:PT:GV}) via wall-crossing. 

\subsection{Motivation on GV/PT formula on CY 4-folds}
The purpose of this paper is to 
 give a similar interpretation of GV formula for 
stable pair invariants on CY 4-folds, using 
$\mathrm{DT}_4$ virtual classes defined in general cases by Borisov-Joyce~\cite{BJ} and in special cases by Cao-Leung~\cite{CL1}. 
In our previous paper with Maulik~\cite{CMT2}, 
we studied stable pair invariants on CY 4-folds $X$
and conjectured that the generating series of 
these invariants 
with exponential insertions $\exp(\gamma)$ 
for $\gamma \in H^4(X, \mathbb{Z})$
is written as a similar infinite product
\begin{align}\label{intro:PTexp}
\mathrm{PT}(X)(\exp(\gamma)) =\prod_{\beta} \left(
\exp(yq^{\beta})^{n_{0, \beta}(\gamma)}
\cdot M(q^{\beta})^{n_{1, \beta}}\right).
\end{align}
Here 
$M(q)=\prod_{k\geqslant 1}(1-q^k)^{-k}$ is the MacMahon function, and 
the invariants
\begin{align*}
n_{0, \beta}(\gamma) \in \mathbb{Q}, \ n_{1, \beta} \in \mathbb{Q}
\end{align*}
are GV type invariants on CY 4-folds
defined by Klemm-Pandharipande~\cite{KP}
from GW invariants on CY 4-folds, which are 
conjectured to be integers.  

We will consider a family of 
LePotier stability conditions on 
pairs (\ref{intro:PTpair}), parametrized by 
a stability parameter $t \in \mathbb{R}$, 
and construct $\mathrm{DT}_4$ type invariants counting such 
pairs. Indeed our invariants 
count certain stable objects 
in the category of D0-D2-D8 bound states, which 
is an abelian subcategory in the derived category of 
coherent sheaves. 
Here we observe a new phenomenon for CY 4-folds: 
stable objects on D0-D2-D8 bound states on CY 4-folds 
are always written as a pair (\ref{intro:PTpair}) while 
this is not the case for stable D0-D2-D6 bound states on 
CY 3-folds. 
We then propose a conjectural wall-crossing formula of 
generating series of our invariants, and 
explain that it recovers the GV formula (\ref{intro:PTexp}). 

\subsection{$\mathrm{DT}_4$ type invariants counting LePotier stable pairs}
Let $(X, \omega)$ be a polarized 
smooth projective CY 4-fold over $\mathbb{C}$. 
For $t \in \mathbb{R}$, a pair (\ref{intro:PTpair})
for a pure one dimensional sheaf $F$
is called
\textit{$Z_t$-stable} if the following conditions holds
(here we denote $\mu(F)=\chi(F)/(\omega \cdot [F]$)):
\begin{enumerate}
\renewcommand{\labelenumi}{(\roman{enumi})}
\item for any subsheaf $0\neq F' \subseteq F$, we have 
$\mu(F')<t$,
\item for any
subsheaf $ F' \subsetneq F$ 
such that $s$ factors through $F'$, 
we have 
$\mu(F/F')>t$. 
\end{enumerate}
Indeed we will see that the above 
stability condition is a special case of LePotier's 
stability conditions~\cite{LePotier} (ref.~Proposition \ref{identify stab}). 
For a given $\beta \in H_2(X, \mathbb{Z})$ and $n\in \mathbb{Z}$, 
we denote by
\begin{align*}
P_n^t(X, \beta)
\end{align*}
the moduli space of 
$Z_t$-stable pairs 
$(F, s)$ with $([F], \chi)=(\beta, n)$. It has a wall-chamber structure and for a generic $t \in \mathbb{R}$ (generic means outside a finite subset of $\mathbb{R}$), the above 
moduli space is a projective scheme, and 
coincides with the moduli space of PT stable 
pairs for the $t\to \infty$ limit. 

An important question is whether we can define $\DT_4$ type counting invariants of $P_n^t(X, \beta)$. When $t\to \infty$, this was done 
in \cite{CMT2} by using the well-known fact that moduli spaces of PT stable pairs can be regarded as moduli spaces of objects in derived categories of coherent sheaves \cite{PT} (note that the natural pair deformation-obstruction theory of PT moduli spaces $P_n(X, \beta)$ does not seem to give rise to a virtual class
even when $X$ is 3-dimensional). Therefore we are allowed to apply Pantev-T\"{o}en-Vaqui\'{e}-Vezzosi's construction of $(-2)$-shifted symplectic structures \cite{PTVV} and Borisov-Joyce's 
virtual classes \cite{BJ}.

Our first result is that 
the moduli space $P_n^t(X, \beta)$ is indeed a moduli space of two 
term complexes in the derived category for any choice of $t\in \mathbb{R}$: 
\begin{thm}\emph{(Proposition \ref{openess}, Theorem~\ref{prop:stack:AB})}\label{intro:thm1}
Let $\mM_0$ be the moduli stack of 
$E \in \mathrm{D^{b}(Coh(\textit{X\,}))}$
satisfying $\Ext^{<0}(E, E)=0$ and $\det E \cong \oO_X$. 
Then the natural 
morphism
\begin{align*}
P_n^t(X, \beta) \to \mM_0, \quad 
(F, s) \mapsto (\oO_X \stackrel{s}{\to} F)
\end{align*}
is an open immersion.
\end{thm}
As we mentioned in the previous subsection, this 
is a new phenomenon for CY 4-folds, 
as the similar statement is not true 
for CY 3-folds except the $t \to \infty$ limit, i.e. 
PT stable pairs. Indeed the moduli space 
$P_n^t(X, \beta)$ is regarded as a moduli space of 
certain stable objects in the extension 
closure 
\begin{align*}\mathcal{A}_X:=\langle \oO_X, \Coh_{\leqslant1}(X)[-1] \rangle_{\ext}
\subset  \mathrm{D^{b}(Coh(\textit{X\,}))},
\end{align*}
called the category of \textit{D0-D2-D8 bound states}. 
We will show that the above category is equivalent
to the category of pairs
$(\vV \to F)$, where $\vV$ is an iterated extensions of $\oO_X$ (ref.~Proposition \ref{A_X=B_X}). 
The above mentioned equivalence is not true for CY 3-folds, 
as we need the vanishing $\Ext^{2}(F, \oO_X)=0$ for 
any one dimensional sheaf $F$. 

Thanks to Theorem~\ref{intro:thm1}, we are able to define a virtual class (ref.~Theorem \ref{vir class of pair moduli})
\begin{align*}
[P^t_n(X, \beta)]^{\rm{vir}} \in H_{2n}(P_n(X, \beta), \mathbb{Z}),
\end{align*}
depending on the choice of orientation on
certain real line bundle on it \cite{CGJ}.
In order to define invariants from the above virtual class, 
we need to involve insertions: define the map 
\begin{align*}
\tau \colon H^{4}(X,\mathbb{Z})\to H^{2}(P^t_n(X,\beta),\mathbb{Z}), \quad 
\tau(\gamma) \cneq (\pi_{P})_{\ast}(\pi_X^{\ast}\gamma \cup\ch_3(\mathbb{F}) ),
\end{align*}
where $\pi_X$, $\pi_P$ are projections from $X \times P^t_n(X,\beta)$
onto corresponding factors, $\mathbb{I}=(\pi_X^*\oO_X\to \mathbb{F})$ is the universal pair and $\ch_3(\mathbb{F})$ is the
Poincar\'e dual to the fundamental cycle of $\mathbb{F}$.

For a generic $t \in \mathbb{R}$, 
we define the $Z_t$-stable pair invariant by 
\begin{align*}P^t_{n,\beta}(\gamma):=\int_{[P^t_n(X,\beta)]^{\rm{vir}}} \tau(\gamma)^n\in \mathbb{Z}. \end{align*}
Here we also write $P^t_{0,\beta}:=P^t_{n,\beta}(\gamma)$ when $n=0$.

When $t\to \infty$, $Z_t$-stable pairs are PT stable pairs. So we denote
\begin{align}
\label{intro:PT:def}
P_{n,\beta}(\gamma) \cneq P^t_{n,\beta}(\gamma)\big|_{t\to \infty}\,, \end{align}
which
is nothing but the stable pair invariant on CY 4-fold $X$ studied in \cite{CMT2}.

\subsection{Conjectures} 
The main conjecture of this paper is the following: 
\begin{conj}\label{main conj intro}\emph{(Conjecture \ref{main conj})}
Let $(X,\omega)$ be a smooth projective Calabi-Yau 4-fold, $\beta\in H_2(X,\mathbb{Z})$ and $n\in\mathbb{Z}_{\geqslant0}$.
Choose a generic $t\in \mathbb{R}_{>0}$. 
Then for certain choice of orientation, we have 
\begin{align}\label{intro:WCF}
P^t_{n,\beta}(\gamma)=\sum_{\begin{subarray}{c}\beta_0+\beta_1+\cdots+\beta_n=\beta  \\  \omega\cdot\beta_i>\frac{1}{t},\,  i=1,\ldots,n \end{subarray}}P_{0,\beta_0}\cdot\prod_{i=1}^nn_{0,\beta_i}(\gamma). 
\end{align}
In particular, $P^t_{0,\beta}=P_{0,\beta}$ is independent of the choice of $t>0$.
\end{conj}
As in the previous paper~\cite{CMT2}, the conjecture is based on a
heuristic argument given in Section~\ref{heuristic argument},
where we verify it assuming the CY 4-fold $X$ to be `ideal', i.e. curves in $X$ deform in some family of expected dimensions and have 
expected generic properties.

The conjectural formula (\ref{intro:WCF})
can be expressed in terms of generating series as follows. 
Set 
\begin{align*}
\PT^t(X)(\exp(\gamma))
\cneq \sum_{n, \beta}\frac{P_{n, \beta}^{t}(\gamma)}{n!}y^n q^{\beta}. 
\end{align*}
Then for each $t_0 \in \mathbb{R}_{>0}$, the formula (\ref{intro:WCF}) implies
the wall-crossing formula
\begin{align}\label{intro:WCF2}
\lim_{t \to t_0+}\PT^t(X)(\exp(\gamma))=\prod_{\omega \cdot \beta=\frac{1}{t_0}} \exp(yq^{\beta})^{n_{0, \beta}(\gamma)} \cdot \lim_{t \to t_0-}\PT^t(X)(\exp(\gamma)). 
\end{align}
In the $t \to \infty$ limit,
\begin{align*}
\mathrm{PT}(X)(\exp(\gamma)) =
\lim_{t \to \infty} \mathrm{PT}^{t}(X)(\exp(\gamma))
\end{align*}
is the generating series of PT stable pair invariants on $X$
by (\ref{intro:PT:def}). 

In the $t \to +0$ limit, by (\ref{intro:WCF}), we have  
\begin{align*}
\lim_{t\to +0} \PT^t(X)(\exp(\gamma))=\sum_{\beta}P_{0, \beta}q^{\beta}
=\prod_{\beta} M(q^{\beta})^{n_{1, \beta}},
\end{align*}
where the second identity is conjectured in~\cite[Conj.~1.2]{CMT2}. 
Therefore the wall-crossing formula from 
$t\to +0$ to $t\to \infty$ recovers
the conjectural GV formula (\ref{intro:PTexp}) of stable pair invariants
on CY 4-folds, giving an interpretation of (\ref{intro:PTexp}) in terms of
wall-crossing in the derived category. 

A particularly interesting choice of $t \in \mathbb{R}$ is 
$t=n/(\omega \cdot \beta)+0$, which sits in the first non-trivial chamber for a fixed $\beta$ and $n$.
In this case, 
the moduli space
\begin{align*}
P_n^t(X, \beta),  \quad
t=\frac{n}{\omega \cdot \beta}+0
\end{align*}
is the moduli space of \textit{``Joyce-Song type" 
stable pairs}, i.e. one dimensional semistable sheaves with sections satisfying certain property (ref.~Definition \ref{defi:PTJSpair}, compared with \cite[Def.~12.2]{JS} which deals with pairs from a very negative line bundle instead of $\oO_X$).
We define JS stable pair invariant by 
\begin{align*}
P_{n, \beta}^{\rm{JS}}(\gamma) \cneq \lim_{t \to n/(\omega \cdot \beta)+0}
P_{n, \beta}^{t}(\gamma). 
\end{align*}
Then Conjecture~\ref{main conj intro} in particular implies the following:

\begin{conj}\label{JS conj intro}\emph{(Conjecture \ref{main conj})}
In the same situation of Conjecture~\ref{main conj intro}, we have the 
identity
\begin{align}\label{intro:JS}
(1)\,\, P^{\mathrm{JS}}_{n,\beta}(\gamma)=\sum_{\begin{subarray}{c}\beta_1+\cdots+\beta_n=\beta  \\  \omega\cdot\beta_i= \frac{\omega\cdot \beta}{n},\,  i=1,\ldots,n \end{subarray}}\prod_{i=1}^nn_{0,\beta_i}(\gamma), \,\, \mathrm{if}\,\, n\geqslant 1, \quad  (2) \,\, P^{\mathrm{JS}}_{0,\beta}=P_{0,\beta}. 
\end{align}
In particular, $P^{\mathrm{JS}}_{1,\beta}(\gamma)=n_{0,\beta}(\gamma)$.
\end{conj}
Remarkably, JS stable pair invariants (based on above conjecture) also encode information of all genus GV type invariants.

\subsection{Verifications of conjectures}
In Section \ref{master spa argu}, we give a `heuristic' master space argument for Conjecture \ref{main conj intro} in the case of simple wall-crossing.
There we discuss the construction of master spaces and virtual classes locally and `heuristic' means we assume they extend to the global moduli spaces. 
We show that the wall-crossing formula given by the master space heuristics coincides with the formula in our main conjecture (Proposition \ref{lem:wcf:equiv}). 

Besides that, in Section~\ref{sec:JS/GV}, we 
check Conjecture~\ref{main conj intro} or Conjecture~\ref{JS conj intro}
in several examples: for some 
compact CY 4-folds (sextic 4-fold, elliptic fibered CY 4-fold), 
local Fano 3-folds, 
local surfaces ($\mathrm{Tot}_{\mathbb{P}^2}(\oO(-1) \oplus \oO(-2)), 
\mathrm{Tot}_{\mathbb{P}^1 \times \mathbb{P}^1}(\oO(-1, -1)^{\oplus 2}))$.
The most important check among them is the 
comparison between JS stable pair invariants
$P_{1, \beta}^{\rm{JS}}(\gamma)$ and 
$\mathrm{DT}_4$ invariants counting one dimensional 
stable sheaves. 
\begin{thm}\emph{(Theorem~\ref{prod of cy3},~\ref{thm:locFano}, \ref{n=1 local surface})}
\label{intro:thm:JS=N}
Suppose that $X$ is either 
$Y \times E$ where $Y$ is a projective CY 3-fold and $E$ is an elliptic 
curve, or $\mathrm{Tot}_Y(K_Y)$ for a Fano 3-fold $Y$.
Then for any curve class $\beta\in H_2(Y)\subseteq H_2(X)$, we have the identity
\begin{align}\label{intro:id:JS=N}
P_{1, \beta}^{\rm{JS}}(\gamma)=
\int_{[M_1(X, \beta)]^{\rm{vir}}}\tau(\gamma),
\end{align}
for certain choice of orientation.
Here $M_1(X, \beta)$ is the moduli space of one 
dimensional stable sheaves $F$ on $X$ with $([F], \chi(F))=(\beta,1)$. 
\end{thm}
The right hand side 
of (\ref{intro:id:JS=N}) is conjectured to 
be equal to $n_{0, \beta}(\gamma)$ in~\cite{CMT1} (see also Conjecture \ref{g=0 one dim sheaf conj}).
If this is the case, Theorem~\ref{intro:thm:JS=N}
implies Conjecture~\ref{JS conj intro}
in the case of $X=Y \times E$, $X=\mathrm{Tot}_Y(K_Y)$ with $\beta\in H_2(Y)\subseteq H_2(X)$ (see Corollary \ref{prim check}, \ref{loc fano check}). \\

Apart from them, we also study Conjecture~\ref{JS conj intro} for 
local $\mathbb{P}^1$, i.e.
\begin{align*}
X=\mathrm{Tot}_{\mathbb{P}^1}(\oO_{\mathbb{P}^1}(-1) \oplus 
\oO_{\mathbb{P}^1}(-1) \oplus \oO_{\mathbb{P}^1}).
\end{align*}
In this case, 
the four dimensional complex torus $(\mathbb{C}^{\ast})^4$
acts on $X$, and we denote by $T \subset (\mathbb{C}^{\ast})^4$
the subtorus preserving the CY 4-form. 
We will define the $T$-equivariant JS stable 
pair invariant by
(see Definition~\ref{def of JS inv for -1-10}):
\begin{align*}
P^{\mathrm{JS}}_{n,d}:=\sum_{I\in P^{\mathrm{JS}}_n(X,d\,[\mathbb{P}^1])^T}(-1)^{d+1}e_T(\chi_X(I,I)^{\frac{1}{2}}_0)\in \frac{\mathbb{Q}(\lambda_0, \lambda_1,\lambda_2,\lambda_3)}{(\lambda_0+\lambda_1+\lambda_2+\lambda_3)}.
\end{align*}
Here we make a particular choice 
of square root 
$\chi_X(I,I)^{\frac{1}{2}}_0$ as in Lemma \ref{choice of squ root} and the sign $(-1)^{d+1}$ denotes a choice of orientation to normalize
the expression.
We will give an explicit computation of the above invariant. 
\begin{thm}\emph{(Theorem~\ref{main thm local curve})}
\label{intro:thm:loccurve}
We have:
\begin{align*}
P^{\mathrm{JS}}_{n,d} =&\frac{(-1)^{k(d+1)}}{1!\,2!\,\cdots k!}\cdot\frac{1}{\lambda_0^{k(k+1)/2}\lambda_3^{d}}\cdot 
\sum_{\begin{subarray}{c}d_0+\cdots+d_k=d  \\  d_0,\ldots, d_k\geqslant 0 \end{subarray}}\frac{1}{d_0!\cdots d_k!}\cdot \prod_{\begin{subarray}{c}i<j  \\  0\leqslant i,j \leqslant k \end{subarray}}\Big((j-i)\lambda_0+(d_i-d_j)\lambda_3\Big) \\
&\times \prod_{i=0}^k\Bigg(\prod_{\begin{subarray}{c}1\leqslant a\leqslant d_i  \\  1\leqslant b\leqslant k-i  \end{subarray}}\frac{1}{a\lambda_3+b\lambda_0}\cdot
\prod_{\begin{subarray}{c}1\leqslant a\leqslant d_i  \\  1\leqslant b\leqslant i  \end{subarray}}\frac{1}{a\lambda_3-b\lambda_0}\Bigg), 
\quad \mathrm{if}\,\, n=d(k+1), \,\, k\geqslant 0,
\end{align*}
and $P^{\mathrm{JS}}_{n,d}=0$ otherwise.
\end{thm}
The formula in Theorem~\ref{intro:thm:loccurve}
is complicated, but we expect  
significant cancellations of 
rational functions. 
Indeed as an analogy of Conjecture~\ref{JS conj intro}, 
we should have the identities:
\begin{align}\label{intro:locP1}
P_{n, d}^{\rm{JS}}=
\left\{
\begin{array}{cc}
\frac{1}{d!(\lambda_3)^d}, & n=d, \\
& \\
0, & n\neq d. 
\end{array}
\right. 
\end{align}
By an residue argument and a `Mathematica' program, we show the following:
\begin{thm}\emph{(Theorem~\ref{thm:id:rational})}
The identity (\ref{intro:locP1}) holds in the following cases 
\begin{itemize}
\item $d\nmid n$, 
\item $d=1,2$ with any $n$,
\item  $n=d,2d$ with any $d$.
\end{itemize}
The identity \eqref{intro:locP1} is also checked 
in many other cases by Mathematica
 (see~Proposition~\ref{prop:mathematica}). 
 \end{thm}
 \begin{rmk}
 Recently \eqref{intro:locP1} has been proved in full generality in \cite{CT2}.
 \end{rmk}
Finally we remark that one issue of the current proposal (this also happened in previous related works, e.g.~\cite{CMT2}) is that we do not have a general way 
to fix the choice of orientation in the virtual classes and invariants. 
Our choice of orientation in verifications is based on case by case studies. 
Nevertheless, we expect our wall-crossing interpretation in this paper will shed new light on this issue, i.e.~we expect the choice of orientation 
on different moduli spaces should be compatible with wall-crossing. 
In fact, motivated by this, explicit choice of orientation for moduli spaces of PT stable pairs on $K_Y$ (where $Y$ is Fano 3-fold) is given in 
\cite[(1.5)]{CKM2} (at least when stable pairs are scheme theoretically supported on $Y$) and used to verify \eqref{intro:PTexp} in examples. 
We hope to explore this more in the future.


\subsection{Notation and convention}
In this paper, all varieties and schemes are defined over $\mathbb{C}$. 
For a morphism $\pi \colon X \to Y$ of schemes, 
and for $\fF, \gG \in \mathrm{D^{b}(Coh(\textit{X\,}))}$, we denote by 
$\dR \hH om_{\pi}(\fF, \gG)$ 
the functor $\dR \pi_{\ast} \dR \hH om_X(\fF, \gG)$. 
We also denote by $\mathrm{ext}^i(\fF, \gG)$ the dimension of 
$\Ext^i_X(\fF, \gG)$. 

A class $\beta\in H_2(X,\mathbb{Z})$ is called \textit{irreducible} (resp. \textit{primitive}) if it is not the sum of two non-zero effective classes
(resp. if it is not a positive integer multiple of an effective class).

\subsection{Acknowledgement}
Both authors are supported by the World Premier International Research Center Initiative (WPI), MEXT, Japan.
Y. C. is partially supported by RIKEN Interdisciplinary Theoretical and Mathematical Sciences
Program (iTHEMS), JSPS KAKENHI Grant Number JP19K23397 and Newton International Fellowships Alumni 2019 and 2020.
Y. T. is supported by Grant-in Aid for Scientific Research grant (No. 26287002) from MEXT, Japan.

\section{Definitions}
Throughout this paper, unless stated otherwise,
$(X,\omega)$ is always denoted to be a smooth projective Calabi-Yau 4-fold (i.e. $K_X\cong \oO_X$) with 
an ample divisor $\omega$ on it.

\subsection{Category of D0-D2-D8 bound states}
We define the category of 
D0-D2-D8 bound states
on $X$ to be the extension closure in
the derived category
\begin{align*}
\mathcal{A}_X:=\langle \oO_X, \Coh_{\leqslant1}(X)[-1] \rangle_{\ext}
\subset \mathrm{D^{b}(Coh(\textit{X\,}))}. 
\end{align*}
Here $\Coh_{\leqslant1}(X)$ is the category of
coherent sheaves $F$ on $X$ whose support have 
dimension less than or equal to one. 
The argument in~\cite[Lem.~3.5]{Toda2}
shows that
 $\mathcal{A}_X$ is the heart of a bounded t-structure on the triangulated 
subcategory of $\mathrm{D^{b}(Coh(\textit{X\,}))}$
generated by $\oO_X$ and $\Coh_{\leqslant1}(X)$. 
In particular, $\mathcal{A}_X$ is an abelian category. 

We also define the category $\bB_X$, whose 
objects consist of triples
\begin{align*}
(\vV, F, s), \quad
\vV \in \langle \oO_X \rangle_{\rm{ext}}, \ 
F \in \Coh_{\leqslant1}(X), \ s \colon \vV \to F. 
\end{align*}
Note that if $H^1(\oO_X)=0$, 
the vector bundle $\vV$ is of the form $V \otimes \oO_X$
for a finite dimensional vector space $V$.
The set of morphisms in $\bB_X$ is given
by commutative diagrams of coherent sheaves
\begin{align}\label{morphism:B}
\xymatrix{
\vV \ar[r]^-{s} \ar[d]_-{\alpha} & F \ar[d]^-{\beta} \\
\vV' \ar[r]^-{s'} & F'. 
}
\end{align} 
We compare the categories $\aA_X$ and $\bB_X$ in the 
following proposition:

\begin{prop}\label{A_X=B_X}
There exists a natural equivalence of categories 
\begin{align}\label{equiv:A=B}
\Phi \colon 
\bB_X \stackrel{\sim}{\to} \aA_X.  
\end{align}
\end{prop}
\begin{proof}
For an object $E=(\vV, F, s)$ in 
$\bB_X$, we have the associated two term complex 
$\Phi(E)=(\vV \stackrel{s}{\to} F) \in \mathrm{D^{b}(Coh(\textit{X\,}))}$,
where $\vV$ is located in degree zero. By the distinguished triangle
\begin{align*}
F[-1] \to \Phi(E) \to \vV,
\end{align*}
the object $\Phi(E)$ lies in $\aA_X$, 
hence we obtain the functor (\ref{equiv:A=B}). 
Indeed the above sequence is a short exact seqeunce 
in the abelian category $\aA_X$. 
Below we show that $\Phi$ is an equivalence 
along with the argument of~\cite[Prop.~2.2]{Toda4}. 

We first show that $\Phi$ is fully-faithful. 
Let us take another triple $E'=(\vV', F', s')$, and 
take a morphism 
$\gamma \colon 
\Phi(E) \to \Phi(E')$
in $\aA_X$. 
By the Serre duality, we have the vanishing 
$\Hom(F[-1], \vV')=\Ext^3(\vV', F)^{\vee}=0$, hence 
we have the unique 
morphisms $(\alpha, \beta)$ which make the following 
diagram commutative
\begin{align}\label{diagram:FV}
\xymatrix{
F[-1] \ar[r] \ar@{.>}[d]_-{\beta}& \Phi(E) \ar[r] \ar[d]_-{\gamma} &\vV \ar@{.>}[d]_-{\alpha} \\
F'[-1] \ar[r] & \Phi(E') \ar[r] & \vV'.
}
\end{align}
By taking cones, we obtain the diagram (\ref{morphism:B}). 
Conversely given a diagram (\ref{morphism:B}), 
there is a morphism $\gamma$ which makes 
the diagram (\ref{diagram:FV}) commutative. 
Because of $\Hom(\vV, F'[-1])=0$, 
such $\gamma$ is uniquely determined. 
Therefore the functor $\Phi$ is fully-faithful. 

It remains to show that 
the functor $\Phi$ is essentially surjective. 
For an object $M \in \aA_X$, by the definition of $\aA_X$,
there is a filtration 
\begin{align*}
M_0 \subset M_1 \subset \cdots \subset M_k=M,
\end{align*} 
such that each $N_i=M_i/M_{i-1}$ is isomorphic to $\oO_X$
or an object in $\Coh_{\leqslant1}(X)[-1]$. 
We show that, by the induction on $j$, 
 each $M_j$ is isomorphic to 
an object of the form $\Phi(E_j)$
for an object $E_j=(\vV_j \to F_j)$ in $\bB_X$. 
The case of $j=0$ is obvious. 
Suppose that $M_{j-1}$ is isomorphic to 
$\Phi(E_{j-1})$. 
If $N_j=\oO_X$, then by taking the cones of the 
commutative diagram
\begin{align*}
\xymatrix{
& \oO_X[-1] \ar[d] \ar[rd] & \\
F_{j-1}[-1] \ar[r] & M_{j-1} \ar[r] & \vV_{j-1},
}
\end{align*}
we obtain the exact sequences in $\aA_X$
\begin{align*}
0 \to F_{j-1}[-1] \to M_j \to \vV_j \to 0, \quad 
0 \to \vV_{j-1} \to \vV_j \to \oO_X \to 0. 
\end{align*}
Therefore $M_j$ is isomorphic to $\Phi(\vV_j \to F_{j-1})$. 
If $N_j=F[-1]$ for $F \in \Coh_{\leqslant1}(X)[-1]$, 
we have a commutative diagram
\begin{align*}
\xymatrix{
& F[-2] \ar[ld] \ar[d] & \\
F_{j-1}[-1] \ar[r] & M_{j-1} \ar[r] & \vV_{j-1},
}
\end{align*}
since $\Hom(F[-2], \vV_{j-1})=H^2(\vV_{j-1}^{\vee} \otimes F)^{\vee}=0$. 
By taking cones, we obtain exact sequences in $\aA_X$:
\begin{align*}
0 \to F_{j}[-1] \to M_j \to \vV_{j-1} \to 0, \quad 
0 \to F_{j-1}[-1] \to F_j[-1] \to F[-1] \to 0. 
\end{align*}
Therefore $M_j$ is isomorphic to 
$\Phi(\vV_{j-1} \to F_{j})$. 
\end{proof}

Below we will be interested in 
objects in $\aA_X$ with  
Chern character of the following form
\begin{align}\label{Chern:v}
v=(1, 0, 0, -\beta, -n) \in H^0(X) \oplus H^2(X) \oplus H^4(X)
\oplus H^6(X) \oplus H^8(X). 
\end{align}
We also identity $\beta$ with an element in $H_2(X)$
by Poincar\'e duality. 
Note that for an object $(\vV \to F)$ in $\bB_X$, 
we have
\begin{align*}
\ch \Phi(\vV \to F)=v \ \Leftrightarrow
\ \vV=\oO_X, \ ([F], \chi(F))=(\beta, n).
\end{align*}
Here $[F] \in H_2(X, \mathbb{Z})$ is the 
fundamental one cycle of $F$. 

We relate objects in $\aA_X$ with 
Chern character of the form (\ref{Chern:v}) to objects 
in a tilting 
of $\Coh(X)$ with respect to the slope
stability. 
For $E \in \Coh(X)$, with respect to the ample divisor $\omega$ on $X$, we set
\begin{align*}
\widehat{\mu}_{\omega}(E) =\frac{c_1(E) \cdot \omega^{3}}{\rank(E)} \in \mathbb{Q}
\cup \{\infty\}. 
\end{align*}
As usual, an object $E \in \Coh(X)$ is 
called $\widehat{\mu}_{\omega}$-semistable if 
for any non-zero subsheaf $E' \subset E$, 
we have $\widehat{\mu}_{\omega}(E')\leqslant \widehat{\mu}_{\omega}(E)$. 
We define subcategories of $\Coh(X)$:
\begin{align*}
&\tT_{\omega}=\langle \widehat{\mu}_{\omega}\mbox{-semistable } E
\mbox{ with } \widehat{\mu}_{\omega}(E)>0\rangle_{\rm{ext}}, \\
&\fF_{\omega}=\langle \widehat{\mu}_{\omega}\mbox{-semistable } E
\mbox{ with } \widehat{\mu}_{\omega}(E) \leqslant 0\rangle_{\rm{ext}}. 
\end{align*}
By the existence of Harder-Narasimhan filtrations, the pair 
of subcategories $(\tT_{\omega}, \fF_{\omega})$ forms a 
torsion pair of $\Coh(X)$. 
By taking the tilting \cite{HRS}, we obtain the heart of a 
bounded t-structure
\begin{align*}
\widehat{\aA}_X =\langle \fF_{\omega}, \tT_{\omega}[-1]\rangle_{\rm{ext}}
\subset \mathrm{D^{b}(Coh(\textit{X\,}))}. 
\end{align*}
Note that we have $\aA_X \subset \widehat{\aA}_X$ by their 
definitions. 

\begin{lem}\label{lem:A=Ahat}
Let $v \in H^{\ast}(X)$ be of the form (\ref{Chern:v}). 
For an object $E \in \mathrm{D^{b}(Coh(\textit{X\,}))}$
with $\ch(E)=v$ and $\det(E)\cong\oO_X$, it is an object 
in $\aA_X$ if and only if it is an object in $\widehat{\aA}_X$. 
\end{lem}
\begin{proof}
Since $\aA_X \subset \widehat{\aA}_X$, it is enough to show 
that an object $E \in \widehat{\aA}_X$ with $\ch(E)=v$ and $\det(E)\cong\oO_X$ 
is an object in $\aA_X$. 
We have an exact sequence in $\widehat{\aA}_X$
\begin{align}\label{exact:A}
0 \to \hH^0(E) \to E \to \hH^1(E)[-1] \to 0,
\end{align}
such that $\hH^0(E) \in \fF_{\omega}$ and $\hH^1(E) \in \tT_{\omega}$. 
Since 
\begin{align*}
\ch_1(\hH^0(E)) \cdot \omega^3 \leqslant 0, \quad
\ch_1(\hH^1(E)[-1]) \cdot \omega^3 \leqslant 0,
\end{align*}
and their sum is zero, we have 
$\ch_1(\hH^i(E)) \cdot \omega^3=0$
for $i=0, 1$. 
It follows that $\hH^0(E)$ is a $\mu_{\omega}$-semistable sheaf
and $\hH^1(E) \in \Coh_{\leqslant2}(X)$.
Moreover we have 
\begin{align*}
\ch_2(\hH^0(E)) \cdot \omega^2 \leqslant 0, \quad 
\ch_2(\hH^1(E)[-1]) \cdot \omega^2 \leqslant 0,
\end{align*}
where the first inequality follows from 
the Bogomolov-Gieseker inequality. 
As their sum is also zero, 
we have 
$\ch_2(\hH^i(E)) \cdot \omega^2=0$
for $i=0, 1$. 
Therefore $\hH^1(E)[-1] \in \Coh_{\leqslant1}(X)[-1]$. 
As for $\hH^0(E)$, since 
it is of rank one and has trivial determinant, 
we have the exact seqeunce of coherent sheaves
\begin{align*}
0 \to \hH^0(E) \to \hH^0(E)^{\vee \vee} \cong \oO_X \to T \to 0,
\end{align*}
for some $T \in \Coh_{\leqslant2}(X)$.
The vanishing of $\ch_2(\hH^0(E)) \cdot \omega^2$
implies that $T \in \Coh_{\leqslant1}(X)$, 
so we have 
$\hH^0(E) \in \aA_X$. 
From the exact sequence (\ref{exact:A}), 
we conclude $E \in \aA_X$. 
\end{proof}

\subsection{Moduli stacks of objects on $\aA_X$}
Let $\mM$ be the 2-functor
\begin{align*}
\mM \colon \mathrm{Sch}/\mathbb{C} \to 
\mathrm{Groupoid},
\end{align*}
which sends a $\mathbb{C}$-scheme $T$ 
to the groupoid of perfect complexes 
$\eE$
on $X \times T$
such that the derived restriction 
$\eE_t=\eE|_{X \times \{t\}}$ for each 
closed point $t \in T$
is an object in 
$\mathrm{D^{b}(Coh(\textit{X\,}))}$
satisfying $\Ext^{<0}(\eE_t, \eE_t)=0$. 
By a result of Lieblich~\cite{Lieblich}, 
the 2-functor $\mM$ is an Artin stack 
locally of finite type. 

By taking the determinant
of $\eE$, we have a morphism of stacks
\begin{align*}
\det \colon 
\mM \to [\Pic(X)/\mathbb{C}^{\ast}]. 
\end{align*}
We define the Artin stack $\mM_0$ by the following 
Cartesian square
\begin{align*}
\xymatrix{
\mM_0 \ar[r] \ar[d] & \mM \ar[d]^-{\det} \\
\Spec \mathbb{C} \ar[r] & [\Pic(X)/\mathbb{C}^{\ast}].
}
\end{align*}
Here the bottom arrow corresponds to the 
trivial line bundle $\oO_X$. 
We have the decomposition into 
open and closed substacks
\begin{align*}
\mM_0=\coprod_{v \in H^{\ast}(X)} \mM_0(v),
\end{align*}
where $\mM_0(v)$ parametrizes objects in $\mathrm{D^{b}(Coh(\textit{X\,}))}$
 with 
Chern character $v$ and trivial determinant.  
We have substacks
\begin{align*}
\underline{\aA}_X(v) \subset \underline{\widehat{\aA}}_X(v) \subset \mM_0(v),
\end{align*}
where $\underline{\aA}_X(v)$ (resp. $\underline{\widehat{\aA}}_X(v)$) parametrizes 
objects in $\aA_X$ (resp. $\widehat{\aA}_X$),  
with Chern characters $v$ and trivial determinant. 
We have the following proposition. 

\begin{prop}\label{openess}
Suppose that $v$ is of the form (\ref{Chern:v}). 
Then we have 
$\underline{\aA}_X(v)=\underline{\widehat{\aA}}_{X}(v)$, 
and they are open substacks of $\mM_0(v)$. 
\end{prop}
\begin{proof}
The identity $\underline{\aA}_X(v)=\underline{\widehat{\aA}}_{X}(v)$
follows from Lemma~\ref{lem:A=Ahat}. 
It remains to show that $\underline{\widehat{\aA}}_{X}(v)$
is an open substack of $\mM_0(v)$. 
This can be proved literally following the proof of~\cite[Lem.~4.7]{Tst3}, 
where the similar statement is proved for K3 surfaces. 
Altenatively, 
the stacks of objects in 
$(\tT_{\omega}, \fF_{\omega})$ 
detemine the open stack of torsion theories 
in the sense of~\cite[App.~A, Definition]{AB}
on the moduli stack of objects in $\Coh(X)$.
Therefore the openness of $\underline{\widehat{\aA}}_X(v)$
follows from~\cite[Thm.~A.3]{AB}. 
\end{proof}

Let us take $v$ of the form (\ref{Chern:v}). 
We define the moduli stack of objects in $\bB_{X}$
with Chern character $v$ 
to be the 2-functor
\begin{align*}
\underline{\bB}_X(v) \colon 
\mathrm{Sch}/\mathbb{C} \to \mathrm{Groupoid},
\end{align*}
which sends a $\mathbb{C}$-scheme $T$ to the 
groupoid of pairs $(\oO_{X \times T} \stackrel{s}{\to} \fF)$, 
where $\fF$ is a flat family of objects in 
$\Coh_{\leqslant1}(X)$
such that $(\oO_X \to \fF_t)$ has Chern character $v$
for any closed point $t \in T$. 
The isomorphisms in $\underline{\bB}_X(v)$ 
are given by commutative diagrams
\begin{align*}
\xymatrix{
\oO_{X\times T} \ar[r]^-{s} \ar@{=}[d] & \fF \ar[d]^-{\cong} \\
\oO_{X \times T} \ar[r]^-{s'} & \fF'. 
}
\end{align*}
We have a morphism of stacks
\begin{align}\label{stack:AB}
\underline{\Phi} \colon 
\underline{\bB}_X(v) \to \underline{\aA}_X(v),
\end{align}
by sending pairs 
$(\oO_{X \times T} \stackrel{s}{\to} \fF)$
to the associated two term complexes. 
\begin{thm}\label{prop:stack:AB}
The morphism of stacks (\ref{stack:AB})
is an isomorphism of stacks. 
\end{thm}
\begin{proof}
By Proposition~\ref{A_X=B_X}, 
the morphism (\ref{stack:AB}) induces an 
equivalence of groupoid of $\mathbb{C}$-valued points of 
$\underline{\bB}_X(v)$ and $\underline{\aA}_X(v)$. 
It is enough to show that 
the infinitesimal deformation theories of 
$\mathbb{C}$-valued points in $\underline{\bB}_X(v)$ and 
$\underline{\aA}_X(v)$ are equivalent. 
Namely, let $R_0$ be an Artinian local 
$\mathbb{C}$-algebra
and $0 \to I \to R \to R_0 \to 0$ 
a square zero extension. 
Take a $R_0$-valued point 
\begin{align}\label{R0:point}
(\oO_{X \times \Spec R_0} \to \fF_0)
\in \underline{\bB}_{X}(v)(\Spec R_0). 
\end{align} 
Suppose that 
the associated two term complex
\begin{align*}
\eE_0=\Phi(\oO_{X \times \Spec R_0} \to \fF_0)
\in \underline{\aA}_{X}(v)(\Spec R_0)
\end{align*}
extends to a $R$-valued point 
$\eE \in \underline{\aA}_{X}(v)(\Spec R)$. 
Then we show that 
there is an extension of (\ref{R0:point})
to a $R$-valued point of $\underline{\bB}_X(v)$, 
unique up to isomorphisms, 
and corresponds to $\eE$ under $\Phi$. 

Let $\mathbf{m} \subset R_0$ be the maximal ideal, 
$F=\fF_0 \otimes_{R_0}R_0/\mathfrak{m}$ and 
take $E=\eE_0 \otimes_{R_0}R_0/\mathfrak{m}=\Phi(\oO_X \to F)$. 
From the distinguished triangle 
$F[-1] \to E \to \oO_X$, we have the following commutative diagram
\begin{align*}
\xymatrix{
& \RHom(\oO_X, \oO_X)[1] \ar@{=}[r] \ar[d] & \RHom(\oO_X, \oO_X)[1] \ar[d] \\
\RHom(E, F) \ar[r] &\RHom(E, E)[1] \ar[r] \ar[d] & \RHom(E, \oO_X)[1] \ar[d] \\
&  \RHom(E, E)_0[1] \ar[r] & \RHom(F, \oO_X)[2]. 
}
\end{align*}
Here $(-)_0$ means taking the traceless part.
From the above diagram, 
we obtain a 
distinguished triangle
\begin{align*}
\RHom(E, F) \to \RHom(E, E)_0[1] \to \RHom(F, \oO_X)[2]. 
\end{align*}
Since $X$ is a Calabi-Yau 4-fold, we have the following vanishing by Serre duality:
\begin{align*}
\Hom(F, \oO_X[1])=H^3(X, F)^{\vee}=0, \quad
\Hom(F, \oO_X[2])=H^2(X, F)^{\vee}=0. 
\end{align*}
Therefore we have an isomorphism and an injection:
\begin{align}\label{def:obs}
\Hom(E, F) \stackrel{\cong}{\to} \Ext^1(E, E)_0, \quad
\Ext^1(E, F) \hookrightarrow \Ext^2(E, E)_0. 
\end{align}
From the deformation-obstruction theory of pairs, 
the obstruction class of 
extending (\ref{R0:point}) to a $R$-valued point of 
$\underline{\bB}_X(v)$ lies in $\Ext^1(E, F) \otimes I$. 
Its image under the map 
$\Ext^1(E, F)  \otimes I\to \Ext^2(E, E)_0 \otimes I$ is 
the obstruction class extending $\eE_0$ to a $R$-valued point 
of $\underline{\aA}_X(v)$, which vanishes as we assumed $\eE_0$ extends to $\eE$. 
Thus by the injectivity of the right map in (\ref{def:obs}), 
the obstruction class of extending (\ref{R0:point}) vanishes, 
hence it extends to a $R$-valued point of $\underline{\bB}_X(v)$. 
Moreover all possible extensions of (\ref{R0:point})
form $\Hom(E, F) \otimes I$-torsor, and those of $\eE_0$ form $\Ext^1(E, E)_0 \otimes I$-torsor. 
Therefore the uniqueness also holds by 
the left isomorphism in (\ref{def:obs}). 
\end{proof}

\subsection{Moduli spaces of $Z_t$-stable pairs}
We recall Le Potier's stability (such pairs were studied in low dimensions by S. Bradlow,  M. Thaddeus and A. Bertram,~etc.)~for pairs on $X$ (ref. \cite{LePotier},  \cite[Sect.~1.1]{PT}, \cite[pp.~164]{JS}). 
Let us fix an ample line bundle $\oO_X(1)$ on $X$ with 
corresponding divisor $\omega$. 
For a coherent sheaf $F$ on $X$, 
its Hilbert polynomial is defined by 
\begin{align*}
\chi(F(m))=\sum_{i \in \mathbb{Z}} (-1)^i \dim H^i(F(m))
\in \mathbb{Q}[m]. 
\end{align*}
We denote by $r(F)$ the leading coefficient of the above polynomial. 
\begin{defi}
A pair $(s \colon \oO_X \to F)$ in $\bB_X$ is
called $q$-(semi)stable 
for a polynomial $q \in \mathbb{Q}[m]$
if the following conditions hold: 
\begin{enumerate}
\renewcommand{\labelenumi}{(\roman{enumi})}
\item For any subsheaf $F' \subset F$, 
we have the inequality
\begin{align}\label{LP1}
\frac{\chi(F'(m))}{r(F')}<(\leqslant) \frac{\chi(F(m))+q(m)}{r(F)}, \quad m \gg 0.
\end{align}
\item For any subsheaf $F' \subsetneq F$
such that $s$ factors through $F'$, 
we have the inequality
\begin{align}\label{LP2}
\frac{\chi(F'(m))+q(m)}{r(F')}<(\leqslant) \frac{\chi(F(m))+q(m)}{r(F)}, \quad m \gg 0.
\end{align}
\end{enumerate}
\end{defi}

On the other hand for each $t \in \mathbb{R}$, 
we define
the map
\begin{align*}
Z_t \colon K(\aA_X) \to \mathbb{C},
\end{align*}
by 
sending $E \in K(\aA_X)$ 
with $\ch(E)=(r, 0, 0, -\beta, -n)$
to 
\begin{align*}
Z_t(E):=\left\{  \begin{array}{cc}
r(-t+\sqrt{-1}), & r \neq 0, \\
& \\
-n+(\beta \cdot \omega) \sqrt{-1}, & r=0. 
\end{array} \right. 
\end{align*}
By the definition of $\aA_X$, 
we have $Z_t(E) \in \hH \cup \mathbb{R}_{<0}$
for any non-zero $E \in \aA_X$, 
where $\hH \subset \mathbb{C}$ is the upper half plane. 
The pair $(\aA_X, Z_t)$ is a weak stability condition 
introduced in~\cite{Toda3} (see also~\cite{Toda5}), which generalizes Bridgeland's notion of stability conditions \cite{Bri}. 
\begin{defi}\label{def of Z_t stability}
An object $E \in \aA_X$ is called $Z_t$-(semi)stable 
if for any exact sequence $0 \to E' \to E \to E'' \to 0$
in $\aA_X$
with $E'\neq 0$, $E''\neq 0$, we have 
\begin{align*}
\arg Z_t(E')<(\leqslant) \arg Z_t(E'') \in (0, \pi]. 
\end{align*}
\end{defi}
For an object $F \in \Coh_{\leqslant1}(X)$
with $([F], \chi(F))=(\beta, n)$, we set
\begin{align*}
\mu(F) =\frac{n}{\beta \cdot \omega} \in \mathbb{Q} \cup \{\infty\},
\end{align*}
where $\mu(F)=\infty$ if $\beta=0$. 
We have the following characterization of $Z_t$-(semi)stable objects. 
\begin{lem}\label{lem:charactrize}
For an object $(\oO_X \stackrel{s}{\to} F) \in \bB_X$, 
the object $\Phi(\oO_X \stackrel{s}{\to} F) \in \aA_X$
is $Z_t$-(semi)stable
if and only if the following conditions hold:
\begin{enumerate}
\renewcommand{\labelenumi}{(\roman{enumi})}
\item for any subsheaf $0\neq F' \subseteq F$, we have 
$\mu(F')<(\leqslant)t$. 
\item for any
subsheaf $ F' \subsetneq F$ 
such that $s$ factors through $F'$, 
we have 
$\mu(F/F')>(\geqslant)t$. 
\end{enumerate}
\end{lem}
\begin{proof}
Since the object $E=\Phi(\oO_X \to F)$ is of rank one
and $\Phi$ is an equivalence, 
any exact sequence $0 \to E'\to E \to E'' \to 0$ in 
$\aA_X$ is given by the image of
either one of the following exact sequences in $\bB_X$
\begin{align}\notag
&0 \to (0 \to F') \to (\oO_X \stackrel{s}{\to} F) \to (\oO_X \to F'') \to 0, \\
\notag
&0 \to (\oO_X \to F') \to (\oO_X \stackrel{s}{\to} F) \to (0 \to F'') \to 0. 
\end{align}
Therefore the lemma follows from the definition of $Z_t$-(semi)stability. 
\end{proof}

\begin{prop}\label{identify stab}
Let us fix $v \in H^{\ast}(X)$ of the form (\ref{Chern:v}) 
and set $q_t(m) \in \mathbb{Q}[m]$ to be
the constant polynomial
\begin{align}\label{def:q(m)}
q_t(m) \equiv (\beta \cdot \omega)\,t-n. 
\end{align}
Then 
a pair $(\oO_X \to F) \in \bB_X$
with $([F], \chi(F))=(\beta, n)$ 
 is 
$q_t$-(semi)stable if and only if 
$\Phi(\oO_X \to F) \in \aA_X$
is $Z_t$-(semi)stable. 
\end{prop}
\begin{proof}
For $F \in \Coh_{\leqslant1}(X)$
with $([F], \chi(F))=(\beta, n)$, 
its Hilbert polynomial is 
written as
\begin{align*}
\chi(F(m))=(\beta \cdot \omega)m+n. 
\end{align*}
Therefore we have
\begin{align*}
\frac{\chi(F(m))}{r(F)}=m+\mu(F), \quad
\frac{\chi(F(m))+q_t(m)}{r(F)}=m+t. 
\end{align*}
Thus (\ref{LP1}), (\ref{LP2}) 
are equivalent to the conditions (i), (ii) 
in Lemma~\ref{lem:charactrize} respectively. 
\end{proof}

Let us take an element $v \in H^{\ast}(X)$ of the form (\ref{Chern:v}). 
For each $t \in \mathbb{R}$, 
we denote by 
\begin{align}\label{stack:AX}
P_n^t(X, \beta) \subset \pP_n^t(X, \beta) \subset \underline{\aA}_X(v)
\end{align}
the open substacks of $Z_t$-stable (semistable) 
objects in $\aA_X$
with Chern character $v$. 
Because of Theorem \ref{prop:stack:AB} and Proposition \ref{identify stab}, 
they are also identified with 
open substacks of $\underline{\bB}_X(v)$ 
parametrizing pairs satisfying (i), (ii) in Lemma~\ref{lem:charactrize}. 
Below we write a $\mathbb{C}$-valued point 
of the stacks in (\ref{stack:AX}) 
as a pair $(\oO_X \to F)$
by the above identification.

\begin{thm}\label{existence of proj moduli space}
For $\beta \in H_2(X, \mathbb{Z})$ and $n\in \mathbb{Z}$, 
the moduli space $P^t_n(X, \beta)$ 
is a quasi-projective scheme, and 
$\pP^t_n(X, \beta)$ admits a good moduli space
\begin{align*}
\pP^t_n(X, \beta) \to \overline{P}_n^t(X, \beta),
\end{align*}
where $\overline{P}_n^t(X, \beta)$ is a projective 
scheme which parametrizes $Z_t$-polystable objects. 
\end{thm}
\begin{proof}
By Theorem \ref{prop:stack:AB}, Proposition \ref{identify stab}, $P^t_n(X, \beta)$, $\pP_n(X, \beta)$
are isomorphic to the moduli spaces of 
$q_t$-stable (semistable) pairs in $\bB_X$ (with $q_t\equiv(\beta \cdot \omega)\,t-n$). The latter has a GIT construction due to Le Potier's work on
semistable coherent systems (\cite[Thm.~4.11]{LePotier}, \cite[pp.~164]{JS}).
\end{proof}
Here by Proposition~\ref{identify stab}, 
a rank one $Z_t$-polystable object in $\aA_X$ 
is of the following form
\begin{align}\label{polystable}
(\oO_X \to F_0) \oplus \bigoplus_{i=1}^k V_i \otimes F_i[-1],
\end{align}
where $(\oO_X \to F_0)$ is $Z_t$-stable, 
each $F_i$ for $1\leqslant i\leqslant k$ are mutually non-isomorphic $\mu$-stable 
one dimensional sheaves with $\mu(F_i)=t$, 
and $V_i$ are finite dimensional vector spaces. 

As usual, there is a wall-chamber structure for $Z_t$-stability, where moduli spaces of stable objects stay unchanged inside chambers. 
Namely 
there is a finite set of points $W \subset \mathbb{R}$ such that 
we have 
\begin{align*}
P_n(X, \beta)=\pP_n^t(X, \beta)=\overline{P}_n^t(X, \beta), \quad  t \notin W.
\end{align*}
In particular, $P_n(X, \beta)$ is a projective scheme for $t \notin W$. 
Let $t_0 \in W$ and set 
$t_{\pm}=t_0 \pm \varepsilon$ for $0<\varepsilon \ll 1$. 
We have open immersions
\begin{align*}
P_n^{t_+}(X, \beta) \subset \pP_n^{t_0}(X, \beta)
\supset P_n^{t_-}(X, \beta),
\end{align*}
which induce the following flip type diagram 
of good moduli spaces
\begin{align}\label{diagram:wall}
\xymatrix{
P_n^{t_+}(X, \beta) \ar[rd]_{\pi^+} & & \ar[ld]^{\pi^-} P_n^{t_-}(X, \beta) \\
& \overline{P}_n^{t_0}(X, \beta). & 
}
\end{align}

\subsection{PT stable pairs and JS stable pairs}\label{sect on PT/JS pairs}
We discuss two interesting chambers for $Z_t$-stability. 
Recall the following two notions of stable pairs. 

\begin{defi}\emph{(\cite{PT, JS})}\label{defi:PTJSpair}

(i) A pair $(\oO_X \stackrel{s}{\to} F) \in \aA_X$ is 
called a PT stable pair if
$F$ is a pure one dimensional sheaf and $s$ is surjective in dimension one. 

(ii) A pair $(\oO_X \stackrel{s}{\to} F) \in \aA_X$
is called a JS stable pair if $s$ is a non-zero morphism, $F$ is $\mu$-semistable and 
for any subsheaf $0\neq F' \subsetneq F$ such that $s$ factors through 
$F'$ we have $\mu(F')<\mu(F)$. 
\end{defi}
\begin{rmk}
Strictly speaking, JS stable pairs in \cite[Def.~12.2]{JS} are of type $(\oO_X(-N)\to F)$ for a sufficiently negative line bundle $\oO_X(-N)$,
different from the definition given here. 
\end{rmk}
The above stable pairs appear as 
$Z_t$-stable objects in some chambers. 
\begin{prop}\label{prop:chambers}
For a fixed $v \in H^{\ast}(X, \mathbb{Q})$ 
of the form (\ref{Chern:v}), we have the following: 

(i) There exists $t(v)>0$ such that 
an object $(\oO_X \stackrel{s}{\to} F) \in \aA_X$
with Chern character $v$
is $Z_t$-stable for $t>t(v)$ if and only if 
it is a PT stable pair. 

(ii) There exists $\varepsilon(v)>0$ such that 
an object $(\oO_X \stackrel{s}{\to} F) \in \aA_X$
with Chern character $v$
is $Z_t$-stable for $\frac{n}{\omega \cdot \beta}
<t<\frac{n}{\omega \cdot \beta}+\varepsilon(v)$
if and only if it is JS stable pair. 

(iii) For $t<\frac{n}{\omega \cdot \beta}$, 
there is no $Z_t$-semistable object 
$(\oO_X \stackrel{s}{\to} F) \in \aA_X$
with Chern character $v$. 
\end{prop}
\begin{proof}
(i) follows from the same argument for CY 3-folds, 
see e.g.~\cite[Prop.~5.4 (i)]{Toda5}. 
(ii) easily follows from Lemma~\ref{lem:charactrize}, by setting 
$\epsilon(v)$ satisfying the following condition
\begin{align*}
\frac{n'}{\omega \cdot \beta'}+\epsilon(v)<
\frac{n}{\omega \cdot \beta},
\end{align*}
for any effective class $\beta' \in H_2(X, \mathbb{Z})$
and $n' \in \mathbb{Z}$
satisfying $n'/(\omega \cdot \beta')<n/(\omega \cdot \beta)$. 
(iii) also follows by setting $F'=F$ in 
Lemma~\ref{lem:charactrize}. 
\end{proof}
Following Proposition~\ref{prop:chambers}, 
we discuss three distinguished chambers as follows. 
\begin{enumerate}
\renewcommand{\labelenumi}{(\roman{enumi})}
\item 
  \textbf{Pandharipande-Thomas chamber}.
For $t \to \infty$, 
we have the moduli space of PT stable pairs
(see Proposition~\ref{prop:chambers} (i))
\begin{align*}
P_n(X,\beta) \cneq P^t_n(X,\beta)\big|_{t\to \infty}\,. 
\end{align*}

\item \textbf{Joyce-Song chamber}. For 
$t=\frac{n}{\omega \cdot \beta}+0$,
we have 
the moduli space of JS stable pairs 
(see Proposition~\ref{prop:chambers} (ii))
\begin{align*}
P^{\mathrm{JS}}_n(X,\beta)
\cneq P^t_n(X,\beta)\big|_{t=\frac{n}{\omega\cdot\beta}+0}\,. 
\end{align*}

\item \textbf{Empty chamber}.
For $t<\frac{n}{\omega \cdot \beta}$, 
we have (see Proposition~\ref{prop:chambers} (iii))
\begin{align*}
P_n^t(X, \beta)=\emptyset. 
\end{align*} 
\end{enumerate}

For $(\beta, n)$, we denote by 
\begin{align*}
\iota_M \colon 
\mM_n(X, \beta) \to M_n(X, \beta)
\end{align*}
the moduli stack of $\mu$-semistable one dimensional sheaves $F$
on $X$ with $([F], \chi(F))=(\beta, n)$, 
and its good moduli space parametrizing $\mu$-polystable objects. 
Since the target of JS stable pair $(\oO_X \to F)$ is $\mu$-semistable, 
we have a natural morphism
\begin{align}\label{map:JStoM}
P_n^{\rm{JS}}(X, \beta) \to M_n(X, \beta), \quad 
(\oO_X \to F) \mapsto \iota_M(F). 
\end{align}
In this way, 
the wall-crossing diagrams (\ref{diagram:wall}) 
relate $P_n(X, \beta)$, $P_n^{\rm{JS}}(X, \beta)$
and $M_n(X, \beta)$ in terms of 
flip type diagrams as in (\ref{diagram:wall}) (see also the diagram in Example \ref{exam:locP2}). 

Nevertheless, in some cases there is no wall in $t>\frac{n}{\omega \cdot \beta}$
so that PT stable pairs and JS stable pairs coincide. 
For an effective curve class $\beta \in H_2(X, \mathbb{Z})$, 
we define
\begin{align*}
n(\beta) \cneq \mathrm{inf}\big\{ \chi(\oO_C) : 
C \subset X \mbox{ is a one dimensional closed subscheme with }
[C]=\beta \big\} >-\infty. 
\end{align*}
Then we state the following proposition:
\begin{prop}\label{prop:nowall}
Let $\beta$ be an effective curve class and $(\beta, n) \in H_2(X, \mathbb{Z}) \oplus \mathbb{Z}$. Suppose the following inequality 
\begin{align}\label{ineq:betan}
\frac{n}{\omega \cdot \beta}
\leqslant \frac{n(\beta')}{\omega \cdot \beta'}
\end{align}
holds for any effective class $0<\beta' < \beta$. 
Then $P_n^t(X, \beta)$ is independent of $t$ if $t>\frac{n}{\omega \cdot \beta}$. 
In particular in this case, 
we have $P_n(X, \beta)=P_n^{\rm{JS}}(X, \beta)$. 
\end{prop}
\begin{proof}
We first show that any $Z_t$-stable pair 
$(\oO_X \stackrel{s}{\to} F) \in P_n^t(X, \beta)$ is a PT stable pair
if $t>\frac{n}{\omega \cdot \beta}$. 
Note that the $Z_t$-stability always implies that 
$F$ is a pure one dimensional sheaf. 
The image of $s$ is written as $\oO_{C'}$
for a one dimensional subscheme $C'$ 
such that $\beta'=[C'] \leqslant \beta$. 
Suppose by contradiction that $\beta'<\beta$
so that $\beta-\beta'>0$, and 
set $Q=F/\oO_{C'}$. 
We have two exact sequences in $\aA_X$
\begin{align*}
&0 \to (0 \to \oO_{C'}) \to (\oO_X \to F) \to (\oO_X \to Q) \to 0, \\
&0 \to (\oO_X \to \oO_{C'}) \to (\oO_X \to F) \to (0 \to Q) \to 0. 
\end{align*}
Then the $Z_t$-stability yields
\begin{align*}
\frac{n(\beta')}{\omega \cdot \beta'} \leqslant
\frac{\chi(\oO_{C'})}{\omega \cdot \beta'}
< t< \frac{n-\chi(\oO_{C'})}{\omega \cdot (\beta-\beta')}
\leqslant \frac{n-n(\beta')}{\omega \cdot (\beta-\beta')}. 
\end{align*}
The above inequalities contradict with the inequality (\ref{ineq:betan}).

Conversely, we show that any PT stable pair $(\oO_X \to F) \in P_n(X, \beta)$
is $Z_t$-stable for $t>\frac{n}{\omega \cdot \beta}$.  
It is enough to show that for any subsheaf $F' \subset F$
with $([F'], \chi(F'))=(\beta', n')$, we have 
\begin{align*}
\mu(F')=\frac{n'}{\omega \cdot \beta'} \leqslant \frac{n}{\omega \cdot \beta} <t.
\end{align*}
The above inequality is obvious if $\beta'=\beta$, so 
we may assume that $\beta'<\beta$ and set $\beta''=\beta-\beta'$. 
The composition $\oO_X \to F \to F/F'$ is surjective in dimension 
one, so we have
$
\chi(F/F')=n-n' \geqslant n(\beta'')$. 
Therefore we have 
\begin{align*}
\frac{n'}{\omega \cdot \beta'} \leqslant 
\frac{n-n(\beta'')}{\omega \cdot (\beta-\beta'')} \leqslant
\frac{n}{\omega \cdot \beta},
\end{align*}
where the last inequality follows from (\ref{ineq:betan})
for $\beta''$. 
Therefore we obtain the proposition. 
\end{proof}

\begin{exam}\label{exam:locP2}
Let $X$ be the non-compact CY 4-fold
given by
\begin{align*}
X=\mathrm{Tot}_{\mathbb{P}^2}(\oO_{\mathbb{P}^2}(-1) \oplus 
\oO_{\mathbb{P}^2}(-2)).
\end{align*}
Let $[l] \in H_2(X, \mathbb{Z})=H_2(\mathbb{P}^2, \mathbb{Z})$
be the class of a line. 
In this case, 
the numerical class
$(4[l], 1)$ does not satisfy 
the condition (\ref{ineq:betan}) 
as $n(3[l])=0$. 
Indeed a PT stable pair 
in $P_1(X, 4[l]))$
is destabilized at $t=1$
if and only if 
it is of the form $(I_C \stackrel{s}{\to} \oO_l)$
for a cubic curve $C \subset \mathbb{P}^2$, 
a line $l \subset \mathbb{P}^2$, 
and a non-zero morphism $s$. 
The destabilizing sequence is given by 
\begin{align*}
0 \to \oO_l[-1] \to (I_C \stackrel{s}{\to} \oO_l) \to I_C \to 0.
\end{align*}
We can show that $t=1$ is the only wall and we have 
the following wall-crossing phenomena of 
moduli spaces of $Z_t$-stable objects
\begin{align*}
P_1^t(X, 4[l])=
\left\{ \begin{array}{cc}
P_1(X, 4[l]),  & t>1, \\
& \\
P_1^{\rm{JS}}(X, 4[l]), & \frac{1}{4}<t<1, \\
& \\
\emptyset, & t<\frac{1}{4}.
\end{array}  \right. 
\end{align*}
The corresponding flip type diagram is
\begin{align*}
\xymatrix{
P_1(X, 4[l]) \ar[rd] & & P_1^{\rm{JS}}(X, 4[l]) \ar[ld]
\ar[rd] & \\
& \overline{P}_1^{t=1}(X, 4[l])  & & M_1(X, 4[l]). 
}
\end{align*}

\end{exam}

\section{$\mathrm{DT}_4$ type invariants for $Z_t$-stable pairs}
 
\subsection{Review of $\mathrm{DT}_4$ invariants}\label{review DT4}
Before defining $\mathrm{DT}_4$ type counting invariants associated with moduli spaces $P^t_n(X,\beta)$ of $Z_t$-stable objects,
we first introduce the set-up of $\mathrm{DT}_4$ invariants.
We fix an ample divisor $\omega$ on $X$
and take a cohomology class
$v \in H^{\ast}(X, \mathbb{Q})$.

The coarse moduli space $M_{\omega}(v)$
of $\omega$-Gieseker semistable sheaves
$E$ on $X$ with $\ch(E)=v$ exists as a projective scheme.
We always assume that
$M_{\omega}(v)$ is a fine moduli space, i.e.
any point $[E] \in M_{\omega}(v)$ is stable and
there is a universal family
\begin{align}\label{universal}
\eE \in \Coh(X \times M_{\omega}(v)).
\end{align}
For instance, moduli spaces of one dimensional stable sheaves $E$ with $\chi(E)=1$ and Hilbert schemes of 
closed subschemes satisfy this assumption \cite{CaoFano, CK1, CK2, CMT1}.

In~\cite{BJ, CL1}, under certain hypotheses,
the authors construct 
a $\mathrm{DT}_{4}$ virtual
class
\begin{align}\label{virtual}
[M_{\omega}(v)]^{\rm{vir}} \in H_{2-\chi(v, v)}(M_{\omega}(v), \mathbb{Z}), \end{align}
where $\chi(-,-)$ is the Euler pairing.
Notice that this class will not necessarily be algebraic.

Roughly speaking, in order to construct such a class, one chooses at
every point $[E]\in M_{\omega}(v)$, a half-dimensional real subspace
\begin{align*}\Ext_{+}^2(E, E)\subset \Ext^2(E, E)\end{align*}
of the usual obstruction space $\Ext^2(E, E)$, on which the quadratic form $Q$ defined by Serre duality is real and positive definite. 
Then one glues local Kuranishi-type models of form 
\begin{equation}\kappa_{+}=\pi_+\circ\kappa: \Ext^{1}(E,E)\to \Ext_{+}^{2}(E,E),  \nonumber \end{equation}
where $\kappa$ is a Kuranishi map of $M_{\omega}(v)$ at $E$ and $\pi_+$ is the projection 
according to the decomposition $\Ext^{2}(E,E)=\Ext_{+}^{2}(E,E)\oplus\sqrt{-1}\cdot\Ext_{+}^{2}(E,E)$.  

In \cite{CL1}, local models are glued in three special cases: 
\begin{enumerate}
\item when $M_{\omega}(v)$ consists of locally free sheaves only; 
\item  when $M_{\omega}(v)$ is smooth;
\item when $M_{\omega}(v)$ is a shifted cotangent bundle of a derived smooth scheme. 
\end{enumerate}
And the corresponding virtual classes are constructed using either gauge theory or algebro-geometric perfect obstruction theory.

The general gluing construction is due to Borisov-Joyce \cite{BJ} (one needs to assume that $M_{\omega}(v)$ can be given a $(-2)$-shifted symplectic structure as in Claim 3.29 \cite{BJ} to apply their constructions. For the $Z_t$-stable pairs case, we show this can be done in Lemma \ref{exist of -2 str}), based on Pantev-T\"{o}en-Vaqui\'{e}-Vezzosi's theory of shifted symplectic geometry \cite{PTVV} and Joyce's theory of derived $C^{\infty}$-geometry.
The corresponding virtual class is constructed using Joyce's
D-manifold theory (a machinery similar to Fukaya-Oh-Ohta-Ono's theory of Kuranishi space structures used in defining Lagrangian Floer theory).

In this paper, all computations and examples will only involve the virtual class constructions in situations (2), (3), mentioned above. We briefly 
review them as follows:  
\begin{itemize}
\item When $M_{\omega}(v)$ is smooth, the obstruction sheaf $Ob\to M_{\omega}(v)$ is a vector bundle endowed with a quadratic form $Q$ via Serre duality. Then the $\DT_4$ virtual class is given by
\begin{equation}[M_{\omega}(v)]^{\rm{vir}}=\mathrm{PD}(e(Ob,Q)).   \nonumber \end{equation}
Here $e(Ob, Q)$ is the half-Euler class of 
$(Ob,Q)$ (i.e. the Euler class of its real form $Ob_+$), 
and $\mathrm{PD}(-)$ is its 
Poincar\'e dual. 
Note that the half-Euler class satisfies 
\begin{align*}
e(Ob,Q)^{2}&=(-1)^{\frac{\mathrm{rk}(Ob)}{2}}e(Ob),  \textrm{ }\mathrm{if}\textrm{ } \mathrm{rk}(Ob)\textrm{ } \mathrm{is}\textrm{ } \mathrm{even}, \\
 e(Ob,Q)&=0, \textrm{ }\mathrm{if}\textrm{ } \mathrm{rk}(Ob)\textrm{ } \mathrm{is}\textrm{ } \mathrm{odd}. 
\end{align*}
\item When $M_{\omega}(v)$ is a $(-2)$-shifted cotangent bundle of a derived smooth scheme, roughly speaking, this means that at any closed point $[E]\in M_{\omega}(v)$, we have Kuranishi map of type
\begin{equation}\kappa:
 \Ext^{1}(E,E)\to \Ext^{2}(E,E)=V_E\oplus V_E^{*},  \nonumber \end{equation}
where $\kappa$ factors through a maximal isotropic subspace $V_E$ of $(\Ext^{2}(E,E),Q)$. Then the $\DT_4$ virtual class of $M_{\omega}(v)$ is, 
roughly speaking, the 
virtual class of the perfect obstruction theory formed by $\{V_E\}_{E\in M_{\omega}(v)}$. 
When $M_{\omega}(v)$ is furthermore smooth as a scheme, 
then it is
simply the Euler class of the vector bundle 
$\{V_E\}_{E\in M_{\omega}(v)}$ over $M_{\omega}(v)$. 
\end{itemize}
To construct the above virtual class (\ref{virtual}) with coefficients in $\mathbb{Z}$ (instead of $\mathbb{Z}_2$), we need an orientability result 
for $M_{\omega}(v)$, which is stated as follows.
Let  
\begin{equation}\label{det line bdl}
 \lL:=\mathrm{det}(\dR \hH om_{\pi_M}(\eE, \eE))
 \in \Pic(M_{\omega}(v)), \quad  
\pi_M \colon X \times M_{\omega}(v)\to M_{\omega}(v),
\end{equation}
be the determinant line bundle of $M_{\omega}(v)$, equipped with a symmetric pairing $Q$ induced by Serre duality.  An \textit{orientation} of 
$(\mathcal{L},Q)$ is a reduction of its structure group (from $O(1,\mathbb{C})$) to $SO(1, \mathbb{C})=\{1\}$; in other words, we require a choice of square root of the isomorphism
\begin{equation}\label{Serre duali}Q: \lL\otimes \lL \to \oO_{M_{\omega}(v)}  \end{equation}
to construct the virtual class (\ref{virtual}).
An orientability result was first obtained for $M_{\omega}(v)$ when the CY 4-fold $X$ satisfies
$\mathrm{Hol}(X)=SU(4)$ and $H^{\rm{odd}}(X,\mathbb{Z})=0$ \cite[Thm.~2.2]{CL2} and it has recently 
been generalized to arbitrary CY 4-folds \cite[Cor.~1.17]{CGJ}. 
Note that the set of orientations forms a torsor for $H^{0}(M_{\omega}(v),\mathbb{Z}_2)$.

\subsection{$Z_t$-stable pair invariants}\label{subsection stable pair invs}
 
For $\beta \in H_2(X, \mathbb{Z})$ and $n\in \mathbb{Z}$,
let
\begin{align}\label{PT:CY4}
P^t_n(X, \beta)
\end{align}
be the moduli space of $Z_t$-stable objects $(F, s)$ on $X$
such that $[F]=\beta$, $\chi(F)=n$.
By Theorem \ref{existence of proj moduli space}, it is a quasi-projective scheme whose closed points correspond to two-term complexes
\begin{align*}
I=(\oO_X \stackrel{s}{\to} F) \in \mathrm{D^b(\Coh(\textit{X\,}))},
\end{align*}
in the derived category of coherent sheaves on $X$, 
satisfying the $Z_t$-stability condition. 

Similar to moduli spaces of stable sheaves,
the moduli space $P^t_n(X, \beta)$
admits a deformation-obstruction theory, whose tangent, obstruction and `higher' obstruction spaces are
given by
\begin{align*}
\Ext^1(I, I)_0, \ \Ext^2(I, I)_{0},  \ \Ext^3(I, I)_{0},
\end{align*}
where $(-)_0$ denotes the trace-free part. Note that Serre duality gives an isomorphism $\Ext_0^1\cong (\Ext_0^3)^{\vee}$
and a non-degenerate quadratic form on $\Ext_0^2$. Moreover, we have
the following lemma: 
\begin{lem}\label{exist of -2 str}
The moduli space $P^t_n(X, \beta)$ can be given the structure of a $(-2)$-shifted symplectic derived scheme in the sense of Pantev-T\"{o}en-Vaqui\'{e}-Vezzosi \cite{PTVV}.
\end{lem}
\begin{proof}
By Proposition \ref{openess}, $P^t_n(X, \beta)$ is an open substack of the moduli stack of perfect complexes of coherent sheaves with trivial determinant on $X$, whose $(-2)$-shifted symplectic structure is constructed by~\cite[Thm.~0.1]{PTVV} (see also \cite[Sect.~3.2,~pp.~48]{PTVV} for pull-back to the determinant fixed substack).
\end{proof}
Let $\mathbb{I}$
be the universal pair
\begin{align}\label{Upair}
\mathbb{I}=(\mathcal{O}_{X\times P^t_n(X, \beta)}\rightarrow \mathbb{F}).
\end{align}
Then the determinant line bundle
\begin{align*}
\mathcal{L}:=\mathrm{det}(\dR \hH om_{\pi_P}(\mathbb{I}, \mathbb{I})_0)
\in \Pic(P^t_n(X, \beta))
\end{align*}
is endowed with a non-degenerate quadratic form $Q$ defined by Serre duality,
where $\pi_P\colon X\times P^t_n(X, \beta)\rightarrow P^t_n(X, \beta)$ is the projection.
Similarly as before, the orientability issue for the moduli space $P^t_n(X, \beta)$ is whether the structure group of the quadratic line bundle $(\mathcal{L},Q)$ can be reduced from $O(1,\mathbb{C})$ to $SO(1,\mathbb{C})=\{1\}$. 

As $P^t_n(X, \beta)$ is an open substack of the moduli stack of perfect complexes of coherent sheaves with trivial determinant on $X$,
it is always orientable in the above sense by \cite[Cor.~1.17]{CGJ}.
Combining this with Theorem \ref{existence of proj moduli space} and Lemma \ref{exist of -2 str}, 
we can construct their virtual classes. 
\begin{thm}\label{vir class of pair moduli}
Let $(X,\omega)$ be a smooth projective Calabi-Yau 4-fold, $\beta\in H_{2}(X,\mathbb{Z})$ and $n\in\mathbb{Z}$.
For a generic choice of $t\in \mathbb{R}$ such that $P^t_n(X, \beta)$ is projective, there exists a virtual class 
\begin{align}\label{pair moduli vir class}
[P^t_n(X, \beta)]^{\rm{vir}}\in H_{2n}\big(P^t_n(X, \beta),\mathbb{Z}\big),  \end{align}
in the sense of Borisov-Joyce \cite{BJ}, depending on the choice of orientation.
\end{thm}
As in \cite{CMT2}, we consider primary insertions: for integral classes $\gamma \in H^{4}(X,\mathbb{Z})$, let 
\begin{align}\label{pri insertion}
\tau \colon H^{4}(X,\mathbb{Z})\to H^{2}(P^t_n(X,\beta),\mathbb{Z}), \quad 
\tau(\gamma):=(\pi_{P})_{\ast}(\pi_X^{\ast}\gamma \cup\ch_3(\mathbb{F}) ),
\end{align}
where $\pi_X$, $\pi_P$ are projections from $X \times P^t_n(X,\beta)$
to corresponding factors, 
$\mathbb{F}$ is the target of the universal pair (\ref{Upair}),
and $\ch_3(\mathbb{F})$ is the
Poincar\'e dual to the fundamental cycle of $\mathbb{F}$.
\begin{defi}
For a generic $t \in \mathbb{R}$, the 
$Z_t$-stable pair invariant
is defined to be 
\begin{align}\label{Z_t pair inv}P^t_{n,\beta}(\gamma):=\int_{[P^t_n(X,\beta)]^{\rm{vir}}} \tau(\gamma)^n\in \mathbb{Z}.
 \end{align}
Here we write $P^t_{0,\beta}=P^t_{n,\beta}(\gamma)$ when $n=0$. 
\end{defi}

As in Section \ref{sect on PT/JS pairs}, we are particularly interested in our invariants in the following two distinguished chambers.
\begin{enumerate}
\renewcommand{\labelenumi}{(\roman{enumi})}
\item 
When $t\to \infty$, $Z_t$-stable pairs are PT stable pairs.
We denote 
\begin{align}
\label{PT inv}P_{n,\beta}(\gamma):=P_{n,\beta}^t(\gamma)\big|_{t\to \infty}\,, \quad P_{0,\beta}:=P^t_{0,\beta}\big|_{t\to \infty}\,,
\end{align}
which has been studied before in \cite{CMT2}. \\

\item 
When $t=\frac{n}{\omega\cdot\beta}+0$, $Z_t$-stable pairs are Joyce-Song stable pairs. We denote
\begin{align}\label{JS inv}P^\mathrm{JS}_{n,\beta}(\gamma):=P^{t}_{n,\beta}(\gamma)\big|_{t=\frac{n}{\omega\cdot\beta}+0}\,, \quad P^\mathrm{JS}_{0,\beta}:=P^t_{0,\beta}\big|_{t=+0}\,. \end{align}
\end{enumerate}

\section{Conjectures}

\subsection{GW/GV conjecture}\label{app on GV inv}

Let $X$ be a smooth projective CY 4-fold. 
The genus 0 Gromov-Witten invariants on $X$ are defined using
insertions: for $\gamma \in H^{4}(X, \mathbb{Z})$,
one defines
\begin{equation}
\mathrm{GW}_{0, \beta}(\gamma)
=\int_{[\overline{M}_{0, 1}(X, \beta)]^{\rm{vir}}}\mathrm{ev}^{\ast}(\gamma)\in\mathbb{Q},
\nonumber \end{equation}
where $\mathrm{ev} \colon \overline{M}_{0, 1}(X, \beta)\to X$
is the evaluation map.

The \textit{genus 0 Gopakumar-Vafa type invariants} 
\begin{align}\label{g=0 GV}
n_{0, \beta}(\gamma) \in \mathbb{Q}
\end{align}
are defined by Klemm-Pandharipande \cite{KP} from the identity
\begin{align*}
\sum_{\beta>0}\mathrm{GW}_{0, \beta}(\gamma)q^{\beta}=
\sum_{\beta>0}n_{0, \beta}(\gamma) \sum_{d=1}^{\infty}
d^{-2}q^{d\beta}.
\end{align*}

For genus 1 case, virtual dimensions of moduli spaces of stable maps are zero, so Gromov-Witten invariants
\begin{align*}
\mathrm{GW}_{1, \beta}=\int_{[\overline{M}_{1, 0}(X, \beta)]^{\rm{vir}}}
1 \in \mathbb{Q}
\end{align*}
can be defined
without insertions.
The \textit{genus 1 Gopakumar-Vafa type invariants}
\begin{align}\label{g=1 GV}
n_{1, \beta} \in \mathbb{Q}
\end{align}
 are defined in~\cite{KP} by the identity
\begin{align*}
\sum_{\beta>0}
\mathrm{GW}_{1, \beta}q^{\beta}=
&\sum_{\beta>0} n_{1, \beta} \sum_{d=1}^{\infty}
\frac{\sigma(d)}{d}q^{d\beta}
+\frac{1}{24}\sum_{\beta>0} n_{0, \beta}(c_2(X))\log(1-q^{\beta}) \\
&-\frac{1}{24}\sum_{\beta_1, \beta_2}m_{\beta_1, \beta_2}
\log(1-q^{\beta_1+\beta_2}),
\end{align*}
where $\sigma(d)=\sum_{i|d}i$ and $m_{\beta_1, \beta_2}\in\mathbb{Z}$ are called meeting invariants which can be inductively determined by genus 0 Gromov-Witten invariants.
In~\cite{KP}, both of the invariants (\ref{g=0 GV}), (\ref{g=1 GV})
are conjectured to be integers, and Gromov-Witten invariants on $X$ are computed  to support the conjectures in many examples
by localization technique or mirror symmetry.

\subsection{Katz/GV conjecture}
In \cite{CMT1}, Maulik and the authors define $\DT_4$ counting invariants for one dimensional stable sheaves and use them to
give a sheaf theoretical interpretation to the genus 0 GV type invariants (\ref{g=0 GV}).

To be precise, we consider the moduli scheme $M_1(X, \beta)$ of 
one dimensional stable sheaves $F$ on $X$ with $[F]=\beta\in H_2(X,\mathbb{Z})$ and $\chi(F)=1$. 
The spherical twist (here we need to assume $h^{0,1}(X)=h^{0,2}(X)=0$,~see~\cite[Def.~0.1]{ST}):
$$\Phi_{\oO_X}(\bullet)=\mathrm{cone}\big(\RHom(\oO_X,\bullet)\otimes \oO_X\to \bullet\, \big)$$
identifies $M_1(X, \beta)$ with some moduli stack of rank one objects in $\mathrm{D^b(\Coh(\textit{X\,}))}$.
As in Theorem \ref{vir class of pair moduli}, there exists a virtual class 
$$[M_1(X,\beta)]^{\mathrm{vir}}\in H_2(M_1(X,\beta),\mathbb{Z}). $$
Consider primary insertions: for integral classes $\gamma \in H^{4}(X, \mathbb{Z})$, let 
\begin{align}\label{insertion for one dim sheaves} 
\tau \colon H^{4}(X, \mathbb{Z})\to H^{2}(M_1(X,\beta), \mathbb{Z}), \quad 
\tau(\gamma):=(\pi_{M})_{\ast}(\pi_X^{\ast}\gamma \cup\ch_3(\mathbb{F}) ),
\end{align}
where $\pi_X$, $\pi_M$ are projections from $X \times M_1(X,\beta)$
to corresponding factors, $\mathbb{F}$ is the universal sheaf and $\ch_3(\mathbb{F})$ is the
Poincar\'e dual to the fundamental cycle of $\mathbb{F}$.

The following may be thought as an analogue of Katz/GV conjecture \cite{Katz} on CY 4-folds.
\begin{conj}\emph{(\cite[Conjecture 0.2]{CMT1})}\label{g=0 one dim sheaf conj}
For certain choice of orientation, we have 
$$\int_{[M_{1}(X,\beta)]^{\rm{vir}}}\tau(\gamma)=n_{0,\beta}(\gamma), $$
where $n_{0,\beta}(\gamma)$ is the g=0 Gopakumar-Vafa type invariant (\ref{g=0 GV}).
\end{conj}
\begin{rmk}
See also \cite{CT} for a discussion on the higher genus case. 
\end{rmk}

\subsection{PT/GV conjecture}
In \cite{CMT2}, Maulik and the authors define Pandharipande-Thomas type invariants (\ref{PT inv}) on Calab-Yau 4-folds and conjecture the following PT/GV correspondence. 
\begin{conj}\label{PT/GV conj} \emph{(\cite[Conj.~1.1,~1.2,~Sect.~1.7]{CMT2})}
Let $(X,\omega)$ be a smooth projective Calabi-Yau 4-fold, $\beta\in H_2(X,\mathbb{Z})$ and $n\in\mathbb{Z}_{\geqslant0}$.
Fix $\gamma \in H^{4}(X, \mathbb{Z})$, then for certain choice of orientation, we have 
\begin{align*}
P_{n,\beta}(\gamma)=\sum_{\begin{subarray}{c}\beta_0+\beta_1+\cdots+\beta_n=\beta, \\ \omega\cdot\beta_i>0,\,  i=1,\ldots,n \end{subarray}}
P_{0,\beta_0}\cdot\prod_{i=1}^nn_{0,\beta_i}(\gamma), \quad
\sum_{\beta \geqslant 0}
P_{0, \beta}q^{\beta}=
\prod_{\beta>0} M\big(q ^{\beta}\big)^{n_{1, \beta}}.
\end{align*}
Here $n_{0,\beta}(\gamma)$, $n_{1,\beta}$ are the Gopakumar-Vafa type invariants (\ref{g=0 GV}), (\ref{g=1 GV}) of $X$. And
$M(q)=\prod_{k\geqslant 1}(1-q^{k})^{-k}$ is the MacMahon function.
\end{conj}

\subsection{Main conjecture}
Pandharipande-Thomas stable pairs are $Z_t$-stable pairs when $t\to \infty$ (Definition \ref{def of Z_t stability}),
and the corresponding PT type invariants (\ref{PT inv}) are $Z_{t\to \infty}$-stable pair invariants (\ref{Z_t pair inv}).
The following main conjecture of this paper generalizes Conjecture \ref{PT/GV conj}.
\begin{conj}\label{main conj}
Let $(X,\omega)$ be a smooth projective Calabi-Yau 4-fold, $\beta\in H_2(X,\mathbb{Z})$ and $n\in\mathbb{Z}_{\geqslant0}$.
Choose a generic $t\in \mathbb{R}_{>0}$. 
Then for certain choice of orientation, we have 
\begin{align}\label{WCF:main}
P^t_{n,\beta}(\gamma)=\sum_{\begin{subarray}{c}\beta_0+\beta_1+\cdots+\beta_n=\beta  \\  \omega\cdot\beta_i>\frac{1}{t},\,  i=1,\ldots,n \end{subarray}}P_{0,\beta_0}\cdot\prod_{i=1}^nn_{0,\beta_i}(\gamma),
\end{align}
where $\gamma \in H^{4}(X, \mathbb{Z})$, $n_{0,\beta}(\gamma)$ is the genus $0$ Gopakumar-Vafa type invariant (\ref{g=0 GV}).

In particular, $P^t_{0,\beta}=P_{0,\beta}$ is independent of the choice of $t>0$.
\end{conj}
Combining with the second formula in Conjecture \ref{PT/GV conj}, we may express $Z_t$-stable pair invariants in terms of all genus 
GV type invariants of $X$. 
Indeed let us consider the generating 
series for $\gamma \in H^4(X, \mathbb{Z})$:
\begin{align}\label{PTt:series}
\PT^t(X)(\exp(\gamma)) \cneq \sum_{n\in \mathbb{Z}, \beta \geqslant 0} \frac{P_{n, \beta}^t(\gamma)}{n!}
y^n q^{\beta}. 
\end{align}
Then for a very generic $t \in \mathbb{R}$ (very generic means outside a countable subset of rational numbers in $\mathbb{R}$), 
the identity in Conjecture~\ref{main conj}
together with the second formula in Conjecture~\ref{PT/GV conj}
implies that
\begin{align*}
\PT^t(X)(\exp(\gamma))=\prod_{\omega \cdot \beta>\frac{1}{t}}
\exp(yq^{\beta})^{n_{0, \beta}(\gamma)} \cdot 
\prod_{\beta>0} M(q^{\beta})^{n_{1, \beta}}. 
\end{align*}
By taking the $t \to \infty$ limit, 
we recover the conjectural formula (\ref{intro:PTexp}). 

\subsection{JS/GV conjecture}
In the Joyce-Song chamber, there are two particularly interesting special cases of Conjecture \ref{main conj}.
\begin{conj}\label{JS/GV conj}\emph{(Special case of Conjecture \ref{main conj})}
In the same setting as Conjecture \ref{main conj}, we have 
\begin{align*}(1)\,\, P^{\mathrm{JS}}_{n,\beta}(\gamma)=\sum_{\begin{subarray}{c}\beta_1+\cdots+\beta_n=\beta  \\  \omega\cdot\beta_i= \frac{\omega \cdot \beta}{n},\,  i=1,\ldots,n \end{subarray}}\prod_{i=1}^nn_{0,\beta_i}(\gamma), \,\, \mathrm{if}\,\, n\geqslant 1, \quad  (2) \,\, P^{\mathrm{JS}}_{0,\beta}=P_{0,\beta}. 
\end{align*}
In particular, $P^{\mathrm{JS}}_{1,\beta}(\gamma)=n_{0,\beta}(\gamma)$.
\end{conj}
When $n=1$, we recover genus 0 GV type invariants $n_{0,\beta}(\gamma)$ (\ref{g=0 GV}). While in the $n=0$ case, we recover 
genus 1 GV type invariants $n_{1,\beta}$ (\ref{g=1 GV}) by assuming the conjectural relation  
between $P_{0,\beta}$ and $n_{1,\beta}$ (as in Conjecture \ref{PT/GV conj}):
\begin{align*}
\sum_{\beta \geqslant 0}
P_{0, \beta}q^{\beta}=
\prod_{\beta>0} M\big(q ^{\beta}\big)^{n_{1, \beta}}.
\end{align*}
Therefore (conjecturally), we may use $\DT_4$ counting invariants for (semi)stable one dimensional sheaves together with sections 
(more precisely JS stable pairs) to recover 
all genus GV type invariants of Calabi-Yau 4-folds.

\section{Heuristic explanations of the main conjecture}
Our main conjecture~\ref{main conj} is difficult to prove, due to the 
difficulties of $\mathrm{DT}_4$-virtual classes and 
the absence of wall-crossing formulae available in Donaldson invariants~\cite{Moc}
and Donaldson-Thomas invariants~\cite{JS, KS}. 
Here we give heuristic explanations of our main conjecture 
from the viewpoint of ideal geometry, master space argument and 
a virtual push-forward formula. 

\subsection{Heuristic argument on ideal CY 4-folds}\label{heuristic argument}

In this subsection, we give a heuristic argument to explain why we expect Conjecture \ref{main conj} to be true
in an ideal CY4 geometry. In this heuristic discussion, we ignore the issue of orientations.

Let $X$ be an `ideal' CY 4-fold
in the sense that all curves of $X$ deform in families of expected dimensions, and have expected generic properties, i.e.
\begin{enumerate}
\item
any rational curve in $X$ is a chain of smooth $\mathbb{P}^1$ with normal bundle $\mathcal{O}_{\mathbb{P}^{1}}(-1,-1,0)$, and
moves in a compact 1-dimensional smooth family of embedded rational curves, whose general member is smooth with 
normal bundle $\mathcal{O}_{\mathbb{P}^{1}}(-1,-1,0)$. 
\item
any elliptic curve $E$ in $X$ is smooth, super-rigid, i.e. 
the normal bundle is 
$L_1 \oplus L_2 \oplus L_3$
for general degree zero line bundle $L_i$ on $E$
satisfying $L_1 \otimes L_2 \otimes L_3=\oO_E$. 
Furthermore any two elliptic curves are 
disjoint and also disjoint with rational curve families.

\item
there is no curve in $X$ with genus $g\geqslant 2$.
\end{enumerate}
For the moduli space $P^t_{n}(X,\beta)$ of $Z_t$-stable pairs, we want to compute 
\begin{align*}\int_{[P^t_{n}(X,\beta)]^{\rm{vir}}}\tau(\gamma)^n, \quad \gamma\in H^4(X,\mathbb{Z}),  \end{align*}
when $X$ is an ideal CY 4-fold.
Let $\{Z_i\}_{i=1}^n$ be 4-cycles which represent the class $\gamma$. For dimension reasons,
we may assume for any $i\neq j$
the rational curves which meet with $Z_i$ are
disjoint from those with $Z_j$. 
The insertions cut out the moduli space and pick up stable pairs whose support intersects with all $\{Z_i\}_{i=1}^n$.
We denote the moduli space of such `incident' stable pairs by
\begin{align*}Q^t_{n}(X,\beta;\{Z_i\}_{i=1}^n)\subseteq P^t_{n}(X,\beta).  \end{align*}
Then we claim that 
\begin{align}\label{Q_n:identity}
Q^t_{n}(X,\beta;\{Z_i\}_{i=1}^n)=\coprod_{\begin{subarray}{c}
\beta_0+\beta_1+\cdots+\beta_n=\beta   \\
\omega\cdot\beta_i>\frac{1}{t},\,  i=1,\ldots,n
\end{subarray}}P_{0}(X,\beta_0)\times Q_{1}(X,\beta_1;Z_1)\times \cdots \times Q_{1}(X,\beta_n;Z_n), \end{align}
where $Q_{1}(X,\beta_i;Z_i)$ is the (finite) set of rational curves (in class $\beta_i$) which meet with $Z_i$. 

Indeed let us take a $Z_t$-stable pair $(\oO_X\stackrel{s}{\to}  F)$ in 
$Q^t_{n}(X,\beta;\{Z_i\}_{i=1}^n)$. 
Then $F$ decomposes into a direct sum 
\begin{align*}
F=F_0 \oplus \bigoplus_{i=0}^n F_i  \oplus F_{n+1},
\end{align*}
 where 
$F_0$ is supported on elliptic curves, each $F_i$ for $1\leqslant i\leqslant n$ is 
supported on smooth rational curves which meet with 
$Z_i$, and $F_{n+1}$ is supported on rational curves without incident condition.
Here each $F_i$ for $1\leqslant i\leqslant n$ is non-zero due to the 
incidence condition, but $F_0$ and $F_{n+1}$ are possibly zero. 

We take the Harder-Narasimhan filtration of $F_i$ for $0<i \leqslant n$
$$0\subset F_{i,1}\subset F_{i,2}\subset\cdots \subset F_{i,{n_i}}=F_i. $$
If $s'=0$ in the following diagram
\begin{align*}
\xymatrix{
\oO_{X}   \ar[r]^{s} \ar[rd]_{s'} & F_{i,n_i}  \ar@{->>}[d] \\
     &  \quad F_{i,n_i}/F_{i,{n_{i}-1}},}
\end{align*}
then $Z_t$ stability and the HN filtration property gives
$$0<t< \mu(F_{i,{n_i}}/F_{i,{n_i-1}})<\mu(F_{i,{n_i-1}}/F_{i,{n_i-2}})<\cdots < \mu(F_{i,1}), $$
hence also 
$$0<t< \mu(F_{i,{n_i}})=\mu(F_{i}).$$ 

If $s'\neq 0$, then 
the semistable sheaf $F_{i,{n_i}}/F_{i,{n_{i}-1}}$ has a section. 
By noting that $F_{i,{n_i}}/F_{i,{n_{i}-1}}$
is supported on rational curves, 
we see that its Jordan-H\"{o}lder
factors also have sections, 
so  
we conclude 
\begin{align}\label{ineq:chi}
\chi(F_{i,{n_i}}/F_{i,{n_{i}-1}})>0. 
\end{align}
So we have the inequalities
$$0<\mu(F_{i,{n_i}}/F_{i,{n_i-1}})<\mu(F_{i,{n_i-1}}/F_{i,{n_i-2}})<\cdots < \mu(F_{i,1}).$$
Hence we also have $\mu(F_i)>0$. Thus in either case, we have 
$$\chi(F_i)>0, \quad 1\leqslant i\leqslant n.$$
Similar argument also gives 
$\chi(F_0) \geqslant 0$, and $\chi(F_{n+1})> 0$
if $F_{n+1}$ is non-zero. 
Here  
$\chi(F_0)$ can be zero even if $F_0$ is non-zero, 
since the inequality (\ref{ineq:chi}) 
is replaced by $\chi(F_{0,{n_0}}/F_{0,{n_{0}-1}}) \geqslant 0$
as it is supported on an elliptic curve. 
From the identity
$$n=\chi(F)=\chi(F_0)+\sum_{i=1}^n\chi(F_i)+\chi(F_{n+1}), $$
we conclude that $\chi(F_0)=0$, $\chi(F_i)=1$ for $1\leqslant i\leqslant n$, $F_{n+1}=0$, and 
all $F_i$ ($i\geqslant 0$) are semistable.
Further argument using Jordan-H\"{o}lder filtration shows that $F_i$ ($i\geqslant 1$) are stable (otherwise $\chi(F_i)>1$).
Hence $F_i\cong \oO_{C_i}$ for some rational curve
$\mathbb{P}^1 \cong C_i \subset X$. 

Next, we discuss the role of section $s$.
We write 
\begin{align}\label{splus}
s=s_0\oplus \bigoplus_{i=1}^n s_i
\colon\oO_X\to F_0\oplus \bigoplus_{i=0}^n F_i.
\end{align}
Note that $s_i$ for $i\geqslant 1$ is either zero or surjective,
since $F_i\cong \oO_{C_i}$ as we mentioned above. If $s_i=0$, 
then the pair $(\oO_X \to F)$ decomposes as  
$$(\oO_X \to F)=(\oO_X \to *)\oplus (0\to F_i), $$
which violates the $Z_t$-stability 
of $(\oO_X \to F)$. 
Hence $s_i$ is surjective.
Similarly, $s_0$ is also surjective, so $F_0=\oO_Z$ is an iterated extension of $\oO_E$ for an elliptic curve $E$.

The argument
implies that, by setting $\beta_i=[F_i]$ for the pair (\ref{splus}), 
we have the inclusion 
\begin{align}\notag
Q^t_{n}(X,\beta;\{Z_i\}_{i=1}^n)
\subset \coprod_{\begin{subarray}{c}
\beta_0+\beta_1+\cdots+\beta_n=\beta   \\
\omega\cdot\beta_i>0,\,  i=1,\ldots,n
\end{subarray}}P_{0}(X,\beta_0)\times Q_{1}(X,\beta_1;Z_1)\times \cdots \times Q_{1}(X,\beta_n;Z_n).  \end{align}
In order to conclude (\ref{Q_n:identity}), 
it remains to show that a pair of the form (\ref{splus}), 
where each $s_i$ is surjective, 
$F_0 \cong \oO_Z$ and $F_i \cong \oO_{C_i}$
as mentioned above, is a $Z_t$-stable pair 
if and only if 
we have 
\begin{align}\label{ineq:omega}
\omega \cdot \beta_i>\frac{1}{t}, \ 
1 \leqslant i \leqslant n.
\end{align}
Since $s$ is surjective, we only need to know when any $F'\subseteq F$ satisfies $\mu(F')<t$.
Any $F'\subseteq F$ is of form 
$$F'=\bigoplus_{i\in I} F_i', $$
for some subset $I \subseteq \{0, 1, \ldots, n\}$
such that each $F_i'$ is a non-zero subsheaf of $F_i$. 
For a fixed $I$, 
the maximal $\mu(F')$ is achieved when $F_i'=F_i$. 
By taking $I=\{i\}$ for $i\geqslant 1$, 
the $Z_t$-stability of the pair (\ref{splus}) implies the 
inequalities (\ref{ineq:omega}). 
Conversely suppose that (\ref{ineq:omega}) holds. 
Then for any $I\subseteq \{0, 1, \ldots,n\}$, 
by setting $I'=I \cap \{1, \ldots, n\}$, 
we have 
$$\mu(F') 
\leqslant \frac{|I'|}{\sum_{i\in I}\omega\cdot \beta_i} \leqslant \frac{|I'|}{\sum_{i\in I'}\omega\cdot \beta_i}<t, $$
if $I' \neq \emptyset$. 
If $I'=\emptyset$, then $I=\{0\}$
so $\mu(F') \leqslant 0<t$. 
Therefore the pair (\ref{splus}) is $Z_t$-stable, 
and the identity (\ref{Q_n:identity}) is justified.  


Finally, note that each $Q_{1}(X,\beta_i;Z_i)$ consists of 
finitely many rational curves that meet with $Z_i$, 
whose number is exactly $n_{0, \beta_i}(\gamma)$. 
By counting the number of points in $P_{0}(X,\beta_0)$ and 
$Q_{1}(X,\beta_i;Z_i)$'s, we obtain  
\begin{align*}
P^t_{n, \beta}(\gamma) :=\int_{[P^t_{n}(X,\beta)]^{\rm{vir}}}\tau(\gamma)^n=\int_{[Q^t_{n}(X,\beta;\gamma)]^{\mathrm{vir}}}1 
=\sum_{\begin{subarray}{c}\beta_0+\beta_1+\cdots+\beta_n=\beta  \\ \omega\cdot\beta_i>\frac{1}{t},\,  i=1,\ldots,n \end{subarray} }P_{0,\beta_0}\cdot \prod_{i=1}^n n_{0,\beta_i}(\gamma). \end{align*}
Therefore we obtain the formula in Conjecture~\ref{main conj}
from the above heuristic argument. 

\subsection{A master space argument}\label{master spa argu}
For each $t_0 \in \mathbb{R}_{>0}$, 
the formula (\ref{WCF:main})
implies the wall-crossing formula
\begin{align}\label{PTt:WCF2}
\lim_{t \to t_0+}\PT^t(X)(\exp(\gamma))=\prod_{\omega \cdot \beta=\frac{1}{t_0}} \exp(yq^{\beta})^{n_{0, \beta}(\gamma)} \cdot \lim_{t \to t_0-}\PT^t(X)(\exp(\gamma)). 
\end{align}
In this subsection, we give a heuristic
explanation of the above formula 
for a simple wall-crossing
via master spaces,
which are used by Mochizuki~\cite{Moc} in proving
wall-crossing formulae for Donaldson type invariants on algebraic surfaces.

For a fixed $(\beta, n)$, 
suppose that $t_0 \in \mathbb{R}$ is a wall
with respect to the $Z_t$-stability. 
We say that $t_0$ is a \textit{simple wall} if 
any point $p \in \overline{P}_n^{t_0}(X, \beta)$ corresponds to a $Z_{t_0}$-stable 
object, or $Z_{t_0}$-polystable object of the form 
\begin{align}\label{point:AB}
I=A \oplus B, \ A=(\oO_X \to F'), \ B=F''[-1],
\end{align}
where $A$ is $Z_{t_0}$-stable 
and $F'$ is $\mu$-stable with $\mu(F'')=t_0$. 
In other words in the description of polystable objects (\ref{polystable}), 
we have $k\leqslant 1$, and if $k=1$ then $\dim V_1=\mathbb{C}$. 
In this case, $\overline{P}_n^{t_0}(X, \beta)$ is stratified as
\begin{align}\label{stratification}
\overline{P}_n^{t_0}(X, \beta)
=P_n^{t_0}(X, \beta)\coprod \coprod_{\begin{subarray}{c}
(\beta', n')+(\beta'', n'')=(\beta, n) \\
\frac{n''}{\omega \cdot \beta''}=t_0 \end{subarray}}
(P_{n'}^{t_0}(X, \beta') \times M_{n''}(X, \beta'')),
\end{align}
and $M_{n''}(X, \beta'')$ consists of only $\mu$-stable one dimensional sheaves. 

Let us take a point 
$p \in \overline{P}_n^{t_0}(X, \beta)$
corresponding to the polystable object $I$ 
given in (\ref{point:AB}).
Below we give a description of the diagram (\ref{diagram:wall}) 
locally around $p$, following similar
arguments of~\cite{Todstack, Toddbir}. 
Let $\kappa$ be a Kuranishi map for 
the object $I$:
\begin{align*}
\kappa \colon \Ext^1(I, I) \to \Ext^2(I, I). 
\end{align*}
The above map describes
the stack $\pP_n^{t_0}(X, \beta)$ locally 
around $p \in \overline{P}_n^{t_0}(X, \beta)$. 
Namely by~\cite[Theorem~1.1]{Todstack}, the quotient stack 
\begin{align*}
[\kappa^{-1}(0)/\Aut(I)_0] \subset [\Ext^1(I, I)/\Aut(I)_0]
\end{align*}
is isomorphic to the stack 
$\pP_n^{t_0}(X, \beta)$ for the preimage of 
an analytic open neighbourhood 
of $p \in \overline{P}_n^{t_0}(X, \beta)$
under the map 
$\pP_n^{t_0}(X, \beta) \to \overline{P}_n^{t_0}(X, \beta)$. 
Here $\Aut(I)_0 \subset \Aut(I)$ is the traceless
part, given by
\begin{align*}
\mathbb{C}^{\ast}=\Aut(I)_0 \subset \Aut(I)=\Aut(A) \times \Aut(B), \ 
u \to (\id, u). 
\end{align*}
It acts on $\Ext^1(I, I)$ by the conjugation. 
Note that 
\begin{align*}
W\cneq \Ext^1(I, I)=
\Ext^1(A, A) \oplus \Ext^1(B, B) \oplus \Ext^1(A, B) \oplus \Ext^1(B, A),
\end{align*}
and the above $\mathbb{C}^{\ast}=\Aut(I)_0$-action 
on $W$ is of weight $(0, 0, 1, -1)$. 
Let $W^{\pm} \subset W$ be the open subsets defined by
\begin{align*}
W^{+}&=W \setminus \left(
\Ext^1(A, A) \oplus \Ext^1(B, B) \oplus \{0\} \oplus \Ext^1(B, A)\right), \\
W^-&=W \setminus \left(
\Ext^1(A, A) \oplus \Ext^1(B, B) \oplus \Ext^1(A, B) \oplus \{0\} \right). 
\end{align*}
They are GIT stable loci 
with respect to different linearizations. 
 We have the toric flip type diagram
\begin{align*}
\xymatrix{
W^+/\mathbb{C}^{\ast} \ar[rd] & & \ar[ld] W^-/\mathbb{C}^{\ast} \\
& W/\!\!/\mathbb{C}^{\ast}. &
}
\end{align*}
Then locally around $p$ 
(i.e. the preimage of an analytic open neighbourhood of
$p \in \overline{P}_n^{t_0}(X, \beta)$
under the maps $\pi^{\pm}$ in (\ref{diagram:wall})), 
the moduli spaces 
$P_n^{t_{\pm}}(X, \beta)$ are isomorphic to $M^{\pm}$ defined by (see the arguments of~\cite[Thm.~7.7]{Todstack}, 
\cite[Thm.~9.11]{Toddbir}):
\begin{align*}
M^{\pm} \cneq 
(\kappa^{-1}(0) \cap W^{\pm})/\mathbb{C}^{\ast} \subset W^{\pm}/\mathbb{C}^{\ast}. 
\end{align*}
Since we have 
\begin{align*}
\Ext^2(I, I)=\Ext^2(A, A) \oplus \Ext^2(B, B) \oplus \Ext^2(A, B) \oplus \Ext^2(B, A)
\end{align*}
and $\Ext^2(A, B)$, $\Ext^2(B, A)$ are dual to each other, 
we may take 
\begin{align*}
\Ext^2(I, I)^{\frac{1}{2}}=\Ext^2_{+}(A, A)\oplus \Ext^2_{+}(B, B) \oplus \Ext^2(A, B)
\end{align*}
as a `half obstruction space (this half obstruction space is a mixture of positive real subspaces and maximal isotropic subspaces. We use it as it is $\mathbb{C}^{\ast}$-equivariant and can be descended. Its Euler class is the same as the half Euler class of $\Ext^2(I,I)$). 
As $\Ext^2(I, I)^{\frac{1}{2}}  \times W \to W$ is a $\mathbb{C}^{\ast}$-equivariant 
vector bundle, 
it descends to a vector bundle 
$\mathrm{Obs}  \to [W/\mathbb{C}^{\ast}]$, which 
restricts to vector bundles 
\begin{align*}
\mathrm{Obs}^{ \pm} \to W^{\pm}/\mathbb{C}^{\ast}=:\overline{W}^{\pm}.
\end{align*}
Thus locally around $p$, the $\mathrm{DT}_4$ virtual classes on $P_{n}^{t_{\pm}}(X, \beta)$
pushed forward to $W^{\pm}/\mathbb{C}^{\ast}$ are
Euler classes of the above half obstruction bundles
\begin{align*}
[M^{\pm}]^{\rm{vir}}=e(\mathrm{Obs}^{\pm}). 
\end{align*}
We compare the above virtual classes 
using the master space. 
Let $\widetilde{W}$ be defined by
\begin{align*}
\widetilde{W}=(W \times \mathbb{C}^{\ast}) \sqcup (W^+ \times \{0\}) \sqcup
(W^- \times \{\infty\}) \subset W \times \mathbb{P}^1. 
\end{align*}
Let $T_i=\mathbb{C}^{\ast}$ for $i=1, 2$. 
Both of $T_1$ and $T_2$ acts on 
$\widetilde{W}$: for $t_i \in T_i$,  
\begin{align}\label{actions}
t_1 \cdot (x, [s_0, s_1])=(t_1 x, [t_1 s_0, s_1]), \quad 
t_2 \cdot (x, [s_0, s_1])=(x, [s_0, t_2 s_1]).
\end{align}
The above 
$T_1$-action on $\widetilde{W}$ is free, 
and the quotient space 
$Z \cneq \widetilde{W}/T_1$
is called the \textit{master space}. 
Since the two actions (\ref{actions}) commute, 
the $T_2$-action on $\widetilde{W}$ descends 
to a $T_2$-action on $Z$. 
Its fixed locus is 
\begin{align*}
Z^{T_2}=\overline{W}^{+} \sqcup \overline{W}^{-} \sqcup W^{\mathbb{C}^{\ast}}. 
\end{align*}
Similarly as above, 
the $(T_1 \times T_2)$-equivariant vector bundle 
$\Ext^2(I, I)^{\frac{1}{2}}  \times \widetilde{W} \to \widetilde{W}$
descends to the $T_2$-equivariant vector bundle 
$\widetilde{\mathrm{Obs}} \to Z$, whose 
Euler class is denoted by $[Z]^{\rm{vir}}$. 
The $T_2$-localization formula gives
the identity in the localized $T_2$-equivariant 
homology of $Z$
\begin{align}\label{localization}
[Z]^{\rm{vir}}=\frac{[M^+]^{\rm{vir}}}{e(N_{\overline{W}^+/Z})}
+\frac{[M^-]^{\rm{vir}}}{e(N_{\overline{W}^-/Z})}
+\frac{[W^{\mathbb{C}^{\ast}}]^{\rm{vir}}}{e(N_{W^{\mathbb{C}^{\ast}}/Z}^{\rm{mov}})}. 
\end{align}
Note that we have 
\begin{align*}
W^{\mathbb{C}^{\ast}}=\Ext^1(A, A) \oplus \Ext^1(B, B) \oplus \{0\} \oplus \{0\},
\end{align*}
and $[W^{\mathbb{C}^{\ast}}]^{\rm{vir}}=[P^{t_0}_{n'}(X, \beta')]^{\rm{vir}} \times [M_{n''}(X, \beta'')]^{\rm{vir}}$, 
viewed locally around the point $(A, B)$. 
 
The above arguments are local around $p$, so
$M^{\pm}$, $Z$ are non-compact.  However suppose that we have some
globalization of the above argument (e.g. the construction of master space and its virtual class, for the global compact moduli spaces $P_n^{t_{\pm}}(X, \beta)$), and 
pretend that $M^{\pm}$, $Z$ are compact. 
Then the integration of the left hand side of (\ref{localization}), after some insertions, 
is independent of the equivariant parameter $t_2$. 

Note that the real virtual dimension of $M^{\pm}$ is $2n$, while 
that of $W^{\mathbb{C}^{\ast}}$ is $2n'+2$. 
Therefore in order to obtain non-trivial contribution to the 
wall-crossing, by taking the insertions $\tau(\gamma)^n$ in (\ref{localization}) 
and the residue at $t_2=0$, we must have 
$2n=2n'+2$, i.e. $(n', n'')=(n-1, 1)$. 
Then 
\begin{align*}
N^{\rm{mov}}_{W^{\mathbb{C}^{\ast}}/Z}=\Ext^1(A, B) +\Ext^1(B, A)-\Ext^2(A, B)
\end{align*}
has rank $-\chi(A, B)=n''=1$, with $T_2$-weight $1$, $-1$, $1$ 
respectively. Therefore the contribution of the 
 denominator of the last term of (\ref{localization}) 
 to the residue at $t_2=0$ is $-1$. 
By taking the insertion $\tau(\gamma)^n$ and 
the residue at $t_2=0$ of (\ref{localization}), we 
obtain
\begin{align*}
&\int_{[P_n^{t_{+}}(X, \beta)]^{\rm{vir}}}\tau(\gamma)^n
-\int_{[P_n^{t_{-}}(X, \beta)]^{\rm{vir}}}\tau(\gamma)^n \\
&=\sum_{\begin{subarray}{c}\beta'+\beta''=\beta \\
\frac{1}{\omega \cdot \beta''}=t_0
\end{subarray}}
\int_{[P^{t_0}_{n-1}(X, \beta')]^{\rm{vir}} \times [M_{1}(X, \beta'')]^{\rm{vir}}}
(\tau(\gamma)\boxtimes 1 +1 \boxtimes \tau(\gamma))^n. 
\end{align*}
By expanding the RHS and assuming Conjecture~\ref{g=0 one dim sheaf conj}, we obtain the wall-crossing formula
\begin{align}\label{WCF:mspace}
P^{t_{+}}_{n, \beta}(\gamma)-P^{t_-}_{n, \beta}(\gamma)
=\sum_{\begin{subarray}{c}
\beta'+\beta''=\beta  \\
\frac{1}{\omega \cdot \beta''}=t_0
\end{subarray}} n \cdot P^{t_0}_{n-1, \beta'}(\gamma) \cdot 
n_{0, \beta''}(\gamma). 
\end{align}
Indeed this wall-crossing formula is compatible with our main conjecture. 
\begin{prop}\label{lem:wcf:equiv}
Suppose that $t_0 \in \mathbb{R}_{>0}$ is a simple wall 
with respect to $(\beta, n)$. 
Then under
Conjecture~\ref{g=0 one dim sheaf conj},
the coefficient of $y^n q^{\beta}$ in the 
formula (\ref{PTt:WCF2}) is equivalent to (\ref{WCF:mspace}). 
\end{prop}
\begin{proof}
By expanding the formula (\ref{PTt:WCF2}), 
the identity at the coefficient of $y^n q^{\beta}$ is 
\begin{align}\label{id:wcf}
P^{t_{+}}_{n, \beta}(\gamma)-P^{t_-}_{n, \beta}(\gamma) 
=\sum_{k=1}^n \frac{n!}{k!(n-k)!}
\sum_{\begin{subarray}{c}\beta'+\beta''=\beta \\
\beta_1''+\cdots +\beta_k''=\beta'', \ \omega \cdot \beta_i''=\frac{1}{t_0}
\end{subarray}} P_{n-k, \beta'}^{t_-}(\gamma) \cdot \prod_{i=1}^{k} n_{0, \beta_i''}(\gamma). 
\end{align}
If $t_0 \in \mathbb{R}_{>0}$ is a simple wall, 
then a term 
$P_{n-k, \beta'}^{t_-}(\gamma) \cdot \prod_{i=1}^{k} n_{0, \beta_i''}(\gamma)$
is non-zero only if $k=1$. 
Otherwise, we have 
\begin{align*}
P_{n-k}^{t_-}(X, \beta') \times \prod_{i=1}^k M_{1}(X, \beta_i'') 
\neq \emptyset. 
\end{align*}
For a point $(I', F_1'', \cdots, F_k'')$
in the above product, 
$I'$ is $Z_{t_0}$-semistable.
Therefore by 
denoting $\mathrm{gr}(I')$ the associated graded 
with respect to Jordan-H\"{o}lder filtration of the 
$Z_{t_0}$-stability, 
we have 
\begin{align*}
\mathrm{gr}(I') \oplus F_1''[-1] \oplus \cdots \oplus F_k''[-1]
\in \overline{P}_{n}^{t_0}(X, \beta).
\end{align*}
The above $Z_{t_0}$-polystable object
is of the form (\ref{point:AB}) only if $k=1$ and $I'$ is $Z_{t_0}$-stable. 
By the same reason, we have 
$P_{n-1}^{t_-}(X, \beta')=P_{n-1}^{t_0}(X, \beta')=\overline{P}_{n-1}^{t_0}(X, \beta')$, 
hence $P_{n-1, \beta'}^{t_-}(\gamma)=P_{n-1, \beta'}^{t_0}(\gamma)$. 
Therefore the identity (\ref{id:wcf}) is nothing but the formula (\ref{WCF:mspace}) if $t_0$ is a simple wall. 
\end{proof}

If $(\beta, n)$ satisfies the condition (\ref{ineq:betan}), 
there is no wall-crossing 
for $t>\frac{n}{\omega \cdot \beta}$, and 
PT and JS pairs are the same. 
Indeed in this case,
our main conjecture is compatible with
our previous PT/GV conjecture. 
\begin{prop}\label{lem:compatible}
Suppose that $(\beta, n)$ satisfies (\ref{ineq:betan}). 
Then we have 
\begin{align}\label{id:PT}
P_{n, \beta}(\gamma)=P_{n, \beta}^t(\gamma)=P_{n, \beta}^{\rm{JS}}(\gamma), \quad
t>\frac{n}{\omega \cdot \beta}, 
\end{align}
for certain choice of orientation.
And the first identity of Conjecture~\ref{PT/GV conj}, 
Conjecture~\ref{main conj} for $t>\frac{n}{\omega \cdot \beta}$, 
and Conjecture~\ref{JS/GV conj} are equivalent. 
\end{prop}
\begin{proof}
The identities (\ref{id:PT})
follows from Proposition~\ref{prop:nowall}. 
In order to show the compatibilities of conjectures, 
by the argument of Proposition~\ref{lem:wcf:equiv}, 
it is enough to show that for any $t_0>\frac{n}{\omega \cdot \beta}$
the right hand side of (\ref{id:wcf}) vanishes. 
Suppose that it is non-zero, and take 
$(\beta_0, \beta_1, \ldots, \beta_k)$ as in
the right hand side of (\ref{id:wcf}). 
Then
from 
the inequalities 
\begin{align*}
t_0=\frac{1}{\omega \cdot \beta_i}>\frac{n}{\omega \cdot \beta}, \quad
1\leqslant i \leqslant k,
\end{align*}
together with the condition (\ref{ineq:betan}), 
we obtain the inequality
\begin{align}\label{ineq:betak}
\frac{n-k}{\omega \cdot \beta_0}<\frac{n}{\omega \cdot \beta}
\leqslant \frac{n(\beta')}{\omega \cdot \beta'}, \quad 0<\beta' \leqslant \beta_0< \beta. 
\end{align}
Therefore the condition (\ref{ineq:betan}) 
is satisfied for $(\beta_0, n-k)$,
so $P_{n-k}^{t_-}(X, \beta_0)=P_{n-k}(X, \beta_0) \neq \emptyset$. 
Thus $n(\beta_0) \leqslant n-k$ by the definition of $n(\beta_0)$, 
which contradicts to (\ref{ineq:betak}) for $\beta'=\beta_0$. 
\end{proof}

If the condition (\ref{ineq:betan}) is not satisfied, 
we have wall-crossing phenomena as observed in Example~\ref{exam:locP2}.
In this example, there is nontrivial
wall-crossing of our invariants. 

\begin{exam}\label{exam:wall-crossing}
In the situation of Example~\ref{exam:locP2}, 
the $t=1$ is a simple wall, and 
the stratification (\ref{stratification}) is given by 
\begin{align*}
\overline{P}_1^{t=1}(X, 4[l])
=P_1^{t=1}(X, 4[l])
\coprod \big(P_0(X, 3[l]) \times M_1(X, [l])\big). 
\end{align*}
Let $[\mathrm{pt}] \in H^4(X, \mathbb{Z})=H^4(\mathbb{P}^2, \mathbb{Z})$
be the point class. From the formula (\ref{WCF:mspace}), 
we should have 
the identity
\begin{align*}
P_{1, 4[l]}([\mathrm{pt}])-P_{1, 4[l]}^{\rm{JS}}([\mathrm{pt}])=
P_{0, 3[l]} \cdot n_{0, [l]}([\mathrm{pt}]). 
\end{align*}
Indeed in this case, we can compute the invariants and conclude (ref. Proposition \ref{local P2 low deg class}):
\begin{align*}
P_{1, 4[l]}^t([\mathrm{pt}])=\left\{
\begin{array}{cc}
n_{0,4}([\mathrm{pt}]) +P_{0,3}\cdot n_{0,1}([\mathrm{pt}])=3,     &      \mathrm{if} \quad  t> 1 \\
& \\
n_{0,4}([\mathrm{pt}])=2,     &      \mathrm{if}\quad \frac{1}{4}<t<1.
\end{array} \right. 
\end{align*}
\end{exam}

\subsection{A virtual pushforward formula}
The formula for $n=1$ in Conjecture~\ref{JS/GV conj} is
\begin{align}\label{id:JS:GV}
P^{\mathrm{JS}}_{1,\beta}(\gamma)=n_{0,\beta}(\gamma),
\end{align}
which could be understood from
both wall-crossing and 
virtual pushforward formulae. 

In terms of wall-crossing, 
let us take $t_0=\frac{1}{\omega \cdot \beta}$. 
Then $t=t_0$ is a simple wall with respect to 
$(\beta, 1)$, so the formula (\ref{WCF:mspace}) gives 
\begin{align*}
P_{1, \beta}^{t_0+}(\gamma)-P_{1, \beta}^{t_0-}(\gamma)=
\sum_{\begin{subarray}{c} \beta'+\beta''=\beta, \\
\omega \cdot \beta''=\omega \cdot \beta
\end{subarray}}1 \cdot P_{0, \beta'} \cdot n_{0, \beta''}(\gamma). 
\end{align*}
Note that $\omega \cdot \beta'=0$ implies that 
$\beta'=0$ as $\beta'$ is an effective class or zero. 
Since $P_{1, \beta}^{t_0+}(\gamma)=P_{1, \beta}^{\rm{JS}}(\gamma)$, 
$P_{1, \beta}^{t_0-}(\gamma)=0$, 
and $P_{0, 0}=1$, 
we obtain the identity (\ref{id:JS:GV}). 

In terms of virtual push-forward formula, 
we consider the morphism in (\ref{map:JStoM}):
\begin{align*}
P^{\mathrm{JS}}_1(X, \beta)\to M_1(X, \beta), \quad (\oO_X\to F)\mapsto F.
\end{align*}
We expect that the following virtual pushforward formula 
\begin{align}\label{vir pushforward}f_*([P^{\mathrm{JS}}_1(X, \beta)]^{\mathrm{vir}})=[M_1(X, \beta)]^{\mathrm{vir}}\in H_2(M_1(X, \beta),\mathbb{Z}) \end{align}
holds for certain choice of orientation. Capping with insertions gives  
$$P^{\mathrm{JS}}_{1,\beta}(\gamma)=\int_{[M_1(X, \beta)]^{\mathrm{vir}}}\tau(\gamma), $$
where $\tau$ is the primary insertion (\ref{insertion for one dim sheaves}) for one dimensional stable sheaves. 
Assuming this, the equality $P^{\mathrm{JS}}_{1,\beta}(\gamma)=n_{0,\beta}(\gamma)$ is reduced to the `Katz/GV' conjecture mentioned in Conjecture \ref{g=0 one dim sheaf conj}.

Generally speaking, the virtual class in $\DT_4$ theory is difficult to work with. In the special case when the moduli space is a $(-2)$-shifted cotangent bundle of some derived smooth scheme (as reviewed in Section \ref{review DT4}), the virtual class can be described algebraically. The virtual pushforward formula (\ref{vir pushforward}) can be rigorously proved, due to the work of Manolache \cite{Mano}.
We review her formula in the following setting:
\begin{thm}\label{Manolache formula}\emph{(Manolache \cite[Thm.]{Mano})}
Given a proper morphism $f: P\to M$ between Deligne-Mumford stacks which possess perfect obstruction theories $E^{\bullet}_{P}$ and $E^{\bullet}_{M}$. If $f$ has a perfect relative obstruction theory compatible with $E^{\bullet}_{P}$ and $E^{\bullet}_{M}$ and $M$ is connected.
Assume the virtual dimension of the relative obstruction theory is zero, then 
$$f_*[P]^{\mathrm{vir}}=c\cdot [M]^{\mathrm{vir}}, $$ 
where $c\in\mathbb{Q}$ is the degree of the relative perfect obstruction theory.
\end{thm}
In Section \ref{sect on verification of cpt conj}, 
we will apply it to several examples and prove Conjecture \ref{JS/GV conj} (with $n=1$) in those cases.

\section{Examples of JS/GV formula}\label{sect on verification of cpt conj}\label{sec:JS/GV}
Evidence of Conjecture \ref{main conj} in the Pandharipande-Thomas chamber is given in \cite{CMT2}.
In this section, we give further verifications of Conjecture \ref{main conj} mainly concentrated in the Joyce-Song chamber as stated in the form of Conjecture \ref{JS/GV conj}.

\subsection{Irreducible curve class}
When the curve class is irreducible, there is no difference between JS and PT chamber. 
We refer to our previous work \cite[Prop.~1.4,~1.8,~Thm.~1.5,~1.7]{CMT2} for many checks of Conjecture \ref{main conj} in such setting.
\begin{prop}\label{irr class}
Let $(X,\omega)$ be a Calabi-Yau 4-fold and $\beta\in H_2(X,\mathbb{Z})$ be an irreducible curve class. 
Then we have 
\begin{align*}
P^t_{n, \beta}(\gamma)=P_{n, \beta}(\gamma), \quad  
t>\frac{n}{\omega \cdot \beta},
\end{align*}
for certain choice of orientation.
\end{prop}
\begin{proof}
Since $\beta$ is irreducible, the condition (\ref{ineq:betan})
is automatically satisfied.
Therefore we have $P_n^t(X, \beta)=P_n(X, \beta)$
for $t>\frac{n}{\omega \cdot \beta}$ 
by Proposition~\ref{prop:nowall}.
%
\end{proof}

\subsection{Degree two curve class}
Let $X\subseteq \mathbb{P}^{5}$ be a smooth sextic 4-fold with hyperplane class $\omega$.
By Lefschetz hyperplane theorem, 
$H_2(X,\mathbb{Z})\cong H_2(\mathbb{P}^{5},\mathbb{Z})=\mathbb{Z}[l]$, where 
$l$ is the class of a line. 
\begin{prop}\label{deg two class}
For $n=0,1,2$ and degree two class $\beta=2[l]\in H_2(X,\mathbb{Z})$, we have 
\begin{align*}
P^t_{n, \beta}(\gamma)=P_{n, \beta}(\gamma), \quad
t>\frac{n}{2},
\end{align*}
for certain choice of orientation. 
Furthermore, Conjecture \ref{main conj} holds for $\beta=2[l]$ and $n=0,1$.
\end{prop}
\begin{proof}
The condition (\ref{ineq:betan}) is satisfied
for $(2[l], n)$ with $n\leqslant 2$, 
since 
$n([l])=1$. 
%
%
%
Therefore we have 
$P_n^t(X, 2[l])=P_n(X, 2[l])$ for $n\leqslant 2$. 
Under the isomorphism of moduli spaces, virtual classes are identified and invariants are the same by
choosing same orientations and insertions . 
When $n=0,1$ and $t>\frac{1}{2}$, Conjecture \ref{main conj} then reduces to \cite[Prop.~3.1,~3.2]{CMT2} (see also \cite{Caoconic}). The case $n=0,1$ and $t< \frac{1}{2}$ is obvious.
\end{proof}

\subsection{Elliptic fibration}
For $Y=\mathbb{P}^3$, we take general elements
\begin{align*}
u \in H^0(Y, \oO_Y(-4K_Y)), \
v \in H^0(Y, \oO_Y(-6K_Y)).
\end{align*}
We define $X$ to be the hypersurface
\begin{align*}
X =\{zy^2=x^3 +uxz^2+vz^3\}
\subset \mathbb{P}(\oO_Y(-2K_Y) \oplus \oO_Y(-3K_Y) \oplus \oO_Y). 
\end{align*}
Here $[x:y:z]$ is the homogeneous coordinate of the 
projective bundle over $Y$ in the right hand side. 
Then $X$ is a CY 4-fold, and 
the projection to $Y$ 
gives an elliptic fibration
\begin{align}\label{elliptic fib}
\pi \colon X \to Y. 
\end{align}
A general fiber of
$\pi$ is a smooth elliptic curve, and any singular
fiber is either a nodal or cuspidal plane curve.
Moreover, $\pi$ admits a section $\iota$ whose image
corresond to fiber point $[0: 1: 0]$. 

Let $h$ be a hyperplane in $Y$ and $f:=\pi^{-1}(p)$ for a general point $p\in\mathbb{P}^3$, set
$$B=\pi^*h, \quad E=\iota(Y)\in H_6(X,\mathbb{Z}). $$
We consider multiple fiber classes $r[f]$ ($r\geqslant1$) below.
\begin{prop}\label{elliptic fib}
For any $t>0$ and certain choice of orientation, we have 
\begin{align*}
P_{0, r[f]}^t=P_{0, r[f]}.
\end{align*}
Furthermore, Conjecture \ref{main conj} holds for $\beta=r[f]$ \emph{($r\geqslant 1$)} and $n=0$.
\end{prop}
\begin{proof}
The condition (\ref{ineq:betan}) is
satisfied since $n(r'[f])=0$ for any $r'>0$. 
Therefore 
$P_0^t(X, r[f])=P_0(X, r[f])$ and 
their virtual classes are identified. 
Then Conjecture \ref{main conj} follows from 
the statement for PT stable pairs~\cite[Prop.~3.6]{CMT2}.
\end{proof}
\begin{prop} 
For certain choice of orientation, we have 
\begin{align*}
P^{\mathrm{JS}}_{1, r[f]}(\gamma)=\int_{[M_1(X, r[f])]^{\rm{vir}}}\tau(\gamma).
\end{align*}
Moreover, Conjecture \ref{JS/GV conj} holds for $\beta=r[f]$ \emph{($r\geqslant 1$)}, $n=1$ and $\gamma=B\cdot E$ or $B^2$.
\end{prop}
\begin{proof}
Let $M_1(X, r[f])$ be the moduli space of one dimensional stable sheaves $F$ with $[F]=r[f]$ and $\chi(F)=1$. An element $[F]\in M_1(X, r[f])$ is 
scheme theoretically supported on a fiber $\pi^{-1}(p)$ of $\pi$ (ref.~\cite[Lem.~2.2]{CMT1}). Write $F=i_*\mathcal{F}$ for the inclusion $i:\pi^{-1}(p)\to X$, then 
$$H^1(X,F)\cong H^1(\pi^{-1}(p),\mathcal{F})\cong \Hom(\mathcal{F},\oO_{\pi^{-1}(p)})^{\vee}=0, $$
since the slope of $\mathcal{F}$ is bigger than the slope of $\oO_{\pi^{-1}(p)}$.

The morphism 
\begin{align*}
f: P^{\mathrm{JS}}_1(X, r[f])\to M_1(X, r[f]), \quad (\oO_X\to F)\mapsto F
\end{align*}
is then an isomorphism as the fiber is $\mathbb{P}(H^0(X,F))$ and $h^0(X,F)=h^1(X,F)+1=1$. 
Conjecture \ref{JS/GV conj} in this case reduces to Conjecture \ref{g=0 one dim sheaf conj}, which has been verified in \cite[Prop.~2.3]{CMT1}. 
\end{proof}

\subsection{Product of CY 3-fold and elliptic curve}
Let $Y$ be a smooth projective Calabi-Yau 3-fold and $E$ is an elliptic curve.
Argument of Proposition \ref{elliptic fib} leads straightforward to:
\begin{prop}\label{prod elliptic fib}
Let $X=Y\times E$ be as above.
Then for any $t>0$, $r\geqslant1$, we have 
\begin{align*}
P_{0, r[E]}^t=P_{0, r[E]}, 
\end{align*}
for certain choice of orientation. 
Hence Conjecture \ref{main conj} holds for $\beta=r[E]$ \emph{($r\geqslant 1$)} and $n=0$.
\end{prop}

Next, we discuss the $n=1$ case. For any smooth projective variety $Y$, we may define the moduli space $P^{\mathrm{JS}}_n(Y,\beta)$ of Joyce-Song stable pairs on $Y$ by
Definition~\ref{defi:PTJSpair} (ii). 
We have the following
deformation-obstruction theory
(called pair deformation-obstruction theory):
\begin{align}\notag
\dR \hH om_{\pi_P}(\mathbb{I}, \mathbb{F})^{\vee}
\to \tau_{\geqslant -1}\mathbb{L}_{P^{\rm{JS}}_n(X, \beta)}.
\end{align}
Here $\mathbb{I}=(\oO_{X \times P_n^{\rm{JS}}(X, \beta)} \to \mathbb{F})$
is the universal pair, 
and $\pi_P \colon X \times P_n^{\rm{JS}}(X, \beta) \to P_n^{\rm{JS}}(X, \beta)$
is the projection. 

\begin{lem}\label{lem on pair moduli on CY3}
Let $(Y,\omega)$ be a smooth projective Calabi-Yau 3-fold and $\beta\in H_2(Y,\mathbb{Z})$.
Then the
truncated  pair deformation-obstruction theory 
\begin{align}\notag
\tau_{\geqslant -1}\left(
\dR \hH om_{\pi}(\mathbb{I}, \mathbb{F})^{\vee}\right)
\to \tau_{\geqslant -1}\mathbb{L}_{P^{\rm{JS}}_1(X, \beta)}
\end{align}
of $P^{\mathrm{JS}}_1(Y,\beta)$ is perfect in the sense of \cite{BF, LT}. Hence there exists
an algebraic virtual class
of virtual dimension zero
\begin{equation}[P^{\mathrm{JS}}_1(Y,\beta)]_{\mathrm{pair}}^{\mathrm{vir}}\in A_{0}(P^{\mathrm{JS}}_1(Y,\beta),\mathbb{Z}). \nonumber \end{equation}
\end{lem}
\begin{proof}
For any JS stable pair $I_Y=(\oO_Y\stackrel{s}{\to}  F)\in P^{\mathrm{JS}}_1(Y,\beta)$, 
$F$ is stable as it is semistable with $\chi(F)=1$, 
Therefore we have 
\begin{equation}\Ext^3_Y(F,F)\cong \Hom_Y(F,F)^{\vee}\cong\mathbb{C}. \nonumber \end{equation}
Applying $\RHom_Y(-,F)$ to $I_Y\to \oO_Y \to F$, we obtain a distinguished triangle 
\begin{equation}\label{dist triangle CY}\RHom_Y(F,F)\to \RHom_Y(\oO_Y,F) \to \RHom_Y(I_Y,F), \end{equation}
whose cohomology gives an exact sequence 
\begin{equation}0=H^2(Y,F) \to \Ext^2_Y(I_Y,F)\to \Ext^3_Y(F,F) \to 0 \to \Ext^3_Y(I_Y,F) \to 0.  \nonumber \end{equation}
Hence $\Ext^i_Y(I_Y,F)=0$ for $i\geqslant 3$ and $\Ext^2_Y(I_Y,F)\cong\Ext^3_Y(F,F)\cong\mathbb{C}$.
By truncating $\Ext^2_Y(I_Y,F)=\mathbb{C}$, the pair deformation-obstruction theory is perfect. 
\end{proof}

\begin{thm}\label{prod of cy3}
Let $X=Y\times E$ be as above. Assume Conjecture \ref{g=0 one dim sheaf conj} holds for $\beta\in H_2(Y)\subseteq H_2(X)$. 
Then for any $\gamma\in H^4(X)$, we have
$$ 
P^{\mathrm{JS}}_{1,\beta}(\gamma)=n_{0,\beta}(\gamma),
$$
for certain choice of orientation, i.e. Conjecture \ref{JS/GV conj} holds for $\beta\in H_2(Y)\subseteq H_2(X)$ and $n=1$.
\end{thm}
\begin{proof}
We take a JS stable pair 
$(\oO_X\stackrel{s}{\to}  F)\in P_{1}^{\mathrm{JS}}(X,\beta)$. 
Then 
$F$ is stable and scheme theoretically supported on 
$i_p:Y\times \{p\}\hookrightarrow X$ for some $p\in E$ (ref. \cite[Lem.~2.2]{CMT1}). Similar to \cite[Prop.~3.11]{CMT2}, there exists an isomorphism 
$$P^{\mathrm{JS}}_{1}(X,\beta)\cong P^{\mathrm{JS}}_{1}(Y,\beta)\times E,$$ 
$$(\oO_X \stackrel{s}{\to} {i_p}_*\mathcal{E})\mapsto \big((\oO_Y=i_p^*\oO_X \stackrel{s}{\to} \mathcal{E}),\, p\big),$$
under which virtual classes satisfy  
$$[P^{\mathrm{JS}}_{1}(X,\beta)]^{\mathrm{vir}}= [P^{\mathrm{JS}}_{1}(Y,\beta)]_{\mathrm{pair}}^{\mathrm{vir}}\times [E], $$ 
where $[P^{\mathrm{JS}}_{1}(Y,\beta)]_{\mathrm{pair}}^{\mathrm{vir}}$ denotes the virtual class defined in Lemma \ref{lem on pair moduli on CY3}.

There is a forgetful morphism 
$$f:P^{\mathrm{JS}}_{1}(Y,\beta)\to M_1(Y,\beta), \quad  (s\colon\oO_Y \to \mathcal{E}) \mapsto \mathcal{E}, $$
to the
 moduli space $M_1(Y,\beta)$ of one dimensional stable sheaves $F$ on $Y$ with $[F]=\beta$ and $\chi(F)=1$.  
The fiber of $f$ over $F$ is $\mathbb{P}(H^0(Y,F))$.

Let $\mathbb{F} \to M_{1}(Y,\beta) \times Y$ be the universal sheaf. 
Then the above map identifies $P^{\mathrm{JS}}_1(Y, \beta)$ with 
$\mathbb{P}(\pi_{M\ast}\mathbb{F})$
where $\pi_M \colon M_{1}(Y, \beta) \times Y \to M_{1}(Y, \beta)$
is the projection. And the universal stable pair is given by 
\begin{align*}
\mathbb{I}=
(\mathcal{O}_{Y\times P^{\mathrm{JS}}_1(Y, \beta)} 
\stackrel{s}{\rightarrow} \mathbb{F}^{\dag}), \quad
\mathbb{F}^{\dag}:= (\id_Y \times f)^{\ast}\mathbb{F}\otimes \oO(1),
\end{align*} 
where $\oO(1)$ is the tautological line bundle on 
$\mathbb{P}(\pi_{M\ast}\mathbb{F})$
and $s$ is the tautological map.  

As in \cite[Prop.~3.10]{CMT2}, we can apply Theorem \ref{Manolache formula} and obtain
$$f_*[P^{\mathrm{JS}}_1(Y,\beta)]_{\mathrm{pair}}^{\mathrm{vir}}=c\,[M_{1}(Y,\beta)]^{\mathrm{vir}},$$
where the coefficient $c$ can be fixed by restricting the relative perfect obstruction theory to a fiber of $f$ (ref.~\cite[pp.~2022~(18)]{Mano}).
The obstruction bundle over fiber $\mathbb{P}(H^0(Y,F))$ is $H^1(Y,F)\otimes \oO(1)$ whose rank is $(h^0(Y,F)-1)$ as $\chi(F)=1$. Thus
$$c=\int_{\mathbb{P}(H^0(Y,F))}e(H^1(Y,F)\otimes \oO(1))=1. $$
To sum up, we obtain 
$$(f\times \mathrm{id}_E)_*[P^{\mathrm{JS}}_1(X,\beta)]^{\mathrm{vir}}= [M_{1}(X,\beta)]^{\mathrm{vir}},$$
where we use $M_{1}(X,\beta)\cong M_{1}(Y,\beta)\times E$ and $[M_{1}(X,\beta)]^{\mathrm{vir}}=[M_{1}(Y,\beta)]^{\mathrm{vir}}\times [E]$ (ref.~\cite[Lem.~2.6]{CMT1}).

As the insertion $\tau$ (\ref{pri insertion}) depends only on $\mathbb{F}$ (not the section), we have 
$$P^{\mathrm{JS}}_{1,\beta}(\gamma)=\int_{[P^{\mathrm{JS}}_{1}(X,\beta)]^{\rm{vir}}}\tau(\gamma)=\int_{[M_{1}(X,\beta)]^{\rm{vir}}}\tau(\gamma), $$
where we use same notation $\tau$ to denote the insertion for $M_{1}(X,\beta)$. By Conjecture \ref{g=0 one dim sheaf conj},
we have 
$$\int_{[M_{1}(X,\beta)]^{\rm{vir}}}\tau(\gamma)=n_{0,\beta}(\gamma), $$
for certain choice of orientation. Combining the two equalities, we are done.
\end{proof}
Combining with previous verifications of Conjecture \ref{g=0 one dim sheaf conj} (ref.~\cite[Thm.~2.8]{CMT1}), we have:
\begin{cor}\label{prim check}
Let $Y$ be a complete intersection CY 3-fold in a product of projective spaces, and $X=Y\times E$ for an elliptic curve $E$.
Then Conjecture \ref{JS/GV conj} holds for primitive curve class $\beta\in H_2(Y)\subseteq H_2(X)$ and $n=1$. 
\end{cor}

\subsection{Local Fano 3-folds}
Let $Y$ be a smooth Fano 3-fold and consider the total space $X=K_Y$ of canonical bundle of $Y$. Take $\omega$ to be the 
pullback ample line bundle from an ample line bundle on $Y$. The moduli space 
$P^{\mathrm{JS}}_{1}(X,\beta)$ of Joyce-Song stable pairs is proper, since for any 
pair $(\oO_X\to F)$, $F$ is stable and scheme theoretically supported on $Y$. So we can still study 
Conjecture \ref{JS/GV conj} on such non-compact Calabi-Yau 4-folds.
Similar to Theorem \ref{prod of cy3}, we have 
\begin{thm}\label{thm:locFano}
Let $X=K_Y$ be as above. Assume Conjecture \ref{g=0 one dim sheaf conj} holds for $\beta\in H_2(X)$. 
Then for any $\gamma\in H^4(X)$, we have
$$ 
P^{\mathrm{JS}}_{1,\beta}(\gamma)=n_{0,\beta}(\gamma),
$$
for certain choice of orientation, i.e. Conjecture \ref{JS/GV conj} holds for $\beta\in H_2(X)$ and $n=1$.
\end{thm}
Combining with the previous verification of Conjecture 
\ref{g=0 one dim sheaf conj} (ref.~\cite[Prop.~0.3,~0.4,~Thm.~0.6]{CaoFano}), we have:
\begin{cor}\label{loc fano check}
Conjecture \ref{JS/GV conj} holds for $\beta\in H_2(K_Y)$ and $n=1$ in the following cases: 
\begin{itemize}
\item $Y\subseteq \mathbb{P}^{4}$ is a smooth Fano hypersurface and $\beta$ is irreducible. 
\item $Y=S\times \mathbb{P}^1$ and $\beta=n\,[\mathbb{P}^1]$ $(n\geqslant1)$, where $S$ is a del Pezzo surface.
\item $Y=S\times \mathbb{P}^1$, $\beta\in H_2(S)\subseteq H_2(Y)$ and $\gamma\in H^2(S)\otimes H^2(\mathbb{P}^1)\subset H^4(Y)$, where $S$ is a toric del Pezzo surface.
\end{itemize}
\end{cor}

\subsection{Local surfaces}
In this section, we consider two local surfaces:
\begin{align}\label{two local surf}X= \mathrm{Tot}_{\mathbb{P}^2}(\oO(-1)\oplus \oO(-2)), \quad   \mathrm{Tot}_{\mathbb{P}^1\times \mathbb{P}^1}(\oO(-1,-1)\oplus \oO(-1,-1)).  \end{align}
For the first one, we choose $\omega$ to be the pullback of $\oO_{\mathbb{P}^2}(1)\to \mathbb{P}^2$, and for the second one, we choose $\omega$ to be the pullback of $\oO_{\mathbb{P}^1}(l_1)\boxtimes \oO_{\mathbb{P}^1}(l_2)  \to \mathbb{P}^1\times\mathbb{P}^1$, where $l_1,l_2>0$. 

Although $X$ is non-compact, the moduli space of $Z_t$-semistable pairs on $X$ is proper, so we can study 
Conjecture \ref{main conj} on $X$.
\begin{lem}
For any $t>0$, the moduli space $\overline{P}^t_{n}(X,\beta)$ of $Z_t$-semistable pairs on $X$ is proper.
\end{lem} 
\begin{proof}
Let $\overline{X}$ be the compactification of $X$ by adding section at infinity, i.e. 
\begin{align*}\overline{X}:=\mathbb{P}(L_1\oplus L_2\oplus \oO_S). \end{align*}
Then $\overline{P}^t_n(\overline{X}, \beta)$ is proper as $\overline{X}$ is so. For a $Z_t$-semistable pair 
$(\oO_{\overline{X}}\stackrel{s}{\to} F)\in \pP^t_n(\overline{X}, \beta)$, 
$F$ is set theoretically suppoted on the zero section by the 
the negativity of normal bundle of $S\subseteq X$ (ref.~\cite[Prop.~3.1]{CMT1}). 
Therefore $\overline{P}_n^t(\overline{X}, \beta)$ is isomorphic to $\overline{P}_n^t(X, \beta)$ and $\overline{P}_n^t(X, \beta)$ is proper. 
\end{proof}
\begin{prop}\label{local surface identify virtual class}
Assume $(n,\omega\cdot\beta)=1$, then we have an isomorphism 
$$P^{\mathrm{JS}}_{n}(X,\beta)\cong P^{\mathrm{JS}}_{n}(S,\beta), $$
where $S=\mathbb{P}^2$ or $\mathbb{P}^1\times \mathbb{P}^1$ is the zero section of $X$ (\ref{two local surf}). 

Moreover, the virtual class satisfies 
$$[P^{\mathrm{JS}}_{n}(X,\beta)]^{\mathrm{vir}}=
[P^{\mathrm{JS}}_{n}(S,\beta)]^{\mathrm{vir}}\cdot e\Big(-\dR\hH om_{\pi_{P}}(\mathbb{F}, \mathbb{F} \boxtimes L_1)\Big),$$
for certain choice of orientation. Here $\mathbb{I}_S=(\oO_{S \times P^{\mathrm{JS}}_n(S, \beta)} \to \mathbb{F})\in D^b\big(S \times P^{\mathrm{JS}}_n(S, \beta)\big)$
is the universal stable pair and 
$\pi_{P} \colon S \times P^{\mathrm{JS}}_n(S, \beta) \to P^{\mathrm{JS}}_n(S, \beta)$ is the projection.
The virtual class $[P^{\mathrm{JS}}_{n}(S,\beta)]^{\mathrm{vir}}$ is constructed with respect to the perfect obstruction theory 
$\dR \hH om_{\pi_P}(\mathbb{I}_S,\mathbb{F})$.  
\end{prop}
\begin{proof}
Given $(\oO_X\stackrel{s}{\to}  F)\in P^{\mathrm{JS}}_{n}(X,\beta)$ such that $(n,\omega\cdot\beta)=1$, $F$ is stable and hence scheme theoretically supported on $S$ (ref.~\cite[Prop.~3.1]{CMT1}). 
Similar to \cite[Prop.~4.7]{CMT2}, \cite[Prop.~4.2]{CKM2} we have an isomorphism
$$P^{\mathrm{JS}}_{n}(X,\beta)\cong P^{\mathrm{JS}}_{n}(S,\beta) $$
of moduli spaces, under which virtual classes have the desired property.
\end{proof}

As $H^4(X, \mathbb{Z})\cong \mathbb{Z}$ and the primary insertion (\ref{pri insertion}) is linear with respect to $\gamma\in H^4(X, \mathbb{Z})$, so we may 
simply take $\gamma=[\mathrm{pt}]$ to be the generator $H^4(X, \mathbb{Z})\cong \mathbb{Z}$.
By the same argument as in Theorem \ref{prod of cy3}, we have
\begin{thm}\label{n=1 local surface}
Let $X$ be one of the above two local surfaces (\ref{two local surf}). Assume Conjecture \ref{g=0 one dim sheaf conj} holds for $\beta\in H_2(X)$. 
Then we have
$$ 
P^{\mathrm{JS}}_{1,\beta}([\mathrm{pt}])=n_{0,\beta}([\mathrm{pt}]),
$$
for certain choice of orientation, i.e. Conjecture \ref{JS/GV conj} holds for $\beta\in H_2(X)$ and $n=1$.
\end{thm}
Note also the following case where our conjecture holds by a trivial reason.
\begin{prop}\label{coprime vanishing for local surface}
If $n>1$ and $\omega\cdot \beta$ are coprime, then  
$$
P^{\mathrm{JS}}_{n,\beta}([\mathrm{pt}])=0.
$$
Moreover, Conjecture \ref{JS/GV conj} holds in this setting.
\end{prop}
\begin{proof}
As $(n,\,\omega\cdot \beta)=1$, the coarse moduli space $M_{n}(X,\beta)$ of one dimensional semistable sheaves $F$
on $X$ with $[F]=\beta$ and $\chi(F)=n$ consists of stable sheaves only. By \cite[Prop.~3.1]{CMT1}, $F$ is scheme 
theoretically supported on the zero section $S$ of $X$, and we have an isomorphism 
$$M_{n}(X,\beta)\cong M_{n}(S,\beta), $$
to the (smooth) coarse moduli space $M_{n}(S,\beta)$ of one dimensional stable sheaves on $S$.

We have a surjective (as $n>1$) forgetful map 
$$f\colon P^{\mathrm{JS}}_{n}(S,\beta)\to  M_{n}(S,\beta), \quad (\oO_X\to F)\mapsto F. $$
By Riemann-Roch formula, we know 
$$\mathrm{vir.dim}_{\mathbb{C}}(P_{n}(S,\beta))=n+\beta^2 ,\quad \dim_{\mathbb{C}}(M_{n}(S,\beta))=1+\beta^2. $$ 
By Proposition \ref{local surface identify virtual class}, we have 
\begin{align*}
P^{\mathrm{JS}}_{n,\beta}([\mathrm{pt}])&=\int_{[P^{\mathrm{JS}}_{n}(S,\beta)]^{\mathrm{vir}}} \tau([\mathrm{pt}])\cdot 
e\Big(-\dR\hH om_{\pi_{P}}(\mathbb{F}, \mathbb{F} \boxtimes L_1)\Big) \\
&=\int_{[P^{\mathrm{JS}}_{n}(S,\beta)]^{\mathrm{vir}}} f^*\Big(\tau([\mathrm{pt}])\cdot e\big(-\dR\hH om_{\pi_{M}}(\mathbb{F}, \mathbb{F} \boxtimes L_1)\big)\Big) \\
&=\int_{f_*[P^{\mathrm{JS}}_{n}(S,\beta)]^{\mathrm{vir}}} \Big(\tau([\mathrm{pt}])\cdot e\big(-\dR\hH om_{\pi_{M}}(\mathbb{F}, \mathbb{F} \boxtimes L_1)\big)\Big) \\
&=0 \,.
\end{align*}
Here the second equality is because the insertion comes from the pull-back from $M_{n}(S,\beta)$ via $f$, and in the last equality we use 
$f_*[P^{\mathrm{JS}}_{n}(S,\beta)]^{\mathrm{vir}}=0$ by a dimension counting.

In the coprime case $(n,\, \omega\cdot \beta)=1$, the conjectural formula in Conjecture \ref{JS/GV conj} obviously gives vanishing
$P^{\mathrm{JS}}_{n,\beta}([\mathrm{pt}])=0$, which coincides with the above computations.
\end{proof}

To sum up, we verify Conjecture \ref{JS/GV conj} for low degree curve classes.
In the following, we denote by $[l] \in H_2(X, \mathbb{Z})=H_2(\mathbb{P}^2, \mathbb{Z})$
to be the line class. 
\begin{prop}\label{local P2 low deg class}
Conjecture \ref{JS/GV conj} holds for $X=\mathrm{Tot}_{\mathbb{P}^2}(\oO(-1)\oplus \oO(-2))$ in the following cases
\begin{itemize}
\item $\beta=2[l]$ and $n=0,1,2,2k+1$ \emph{($k\geqslant 1$)}.
\item $\beta=3[l]$ and $n=0,1,3k\pm 1$ \emph{($k\geqslant 1$)}.
\item $\beta=4[l]$ and $n=0,\,2k+1$ \emph{($k\geqslant 1$)}. 
\end{itemize}
\end{prop}
\begin{proof}
When $(n,\omega\cdot\beta)=1$, Conjecture \ref{JS/GV conj} is reduced to Theorem \ref{n=1 local surface}, 
Proposition \ref{coprime vanishing for local surface} and our previous verification of Conjecture \ref{g=0 one dim sheaf conj} (ref.~\cite[Sect.~3.2]{CMT1}).
The $\beta=2[l]$, $n=2$ case follows from a similar argument as Proposition \ref{deg two class}.

When $n=0$, we discuss $\beta=4[l]$ case (other cases follow from an easier argument).
In this case, the condition (\ref{ineq:betan}) is
satisfied since $n([l])=n(2[l])=1$ and $n(3[l])=0$. 
Therefore 
$P_0(X, 4[l])=P_0^{\rm{JS}}(X, 4[l])$, and 
we have the identity $P_{0, 4[l]}=P_{0, 4[l]}^{\rm{JS}}$
for certain choice of orientation and we then use \cite[Cor.~1.6]{CKM2}.
\end{proof}
Similar to Proposition \ref{local P2 low deg class}, we also have:
\begin{prop}
Conjecture \ref{JS/GV conj} holds for $X=\mathrm{Tot}_{\mathbb{P}^1\times \mathbb{P}^1}(\oO(-1,-1)\oplus \oO(-1,-1))$ in the following cases
\begin{itemize}
\item $\beta=(0,d)$ \emph{($d\geqslant 1$)} and $n=0,1,2$.
\item $\beta=(1,d)$ \emph{($d\geqslant 1$)} and $n=0,1$.
\item $\beta=(2,2)$ and $n=0,1$.
\end{itemize}
\end{prop}
\begin{proof}
For $n=1$ case, by Theorem \ref{n=1 local surface}, we are reduced to prove Conjecture \ref{g=0 one dim sheaf conj} in those cases.
When $\beta=(2,2)$, this was done in \cite[Sect.~3.2]{CMT1}.
When $\beta=(0,d)$, any one dimensional stable sheaf $F$ in this class is scheme theoretically supported on one $\mathbb{P}^1$ factor (ref.~\cite[Lem.~2.2]{CMT1}). This is possible only when $d=1$, Conjecture \ref{g=0 one dim sheaf conj} then follows easily.
When $\beta=(1,d)$, this is not discussed in the previous literature. To be self-contained, we include the argument here.  Any Cohen-Macaulay curve $C$ in class $(1,d)$ has $\chi(\oO_C)=1$. 
Hence $M_1(X,\beta)$ is isomorphism to the moduli space 
$$\mathcal{M}=\mathbb{P}\Big(H^0\big(\mathbb{P}^1\times \mathbb{P}^1,\oO(1,d)\big)\Big)\cong\mathbb{P}^{2d+1}$$
of curves in $\mathbb{P}^1\times \mathbb{P}^1$ with curve class $\beta=(1,d)$.
The universal curve  
$$\mathcal{Z}\subset \mathcal{M}\times \mathbb{P}^1\times \mathbb{P}^1$$
is the $(1,1,d)$ divisor. By \cite[Prop.~3.1]{CMT1}, any $F=\oO_C\in M_1(X,\beta)$ is scheme theoretically supported on the zero section 
$\iota:S:=\mathbb{P}^{1}\times \mathbb{P}^{1}\hookrightarrow X$, then 
\begin{align*}\Ext^2_X(F,F)&\cong \Ext^2_S(\oO_C,\oO_C)\oplus \Ext^1_S(\oO_C,\oO_C\otimes \oO(-1,-1)^{\oplus2})\oplus \Ext^0_S(\oO_C,\oO_C\otimes \oO(-2,-2)) \\
&\cong \Ext^1_S(\oO_C,\oO_C\otimes \oO(-1,-1))^{\oplus2}.
\end{align*}
By Section \ref{review DT4}, for certain choice of orientation, we have 
$$[M_1(X,\beta)]^{\mathrm{vir}}=[\mathcal{M}]\cap  e\big(\mathcal{E}xt^1_{\pi_{M}}(\oO_{\mathcal{Z}},\oO_{\mathcal{Z}}\otimes \oO(-1,-1))\big),$$
where $\pi_{M}: \mathcal{M}\times \mathbb{P}^{1}\times \mathbb{P}^{1}\to \mathcal{M}$ is the projection.
Therefore 
\begin{align*}&\quad \,\,\int_{[M_{1}(X,\beta)]^{\rm{vir}}}\tau([\mathrm{pt}]) \\
&=\int_{\mathbb{P}^{2d+1}}e\big(\mathcal{E}xt^1_{\pi_{M}}(\oO_{\mathcal{Z}},\oO_{\mathcal{Z}}\otimes \oO(-1,-1))\big)\cdot \tau([\mathrm{pt}])\\
&=\int_{\mathbb{P}^{2d+1}}e\big(-\dR \hH om_{\pi_{M}}(\oO_{\mathcal{Z}},\oO_{\mathcal{Z}}\otimes \oO(-1,-1))\big)\cdot c_1\big(\oO_{\mathbb{P}^{2d+1}}(1)\big) \\
&=\int_{\mathbb{P}^{2d+1}}e\big(-\dR \hH om_{\pi_{M}}(\oO-\oO(-1,-1,-d), (\oO-\oO(-1,-1,-d))\otimes \oO(-1,-1))\big)\cdot c_1\big(\oO_{\mathbb{P}^{2d+1}}(1)\big) \\
&=\int_{\mathbb{P}^{2d+1}}e\big(\dR \hH om_{\pi_{M}}(\oO,\oO(1,0,d-1))+\dR \hH om_{\pi_{M}}(\oO,\oO(-1,-2,-d-1)) \big)\cdot c_1\big(\oO_{\mathbb{P}^{2d+1}}(1)\big) \\
&=\int_{\mathbb{P}^{2d+1}}\big(1+c_1(\oO_{\mathbb{P}^{2d+1}}(1))\big)^{2d}\cdot c_1\big(\oO_{\mathbb{P}^{2d+1}}(1)\big)=1,
\end{align*}
where in the third equality, 
we use $\mathcal{Z}$ is a $(1,1,d)$ divisor in $\mathbb{P}^{2d+1}\times\mathbb{P}^1\times \mathbb{P}^1$.
This computation matches with $n_{0,(1,d)}=1$ (ref.~\cite[pp.~24]{KP}), i.e. Conjecture \ref{g=0 one dim sheaf conj} also holds in this case.

In all other cases, we can identify $P^{\mathrm{JS}}_n(X,\beta)\cong P_n(X,\beta)$ as in Proposition \ref{local P2 low deg class}, so invariants are the same for certain choice of orientation. Conjecture \ref{JS/GV conj} then follows.
\end{proof}

\section{Equivariant computations on local curves}
Let $C$ be a smooth projective curve and
\begin{align}\label{X local curve}
p \colon
X=\mathrm{Tot}_C(L_1 \oplus L_2 \oplus L_3) \to C
\end{align}
be the total space of split rank three bundle on $C$. Denote the zero section by $\iota:C\to X$.
Assuming that
\begin{align}\label{L123}
L_1 \otimes L_2 \otimes L_3 \cong \omega_C,
\end{align}
then the variety (\ref{X local curve}) is a non-compact Calabi-Yau 4-fold and we set $l_i \cneq \deg L_i$. 

In this section, we consider the case 
that
\begin{align*}
C=\mathbb{P}^1, \quad 
l_1+l_2+l_3=-2,
\end{align*}
where the latter is equivalent to (\ref{L123}). 
Let $T\subset (\mathbb{C}^{\ast})^{4}$ be the three dimensional subtorus (when the genus $g(C)>0$, we can use fiberwise two dimensional CY torus action for (\ref{X local curve}) to define equivariant invariants) which preserves the Calabi-Yau 4-form of $X$.
Let $\bullet=\Spec \mathbb{C}$ with trivial $T$-action, 
$\mathbb{C} \otimes t_i$ be the one dimensional $(\mathbb{C}^{\ast})^{4}$-representation with weight $t_i$ ($i=0,1,2,3$),
and $\lambda_i \in H_{(\mathbb{C}^{\ast})^{4}}^{\ast}(\bullet)$ be its first Chern class.
They are generators of equivariant cohomology rings:
\begin{align*}
H_{(\mathbb{C}^{\ast})^{4}}^{\ast}(\bullet)=\mathbb{Z}[\lambda_0, \lambda_1, \lambda_2,\lambda_3], \quad \
H_{T}^{\ast}(\bullet)=\frac{\mathbb{Z}[\lambda_0,\lambda_1, \lambda_2, \lambda_3]}{(\lambda_0+\lambda_1+\lambda_2+\lambda_3)}.\end{align*}
The Calabi-Yau torus $T$ lifts to an action on the moduli space of $Z_t$-stable pairs on $X$ which preserves Serre duality pairing.
Since the moduli space is non-compact, we define (equivariant) stable pair invariants by a localization formula (as in \cite{CL1, CMT2, CK2}):
\begin{align}\label{localization for local curve}
P^t_{n,d}=[P^t_{n}(X,d\,[\mathbb{P}^{1}])^{T}]^{\rm{vir}}\cdot e( \dR \hH om_{\pi_P}(\mathbb{I}, \mathbb{I})_0^{\rm{mov}})^{1/2}, \end{align}  
The PT and JS stable pair invariants are then the 
special limits
\begin{align*}
P_{n,d}:=P^t_{n,d}|_{t\to \infty}\,, \quad
P^{\mathrm{JS}}_{n,d}:=P^t_{n,d}|_{t=\frac{n}{d}+0}\,.
\end{align*}  
Here $\mathbb{I}=(\oO_{X\times P^t_n(X, d\,[\mathbb{P}^{1}])}\to \mathbb{F})$
 is the universal stable pair and $\pi_P:X\times P^t_n(X, d\,[\mathbb{P}^{1}])\to P^t_n(X, d\,[\mathbb{P}^{1}])$ is the projection. 
Of course, the equality (\ref{localization for local curve})
is not a definition as the virtual class of the fixed locus as well as the square root needs justification. 
We will make it precise in cases studied below. 
The PT moduli space $P_n(X,d\,[\mathbb{P}^{1}])$, i.e. $t\to \infty$ case is studied in \cite[Sect.~5.3]{CMT2}, \cite[Sect.~2.2,~2.3]{CK2}. 
Here we concentrate on the moduli spaces of Joyce-Song stable pairs 
(ref. Definition \ref{defi:PTJSpair}):
\begin{align*}P^{\mathrm{JS}}_n(X,d\,[\mathbb{P}^{1}])=\bigg{\{}&\big(\oO_X\stackrel{s}{\to}  F\big)\,\Big{|}\, F\,\, \mathrm{is}\,\, \mathrm{one}\,\, \mathrm{dim}\,\,\mathrm{compactly}\,\,\mathrm{supported}\,\, \mathrm{semistable},\,\,\mathrm{Im}(s)\neq 0, \\ 
&\mathrm{with}\,\,[p_*F]=d\,[\mathbb{P}^{1}],\,\, \chi(p_*F)=n\,\, \mathrm{and}\,\,\frac{\chi(p_*F')}{\deg(p_*F')}<\frac{n}{d}\,\,\,\mathrm{if}\,\,\mathrm{Im}(s)\subseteq F'\subsetneq F \bigg{\}}. 
\end{align*}

\subsection{When $(l_1,l_2,l_3)$ general and $d=1$}\label{local P1 d=1}

For the $d=1$ case, as in Proposition \ref{irr class}, we have $$P^{\mathrm{JS}}_n(X,[\mathbb{P}^1])=P_n(X,[\mathbb{P}^1]), $$
whose torus fixed loci are described by:
\begin{lem}
Let $\iota:\mathbb{P}^1\to X$ be the zero section. Then
$$P^{\mathrm{JS}}_n(X,[\mathbb{P}^1])^T=\Big\{I=\big(\oO_X\stackrel{s}{\to} \iota_*\oO_{\mathbb{P}^1}(aZ_0+bZ_\infty)\big)\,\Big| \,a,b\geqslant 0\,\, \emph{with}\,\, a+b=n-1 \Big\}, $$
where $s$ is given by the canonical section and $Z_0, Z_{\infty}\in \mathbb{P}^1$ are the torus fixed points.
\end{lem}
\begin{proof}
As $s$ is nonzero,  it is surjective, then the result follows.
\end{proof}
For an equivariant line bundle $F$ on $\mathbb{P}^1$ and $I=\big(\oO_X\stackrel{s}{\to} \iota_*F)$, we have 
$$\chi_X(I,I)_0=\chi_X(\iota_*F,\iota_*F)-\chi_X(\oO_X,\iota_*F)-\chi_X(\iota_*F,\oO_X)\in K_0^T(\bullet). $$
By the adjunction formula (ref.~\cite[Lem.~4.1]{CMT1}), we have 
\begin{align}\label{adjunct}\chi_X( \iota_*F, \iota_*F)=\chi_{\mathbb{P}^1}(F,F)-\chi_{\mathbb{P}^1}(F,F\otimes N_{\mathbb{P}^1/X})
+\chi_{\mathbb{P}^1}(F,F\otimes \wedge^2 N_{\mathbb{P}^1/X})-\chi_{\mathbb{P}^1}(F,F\otimes \wedge^3 N_{\mathbb{P}^1/X}), \end{align}
where $$N_{\mathbb{P}^1/X}=\oO_{\mathbb{P}^1}(l_1Z_{\infty})\otimes t_1\oplus \oO_{\mathbb{P}^1}(l_2Z_{\infty})\otimes t_2\oplus \oO_{\mathbb{P}^1}(l_3Z_{\infty})\otimes t_3. $$
We want to choose a \textit{square root} of $\chi_X(I,I)_0$, i.e. finding $\chi_X(I,I)^{\frac{1}{2}}_0\in K_0^T(\bullet)$ such that
$$\chi_X(I,I)_0=\chi_X(I,I)^{\frac{1}{2}}_0+\overline{\chi_X(I,I)^{\frac{1}{2}}_0}\in K_0^T(\bullet), $$
where $\overline{(\cdot)}$ denotes the involution on $K_0^T(\bullet)$ induced by $\mathbb{Z}$-linearly extending the map 
$$t_0^{w_0}t_1^{w_1}t_2^{w_2}t_3^{w_3} \mapsto t_0^{-w_0}t_1^{-w_1}t_2^{-w_2}t_3^{-w_3}.$$
By Serre duality and (\ref{adjunct}), we can define  
\begin{align}
\begin{split} \label{d=1 square root} 
\chi_X(I,I)^{\frac{1}{2}}_0&:=\chi_X(\iota_*F,\iota_*F)^{\frac{1}{2}}-\chi_X(\oO_X,\iota_*F) \\
&:= \chi_{\mathbb{P}^1}(\oO_{\mathbb{P}^1})-\chi_{\mathbb{P}^1}(N_{\mathbb{P}^1/ X})-\chi_{\mathbb{P}^1}(\oO_{\mathbb{P}^1},F).
\end{split}
\end{align}
The ($d=1$) $T$-equivariant JS stable pair invariant is defined in the following: 
\begin{defi}\label{def:equiv:JS}
Let $\chi_X(I, I)_0^{\frac{1}{2}}$ be chosen as in (\ref{d=1 square root}).
Then we define 
\begin{align*}
P_{n, 1}^{\rm{JS}}:= \sum_{I \in P_n^{\rm{JS}}(X, [\mathbb{P}^1])^{T}}
e_T(\chi_X(I, I)_0^{\frac{1}{2}}) \in 
\frac{\mathbb{Q}(\lambda_0, \lambda_1, \lambda_2, \lambda_3)}{(\lambda_0+\lambda_1+\lambda_2+\lambda_3)}. 
\end{align*} 
\end{defi}
By Proposition~\ref{prop:nowall}, $P_{n, 1}^{\rm{JS}}=P_{n, 1}$ which has been studied in \cite{CK2, CKM1}.

\subsection{When $(l_1,l_2,l_3)$ general and $d=2$}\label{local P1 d=2}
In this case, we have 
\begin{align*}
P_n^{\rm{JS}}(X, 2[\mathbb{P}^1])^T=P_n(X, 2[\mathbb{P}^1])^T,
\end{align*}
for $n\leqslant 2$ by Proposition~\ref{prop:nowall}, 
since $n([\mathbb{P}^1])=1$. 
Based on (\ref{localization for local curve}), one can then define $T$-equivariant JS stable pair invariants $P^{\mathrm{JS}}_{n,2}$ such that
$$P^{\mathrm{JS}}_{n,2}=P_{n,2}\in \frac{\mathbb{Q}(\lambda_0, \lambda_1,\lambda_2,\lambda_3)}{(\lambda_0+\lambda_1+\lambda_2+\lambda_3)}, $$
where $P_{n,2}$ has been rigorously defined in \cite[Def.~2.9]{CK2}.
\begin{rmk}${}$ \\
(1) When $n<0$, $P^{\mathrm{JS}}_{n,2}$ and $P_{n,2}$ are not necessarily zero (e.g. the case $n=-1$, $l_1=3$, $l_2=-2$, $l_3=-3$).  \\
(2) When $n\geqslant 3$, we can still define $P^{\mathrm{JS}}_{n,2}$, $P_{n,2}$.
But they are not necessarily the same (e.g. the case $n=3$, $l_1=2$, $l_2=l_3=-2$). 
It is an interesting question to find a formula relating them.
\end{rmk}
 
\subsection{When $(l_1,l_2,l_3)=(-1,-1,0)$ and $d$ is arbitrary}\label{local(-1,-1,0)}

First of all, to have a nonempty moduli space 
$P^{\mathrm{JS}}_{n}(X,d\,[\mathbb{P}^1])$, $d$ must divide $n$ by the Jordan-H\"{o}lder filtration. 
We first classify $T$-fixed JS stable pairs. 
\begin{lem}\label{T-fixed JS pairs}
Let $k\geqslant 0$, $n=d(k+1)$ and $\{Z_0,Z_{\infty}\}=(\mathbb{P}^1)^T$ be the torus fixed points.
Then 
a $T$-fixed JS stable pair 
$I=(\oO_X\stackrel{s}{\to}  F)\in P^{\mathrm{JS}}_{n}(X,d\,[\mathbb{P}^1])^T$
is precisely of the form
\begin{align}\label{general d fixed pt}
F=
\bigoplus_{i=0}^k
\oO_{\mathbb{P}^1}\big((k-i)Z_{\infty} +iZ_{0}\big)
\Big(\sum_{j=0}^{d_i-1} t_3^{-j} \Big),
\end{align}
for some $d_0,\ldots, d_k\geqslant 0$ with $\sum_{i=0}^k d_i=d$, and $s$ is given by canonical sections.
\end{lem}
\begin{proof}
Let 
$\iota$ be the inclusion
\begin{align*}
\iota=i\times \mathrm{id}_{\mathbb{C}}:\mathbb{P}^1\times \mathbb{C}\hookrightarrow \oO_{\mathbb{P}^1}(-1,-1)\times \mathbb{C}=X,
\end{align*}
 where $i$ is the zero section 
 of the projection 
 $\oO_{\mathbb{P}^1}(-1,-1) \to \mathbb{P}^1$. 
Let us take a $T$-fixed JS pair $I=(\oO_X \to F)$. 
By the Jordan-H\"{o}lder filtration, 
$F$ is written as 
\begin{align*}
F=\bigoplus_{i=1}^l
\iota_{\ast}F_i, \ 
F_i=\oO_{\mathbb{P}^1}(k)\boxtimes \oO_{T_i}, 
\end{align*}
where $T_i$ is a zero dimensional subscheme of $\mathbb{C}$
supported at $0 \in \mathbb{C}$. 
We write the section $s$ as 
\begin{align*}
s=(s_1,\ldots,s_l), \quad 0\neq s_i \colon \oO_X\to \iota_*(\oO_{\mathbb{P}^1}(k)\boxtimes \oO_{T_i}). 
\end{align*}
Here each $s_i$ is non-zero by the JS stability. 
By pushforward to $\mathbb{P}^1$, we know $s_i$ is described by commutative diagrams
\begin{align*}
\xymatrix{
\oO_{\mathbb{P}^1}   \ar[r]^{s_i^j\quad\quad} \ar[d]^{=} & \oO_{\mathbb{P}^1}(k)\otimes t_3^{-j}  \ar[d]^{\,\, \cdot\,t_3^{-1}}\\
\oO_{\mathbb{P}^1}  \ar[r]^{s_i^{j+1}\quad\quad} &  \oO_{\mathbb{P}^1}(k)\otimes t_3^{-j-1}.}
\end{align*}
Then $s_i$ is determined by $s_i^{0}$
which gives $\oO_{\mathbb{P}^1}(k)$
an equivariant structure of the form
\begin{align*}
\oO_{\mathbb{P}^1}(a_iZ_0+(k-a_i)Z_{\infty}), \quad 0\leqslant a_i\leqslant k, 
\end{align*}
and $s_i^0$ is the canonical section.
So each $F_i$ is of the form
\begin{align*}
F_i=
\oO_{\mathbb{P}^1}(a_iZ_0+(k-a_i)Z_{\infty})\Big(\sum_{i=0}^{d_i-1}t_3^{-i}\Big),
\end{align*}
for some $d_i \geqslant 0$. 
Furthermore 
we need $a_i\neq a_j$ for $i\neq j$ in order that $(F,s)$ is JS stable.
Indeed suppose that $a_i=a_j$, 
and set $\overline{F}_i=F_i/t_3^{-1} F_i$. 
Then 
there is an isomorphism 
$h \colon \overline{F}_j \stackrel{\cong}{\to} \overline{F}_i$ such that the composition
\begin{align*}
\oO_X 
\stackrel{s}{\to}F \twoheadrightarrow
 \overline{F}_i \oplus \overline{F}_j \twoheadrightarrow 
\overline{F}_i
\end{align*}
is zero, where the last arrow is $(x, y) \mapsto x-h(y)$. 
The above morphism destabilizes 
$(\oO_X \stackrel{s}{\to} F)$ in JS stability, so a contradiction. 
Therefore $F$ is of the form (\ref{general d fixed pt}). 

Conversely it is straightforward 
to check that any pair $(\oO_X \stackrel{s}{\to}F)$
where $(F, s)$ is as in (\ref{general d fixed pt})
is a $T$-fixed JS stable pair. 
\end{proof}
To choose a square root for $\chi_X(I,I)_0$, we recall the following:
\begin{lem}\label{lem on equiv RR}
As elements in $K_0^T(\bullet)$, we have 
\begin{align}
\begin{split}\label{chi O_Z}
\chi(\oO_{\mathbb{P}^1}(aZ_0+bZ_{\infty}))=\left\{
\begin{array}{rcl}
t_0^{-b}+\cdots+t_0^{-1}+1+t_0+\cdots+t_0^{a}\,,     &      &\mathrm{if} \,\, a,\, b\geqslant 0, \\
& & \\
t_0^a\,,  \quad\quad\quad\quad\quad\quad\quad   &    &\mathrm{if} \,\, a=-b>0. 
\end{array} \right. 
\end{split}
\end{align}
\end{lem}
\begin{proof}
The $T$-equivariant Riemann-Roch formula gives
\begin{align*}
\ch\Big(\chi\big(\oO_{\mathbb{P}^1}(aZ_0+bZ_{\infty})\big)\Big)=\frac{e^{a\lambda_0}}{1-e^{-\lambda_0}}+\frac{e^{-b\lambda_0}}{1-e^{\lambda_0}}\
=\frac{e^{(a+1)\lambda_0}-e^{-b\lambda_0}}{e^{\lambda_0}-1},
\end{align*}
from which we can conclude the result.
\end{proof}
\begin{lem}\label{choice of squ root}
Let $k\geqslant 0$, $n=d(k+1)$ and $I=(\oO_X\stackrel{s}{\to}  F)\in P^{\mathrm{JS}}_{n}(X,d\,[\mathbb{P}^1])^T$, where $F$ is given by (\ref{general d fixed pt}). Then we can choose a square root of $\chi_X(I,I)_0$ to be
\begin{align*}
\chi_X(I,I)^{\frac{1}{2}}_0=-\sum_{i=0}^k\Big(\sum_{j=-(k-i)}^{i}t_0^j \Big)\Big(\sum_{j=0}^{d_i-1}t_3^{-j}\Big) +\sum_{i<j}t_0^{j-i}(1-t_3^{d_i}-t_3^{-d_j}+t_3^{d_i-d_j})+\sum_{i=0}^k(1-t_3^{d_i}).
\end{align*}
\end{lem}
\begin{proof}
As in (\ref{d=1 square root}) we are left to choose a square root for $\chi_X(F,F)$ and then 
\begin{align}\label{general d square root 1}
\chi_X(I,I)^{\frac{1}{2}}_0:=-\chi(F)+\chi_X(F,F)^{\frac{1}{2}}, 
\end{align}
where by (\ref{chi O_Z}) and (\ref{general d fixed pt}), we have 
\begin{align}
\label{general d square root 2}
\chi(F)=\sum_{i=0}^k\Big(\sum_{j=-(k-i)}^{i}t_0^j \Big)\Big(\sum_{j=0}^{d_i-1}t_3^{-j}\Big). \end{align}
By Serre duality, we may define 
\begin{align*}
\chi_X(F,F)^{\frac{1}{2}}:=&\sum_{\begin{subarray}{c}i<j  \\  0\leqslant i,j\leqslant k \end{subarray}}
\chi_X\big(\oO_{\mathbb{P}^1}(iZ_0+(k-i)Z_{\infty},\oO_{\mathbb{P}^1}(jZ_0+(k-j)Z_{\infty})\big)\big)\cdot \sum_{s=0}^{d_i-1}t_3^s\sum_{r=0}^{d_j-1}t_3^{-r}\\ 
&+d\cdot \chi_X\big(\oO_{\mathbb{P}^1},\oO_{\mathbb{P}^1})^{\frac{1}{2}} +\chi_X\big(\oO_{\mathbb{P}^1},\oO_{\mathbb{P}^1})\Big(\sum_{i=0}^k\big(-d_i+\sum_{s=0}^{d_i-1}t_3^s\sum_{r=0}^{d_i-1}t_3^{-r} \big)\Big)^{\frac{1}{2}},
\end{align*}
where we use the fact that $\chi_X\big(\oO_{\mathbb{P}^1},\oO_{\mathbb{P}^1})=2-t_3-t_3^{-1}$ which is invariant under involution $\overline{(\cdot)}$. 

We choose 
$$\chi_X\big(\oO_{\mathbb{P}^1},\oO_{\mathbb{P}^1})^{\frac{1}{2}}:=1-t_3, $$
\begin{align*}
\Big(\sum_{i=0}^k\big(-d_i+\sum_{s=0}^{d_i-1}t_3^s\sum_{r=0}^{d_i-1}t_3^{-r} \big)\Big)^{\frac{1}{2}}&:=
\sum_{i=0}^k \bigg(\frac{(1-t_3^{d_i})(1-t_3^{-d_i})}{2-t_3-t_3^{-1}}-d_i\bigg)^{\frac{1}{2}} \\
&=\sum_{i=0}^k \bigg(\frac{2-t_3^{d_i}-t_3^{-d_i}-d_i(2-t_3-t_3^{-1})}{2-t_3-t_3^{-1}}\bigg)^{\frac{1}{2}} \\
&:=\sum_{i=0}^k \bigg(\frac{1-t_3^{d_i}-d_i(1-t_3)}{2-t_3-t_3^{-1}}\bigg). 
\end{align*}
By (\ref{chi O_Z}) and the adjunction formula, we have 
$$\chi_X\big(\oO_{\mathbb{P}^1}(iZ_0+(k-i)Z_{\infty}),\oO_{\mathbb{P}^1}(jZ_0+(k-j)Z_{\infty})\big)=2\,t_0^{j-i}-t_0^{j-i}\,t_3-t_0^{j-i}\,t_3^{-1}. $$ 
Then it is easy to see
\begin{align}
\label{general d square root 3}
\chi_X(F,F)^{\frac{1}{2}}=\sum_{i<j}t_0^{j-i}(1-t_3^{d_i}-t_3^{-d_j}+t_3^{d_i-d_j})+\sum_{i=0}^k(1-t_3^{d_i}). 
\end{align} 
Combining with (\ref{general d square root 1}), (\ref{general d square root 2}), (\ref{general d square root 3}), we are done.
\end{proof}
\begin{defi}\label{def of JS inv for -1-10}
Let $X=\mathcal{O}_{\mathbb{P}^{1}}(-1,-1,0)$ and $T$ be the Calabi-Yau torus. The $T$-equivariant JS stable pair invariants are defined by
$$P^{\mathrm{JS}}_{n,d}:=\sum_{I\in P^{\mathrm{JS}}_n(X,d\,[\mathbb{P}^1])^T}(-1)^{d+1}e_T(\chi_X(I,I)^{\frac{1}{2}}_0)\in \frac{\mathbb{Q}(\lambda_0, \lambda_1,\lambda_2,\lambda_3)}{(\lambda_0+\lambda_1+\lambda_2+\lambda_3)}, $$
where $\chi_X(I,I)^{\frac{1}{2}}_0$ is chosen as in Lemma \ref{choice of squ root} and the sign denotes a choice of orientation.
\end{defi}
We can explicitly compute all $T$-equivariant JS stable pair invariants for $\mathcal{O}_{\mathbb{P}^{1}}(-1,-1,0)$.
\begin{thm}\label{main thm local curve}
In the setting of Definition \ref{def of JS inv for -1-10}, for $k\geqslant 0$ and $n=d(k+1)$, we have  
\begin{align*}
P^{\mathrm{JS}}_{n,d} =&\frac{(-1)^{k(d+1)}}{1!\,2!\,\cdots k!}\cdot\frac{1}{\lambda_0^{k(k+1)/2}\lambda_3^{d}}\cdot 
\sum_{\begin{subarray}{c}d_0+\cdots+d_k=d  \\  d_0,\ldots, d_k\geqslant 0 \end{subarray}}\frac{1}{d_0!\cdots d_k!}\cdot \prod_{\begin{subarray}{c}i<j  \\  0\leqslant i,j \leqslant k \end{subarray}}\Big((j-i)\lambda_0+(d_i-d_j)\lambda_3\Big) \\
&\times \prod_{i=0}^k\Bigg(\prod_{\begin{subarray}{c}1\leqslant a\leqslant d_i  \\  1\leqslant b\leqslant k-i  \end{subarray}}\frac{1}{a\lambda_3+b\lambda_0}\cdot
\prod_{\begin{subarray}{c}1\leqslant a\leqslant d_i  \\  1\leqslant b\leqslant i  \end{subarray}}\frac{1}{a\lambda_3-b\lambda_0}\Bigg).
\end{align*}
\end{thm}
\begin{proof}
Let $F$ be given by (\ref{general d fixed pt}) and choose the square root $\chi_X(I,I)^{\frac{1}{2}}_0$ as in Lemma \ref{choice of squ root}.
A direct calculation gives the following:
\begin{align*}
e_T(\chi_X(I,I)^{\frac{1}{2}}_0)=&\frac{(-1)^{k(d+1)}}{1!\,2!\,\cdots k!}\cdot\frac{1}{\lambda_0^{k(k+1)/2}\lambda_3^{d}}\cdot 
\frac{1}{d_0!\cdots d_k!}\cdot \prod_{\begin{subarray}{c}i<j  \\  0\leqslant i,j \leqslant k \end{subarray}}\Big((j-i)\lambda_0+(d_i-d_j)\lambda_3\Big) \\
&\times \prod_{i=0}^k\Bigg(\prod_{\begin{subarray}{c}1\leqslant a\leqslant d_i  \\  1\leqslant b\leqslant k-i  \end{subarray}}\frac{1}{a\lambda_3+b\lambda_0}\cdot
\prod_{\begin{subarray}{c}1\leqslant a\leqslant d_i  \\  1\leqslant b\leqslant i  \end{subarray}}\frac{1}{a\lambda_3-b\lambda_0}\Bigg).
\end{align*}
Taking a sum over all torus fixed points (as described by Lemma \ref{T-fixed JS pairs}) gives the result
\end{proof}
For $\mathcal{O}_{\mathbb{P}^{1}}(-1,-1,0)$, PT stable pair invariants satisfy $P_{0,d}=P_{1,d+1}=0$ if $d>0$ (ref.~\cite[Prop.~5.2]{CMT2}). In view of Conjecture \ref{PT/GV conj}, the only non-zero ``GV type invariant" exists in the $g=0$, $d=1$ case (and $n_{0,1}=\lambda_3^{-1}$).
The following direct analogue of Conjecture \ref{main conj} is expected to be true.
\begin{conj}\label{conj for local 0,-1,-1}
Let $X=\mathcal{O}_{\mathbb{P}^{1}}(-1,-1,0)$ and $T$ be the Calabi-Yau torus. Then the $T$-equivariant JS stable pair invariants (Definition \ref{def of JS inv for -1-10}) satisfy
$$ P^{\mathrm{JS}}_{n,d}=\left\{
\begin{array}{rcl}
\frac{1}{d\,! (\lambda_3)^{d}}\,,      &      &\mathrm{if} \,\, n=d\geqslant 0, \\
& & \\
0\,,   \quad \quad   &      &\mathrm{otherwise}.
\end{array} \right. 
$$
\end{conj}
By Theorem \ref{main thm local curve}, verification of Conjecture \ref{conj for local 0,-1,-1} reduces to an explicit combinatoric problem.
\begin{thm}\label{thm:id:rational}
Conjecture \ref{conj for local 0,-1,-1} is true in the following cases 
\begin{itemize}
\item $d\nmid n$, 
\item $d=1,2$ with any $n$,
\item  $n=d,2d$ with any $d$.
\end{itemize}
\end{thm}
\begin{proof}
By the Jordan-H\"{o}lder filtration, $d$ must divide $n$, otherwise the moduli space is empty. 
The $d=1$ case is discussed in Section \ref{local P1 d=2}. The $n=d$ case is easy. Here we show a proof when $n=2d$ (the $d=2$ case can be proved using a similar method). 

In this case, the formula simplifies to 
\begin{align*}
&P^{\mathrm{JS}}_{2d,d}=\\
&\frac{1}{\lambda_3^{d}}\cdot 
\sum_{\begin{subarray}{c}d_0+d_1=d  \\  d_1, d_1\geqslant 0 \end{subarray}}\frac{(-1)^{d_1}}{d_0!\cdot d_1!} 
\cdot\frac{\lambda_0+(d_0-d_1)\lambda_3}{(\lambda_0+d_0\lambda_3)\cdots \cdot(\lambda_0+2\lambda_3)\cdot(\lambda_0+\lambda_3)\cdot\lambda_0\cdot (\lambda_0-\lambda_3)\cdot(\lambda_0-2\lambda_3)\cdots (\lambda_0-d_1\lambda_3)},
\end{align*}
and we want to show it is zero. The rational function   
$$\Phi=\sum_{\begin{subarray}{c}d_0+d_1=d  \\  d_1, d_1\geqslant 0 \end{subarray}}\frac{(-1)^{d_1}}{d_0!\cdot d_1!}\cdot\frac{\lambda_0+(d_0-d_1)\lambda_3}{(\lambda_0+d_0\lambda_3)\cdots \cdot(\lambda_0+2\lambda_3)\cdot(\lambda_0+\lambda_3)\cdot\lambda_0\cdot (\lambda_0-\lambda_3)\cdot(\lambda_0-2\lambda_3)\cdots (\lambda_0-d_1\lambda_3)}$$
is homogenous in variable $\lambda_0$ and $\lambda_3$ and all possible poles are of order one. To prove cancellation of poles,
we may set $\lambda_3=1$ and it is enough to prove the residue
is zero at any pole. 

Poles happen at $\lambda_0=m\in \{0,\pm1,\ldots, \pm d\}$.
Say $m\geqslant0$ ($m\leqslant 0$ case is similar), terms involving $(\lambda_0-m)$ exists only when $d_1=m,m+1,\ldots,d$. We consider the residue at $\lambda_0=m$:
\begin{align*}
\mathrm{Res}_{\lambda_0=m}(\Phi|_{\lambda_3=1})&=\quad \sum_{i=m}^d\frac{(-1)^{i}}{(d-i)!\cdot i!}\cdot\frac{m+d-2i}{(m+d-i)\cdots (m+1)\cdot m\cdot(m-1)\cdots \widehat{(m-m)} \cdots(m-i)} \\
&=\sum_{i=m}^d\frac{(-1)^{m}}{(d-i)!\cdot i!}\cdot\frac{m+d-2i}{(i-m)!\cdot(m+d-i)!}\\
&=(-1)^{m}\cdot \sum_{j=d-m}^{m-d}\frac{j}{\big(\frac{m+d-j}{2}\big)!\cdot\big(\frac{d-m+j}{2}\big)!\cdot\big(\frac{d-m-j}{2}\big)!\cdot\big(\frac{m+d+j}{2}\big)!} \\
&=0\, ,
\end{align*}
where we make a change of index $j=m+d-2i$ in the third equality and the last equality is because
the denominator is invariant under $j\to -j$ while the numerator gets a sign change.
\end{proof}
In general, we may have higher order poles in the rational functions involved in $P^{\mathrm{JS}}_{n,d}$ which we do not know how to deal with 
at the moment. Nevertheless, we verify our conjecture in a huge number of examples by direct calculations with the help of a `Mathematica' program.
\begin{prop}\label{prop:mathematica}
We have $P^{\mathrm{JS}}_{n,d}=0$ in the following cases:
\begin{itemize}
\item  $d=3$ and $n/d\leqslant 30$, 
\item  $d=4$ and $n/d\leqslant 20$, 
\item  $d=5$ and $n/d\leqslant 14$, 
\item  $d=6$ and $n/d\leqslant 11$, 
\item  $n=3 d$ and $d\leqslant 60$,
\item  $n=4d$ and $d\leqslant 20$,
\item  $n=5d,6d,7d$ and $d\leqslant 10$,
\item  $n=8d$ and $d\leqslant 9$,
\item  $n=9d$ and $d\leqslant 8$,
\item  $n=10 d$ and $d\leqslant 7$,
\end{itemize}
i.e. Conjecture \ref{conj for local 0,-1,-1} is true in all these cases.
\end{prop}
\begin{rmk}
 Recently Conjecture \ref{conj for local 0,-1,-1} has been proved in general by studying tautological invariants and their compact analogies \cite{CT2, CT3}.
 \end{rmk}

\providecommand{\bysame}{\leavevmode\hbox to3em{\hrulefill}\thinspace}
\providecommand{\MR}{\relax\ifhmode\unskip\space\fi MR }
\providecommand{\MRhref}[2]{%
  \href{http://www.ams.org/mathscinet-getitem?mr=#1}{#2}
}
\providecommand{\href}[2]{#2}


\begin{thebibliography}{MNOP06}

\bibitem[AB13]{AB}
D.~Arcara and A.~Bertram, \emph{Bridgeland-stable moduli spaces for {K}-trivial
  surfaces. {W}ith an appendix by {M}ax {L}ieblich}, J.~Eur.~Math.~Soc.~
  \textbf{15} (2013), 1--38.

\bibitem[BF97]{BF}
K.~Behrend and B.~Fantechi, \emph{The intrinsic normal cone}, Invent.~Math.~
  \textbf{128} (1997), 45--88.

\bibitem[BJ]{BJ}
D.~Borisov and D.~Joyce, \emph{Virtual fundamental classes for moduli spaces of
  sheaves on {C}alabi-{Y}au four-folds}, Geom. Topol. (21), (2017) 3231--3311.
 
  
\bibitem[Bri]{Bri} T. Bridgeland, 
\emph{Stability conditions on triangulated categories}, Ann. of Math. (2) 166 (2007), no. 2, 317--345. 
 

\bibitem[Cao1]{Caoconic}
Y.~Cao,  \emph{Counting conics on sextic 4-folds}, Math. Res. Lett., Vol. 26, No. 5 (2019), pp. 1343--1357.

\bibitem[Cao2]{CaoFano}
Y.~Cao, \emph{Genus zero {G}opakumar-{V}afa type invariants for {C}alabi-{Y}au
  4-folds {II}: {F}ano 3-folds}, Commun. Contemp. Math. 22 (2020), no. 7, 1950060, 25 pages. 
  

\bibitem[CGJ]{CGJ}
Y.~Cao, J.~Gross, and D.~Joyce, \emph{Orientability of moduli spaces of
  {S}pin(7)-instantons and coherent sheaves on {C}alabi-{Y}au 4-folds}, Adv. Math. \textbf{368}, (2020), 107134.

\bibitem[CK18]{CK1}
Y.~Cao and M.~Kool, \emph{Zero-dimensional {D}onaldson-{T}homas invariants of
  {C}alabi-{Y}au 4-folds}, Adv. Math. \textbf{338} (2018), 601--648.
 

\bibitem[CK19]{CK2}
Y.~Cao and M.~Kool,   \emph{Curve counting and {DT/PT} correspondence for
  {C}alabi-{Y}au 4-folds}, Adv. Math. 375 (2020) 107371.


\bibitem[CKM19]{CKM1}
Y.~Cao, M.~Kool, and S.~Monavari, \emph{K-theoretic {DT/PT} correspondence for
  toric {C}alabi-{Y}au 4-folds}, arXiv:1906.07856.
  
  
\bibitem[CKM20]{CKM2}
Y.~Cao, M.~Kool, and S.~Monavari, \emph{Stable pair invariants of local Calabi-Yau 4-folds},
Int. Math. Res. Not. IMRN 2022, no. 6, 4753--4798.
 

\bibitem[CL14]{CL1}
Y.~Cao and N.~C. Leung, \emph{Donaldson-{T}homas theory for {C}alabi-{Y}au
  4-folds}, arXiv:1407.7659.

\bibitem[CL17]{CL2}
Y.~Cao and N.~C. Leung,  \emph{Orientability for gauge theories on
  {C}alabi-{Y}au manifolds}, Adv. Math. \textbf{314} (2017), 48--70.
  

\bibitem[CMT18]{CMT1}
Y.~Cao, D.~Maulik, and Y.~Toda, \emph{Genus zero {G}opakumar-{V}afa type
  invariants for {C}alabi-{Y}au 4-folds}, Adv. Math. \textbf{338} (2018),
  41--92.  

\bibitem[CMT19]{CMT2}
Y.~Cao, D.~Maulik, and Y.~Toda, \emph{Stable pairs and {G}opakumar-{V}afa type invariants for
  {C}alabi-{Y}au 4-folds}, J. Eur. Math. Soc. (JEMS) 24 (2022), no. 2, 527--581.
  
\bibitem[CT20a]{CT} 
Y.~Cao and Y.~Toda, \textit{Gopakumar-Vafa type invariants on Calabi-Yau 4-folds via descendent insertions}, Comm. Math. Phys. 383 (2021), no. 1, 281--310.

\bibitem[CT20b]{CT2} 
Y.~Cao and Y.~Toda, \textit{Tautological stable pair invariants of Calabi-Yau 4-folds},  Adv. Math. 396 (2022) 108176.

\bibitem[CT20c]{CT3} 
Y.~Cao and Y.~Toda, \textit{Counting perverse coherent systems on Calabi-Yau 4-folds}, arXiv:2009.10909. To appear in Math. Ann.

\bibitem[GV]{GV}
R.~Gopakumar and C.~Vafa, \emph{M-theory and topological strings {II}},
  hep-th/9812127.

\bibitem[HRS96]{HRS}
D.~Happel, I.~Reiten, and S.~O. Smal$\o$, \emph{Tilting in abelian categories
  and quasitilted algebras}, Mem.~Amer.~Math.~Soc, vol. 120, 1996.

\bibitem[JS12]{JS}
D.~Joyce and Y.~Song, \emph{A theory of generalized {D}onaldson-{T}homas
  invariants}, Mem.~Amer.~Math.~Soc.~ \textbf{217} (2012).

\bibitem[Kat08]{Katz}
S.~Katz, \emph{Genus zero {G}opakumar-{V}afa invariants of contractible
  curves}, J.~Differential.~Geom.~ \textbf{79} (2008), 185--195.

\bibitem[KP]{KP}
A.~Klemm and R.~Pandharipande, \emph{Enumerative geometry of {C}alabi-{Y}au
  4-folds}, Comm. Math. Phys. 281 (2008), no. 3, 621--653.

\bibitem[KS]{KS}
M.~Kontsevich and Y.~Soibelman, \emph{Stability structures, motivic
  {D}onaldson-{T}homas invariants and cluster transformations}, preprint,
  arXiv:0811.2435.

\bibitem[LT98]{LT}
J.~Li and G.~Tian, \emph{Virtual moduli cycles and {G}romov-{W}itten invariants
  of algebraic varieties}, J.~Amer.~Math.~Soc.~ \textbf{11} (1998), 119--174.
  
\bibitem[Lie06]{Lieblich}
M.~Lieblich, \emph{Moduli of complexes on a proper morphism}, J. Algebraic
  Geom. \textbf{15} (2006), no.~1, 175--206.  

\bibitem[Man12]{Mano}
C.~Manolache, \emph{Virtual push-forwards}, Geom. Topol. \textbf{16} (2012),
  no.~4, 2003--2036. 

\bibitem[MNOP06]{MNOP}
D.~Maulik, N.~Nekrasov, A.~Okounkov, and R.~Pandharipande,
  \emph{Gromov-{W}itten theory and {D}onaldson-{T}homas theory. {I}},
  Compositio.~Math \textbf{142} (2006), 1263--1285.

\bibitem[MT18]{MT}
D.~Maulik and Y.~Toda, \emph{Gopakumar-{V}afa invariants via vanishing cycles},
  Invent. Math. \textbf{213} (2018), no.~3, 1017--1097.  
  
\bibitem[Moc09]{Moc}
T.~Mochizuki, \emph{Donaldson type invariants for algebraic surfaces}, Lecture
  Notes in Mathematics, vol. 1972, Springer-Verlag, Berlin, 2009.

\bibitem[PP17]{PP}
R.~Pandharipande and A.~Pixton, \emph{Gromov-{W}itten/{P}airs correspondence
  for the quintic 3-fold}, J. Amer. Math. Soc. \textbf{30} (2017), no.~2,
  389--449. 

\bibitem[PT09]{PT}
R.~Pandharipande and R.~P. Thomas, \emph{Curve counting via stable pairs in the
  derived category}, Invent.~Math.~ \textbf{178} (2009), 407--447.

\bibitem[PTVV13]{PTVV}
T.~Pantev, B.~To$\ddot{\textrm{e}}$n, M.~Vaquie, and G.~Vezzosi, \emph{Shifted
  symplectic structures}, Publ.~Math.~IHES \textbf{117} (2013), 271--328.
  
\bibitem[Pot93]{LePotier}
J.~Le Potier, \emph{Syst\`emes coh\'{e}rents et structures de niveau},
  Ast\'{e}risque (1993), no.~214, 143. 
  
\bibitem[ST01]{ST}
P. Seidel and R. P. Thomas, \emph{Braid group actions on derived categories of
coherent sheaves}, Duke Math. J. 108 (2001), 37--107.


\bibitem[Tod08]{Tst3}
Y.~Toda, \emph{Moduli stacks and invariants of semistable objects on {K}3
  surfaces}, Adv. in Math. \textbf{217} (2008), 2736--2781.

\bibitem[Tod09]{Toda1}
Y.~Toda, \emph{Limit stable objects on {C}alabi-{Y}au 3-folds}, Duke Math. J.
  \textbf{149} (2009), no.~1, 157--208.  

\bibitem[Tod10a]{Toda2}
Y.~Toda, \emph{Curve counting theories via stable objects {I}. {DT}/{PT}
  correspondence}, J. Amer. Math. Soc. \textbf{23} (2010), no.~4, 1119--1157.

\bibitem[Tod10b]{Toda3}
Y.~Toda, \emph{Generating functions of stable pair invariants via
  wall-crossings in derived categories}, New developments in algebraic
  geometry, integrable systems and mirror symmetry ({RIMS}, {K}yoto, 2008),
  Adv. Stud. Pure Math., vol.~59, Math. Soc. Japan, Tokyo, 2010, pp.~389--434.


\bibitem[Tod10c]{Toda4}
Y.~Toda, \emph{On a computation of rank two {D}onaldson-{T}homas invariants},
  Commun. Number Theory Phys. \textbf{4} (2010), no.~1, 49--102.  

\bibitem[Tod12]{Toda5}
Y.~Toda, \emph{Stability conditions and curve counting invariants on
  {C}alabi-{Y}au 3-folds}, Kyoto J. Math. \textbf{52} (2012), no.~1, 1--50.


\bibitem[Tod17]{Todstack}
Y.~Toda, \emph{Moduli stacks of semistable sheaves and representations of
  {E}xt-quivers}, Geom. Topol. 22 (2018), no. 5, 3083-3144.

\bibitem[Tod18]{Toddbir}
Y.~Toda,  \emph{Birational geometry for d-critical loci and wall-crossing in
  {C}alabi-{Y}au 3-folds}, arXiv:1805.00182.


\end{thebibliography}

\end{document}